\documentclass[12pt]{amsart}
\input epsf

\usepackage{amsfonts,amsthm,amsmath,amssymb,latexsym}
\usepackage{graphicx,color}
\usepackage[all]{xy}
\usepackage{labelfig}
\usepackage{epsfig}
\usepackage[margin=3cm]{geometry}
\usepackage{lipsum}

\begin{document}

\author{Dragomir \v Sari\' c}
\thanks{}

\address{Department of Mathematics, Queens College of CUNY,
65-30 Kissena Blvd., Flushing, NY 11367}
\email{Dragomir.Saric@qc.cuny.edu}

\address{Mathematics PhD. Program, The CUNY Graduate Center, 365 Fifth Avenue, New York, NY 10016-4309}

\theoremstyle{definition}

 \newtheorem{definition}{Definition}[section]
 \newtheorem{remark}[definition]{Remark}
 \newtheorem{example}[definition]{Example}

\newtheorem*{notation}{Notation}

\theoremstyle{plain}

 \newtheorem{proposition}[definition]{Proposition}
 \newtheorem{theorem}[definition]{Theorem}
 \newtheorem{corollary}[definition]{Corollary}
 \newtheorem{lemma}[definition]{Lemma}

\def\H{{\mathbb H}}
\def\F{{\mathcal F}}
\def\R{{\mathbb R}}
\def\Q{{\mathbb Q}}
\def\Z{{\mathbb Z}}
\def\E{{\mathcal E}}
\def\N{{\mathbb N}}
\def\X{{\mathcal X}}
\def\Y{{\mathcal Y}}
\def\C{{\mathbb C}}
\def\D{{\mathbb D}}
\def\G{{\mathcal G}}

\title[train tracks]{Train tracks and measured laminations on infinite surfaces}


\date{\today}

\maketitle

\begin{abstract} Let $X$ be an infinite Riemann surface equipped with its conformal hyperbolic metric such that the action of the fundamental group $\pi_1(X)$ on the universal covering $\tilde{X}$ is of the first kind. We first prove that any geodesic lamination on $X$ is nowhere dense. Given a fixed geodesic pants decomposition of $X$ we define a family of train tracks on $X$ such that any geodesic lamination on $X$ is weakly carried by at least one train track. 
The set of measured laminations on $X$ carried by a train track is in a one to one correspondence with the set of edge weight systems on the  train track. Furthermore, the above correspondence is a homeomorphism when we equipped the measured laminations (weakly carried by a train track) with the weak* topology and the edge weight systems with the topology of pointwise (weak) convergence.

The space $ML_b(X)$ of bounded measured laminations appears prominently when studying the Teichm\"uller space $T(X)$ of $X$. If $X$ has a bounded pants decomposition, a measured lamination on $X$ weakly carried by a train track is bounded if and only if the corresponding edge weight system has a finite supremum norm. The space $ML_b(X)$ is equipped with the uniform weak* topology. The correspondence between bounded measured laminations weakly carried by a train track and their edge weight systems is a homeomorphism for the uniform weak* topology on $ML_b(X)$ and the topology induced by supremum norm on the edge weight system.
\end{abstract}

\section{Introduction}

A Riemann surface $X$ is said to be {\it infinite} if the fundamental group $\pi_1(X)$ of $X$ is infinitely generated. An infinite Riemann surface $X$ has a unique conformal hyperbolic metric and all geometric notions in the paper are with respect to this metric. Let $\tilde{X}$ be the universal covering equipped with the hyperbolic metric such that the covering map $\tilde{X}\to X$ is a local isometry. The universal covering $\tilde{X}$ is isometric to the hyperbolic plane $\mathbb{H}$ and the ideal boundary $\partial_{\infty}\tilde{X}$ is homeomorphic to the unit circle $S^1$. The fundamental group $\pi_1(X)$ is identified with a subgroup of isometries of $\tilde{X}$ such that $\tilde{X}/\pi_1(X)=X$. We introduce a family of train tracks on $X$ starting from a fixed geodesic pants decomposition $\{ P_n\}$ in order to give local coordinates to the space $ML(X)$ of measured laminations on $X$ in terms of their edge weight systems on the train tracks. We prove that basic properties of these local coordinates are the same as for closed surfaces (see Thurston \cite{Thurston}, Bonahon \cite{Bonahon}, Penner-Harer \cite{PennerHarer}). In addition, we also give local coordinates to a subspace $ML_b(X)$ of bounded measured laminations which naturally relates to the quasiconformal deformations of the Riemann surface $X$ and its Teichm\"uller space.

Recently, there is a considerable interest in studying infinite hyperbolic surfaces, their Teichm\"uller spaces and quasiconformal Teichm\"uller mapping class groups and the big mapping class groups (for example, see \cite{ALPS}, \cite{Matsuzaki}, \cite{FujikawaMatsuzaki}, \cite{Vlamis}, \cite{Aramoyana}, \cite{Bavard}, \cite{Patel}, \cite{Kin}, \cite{Durham}, etc). 
The topology of an infinite surface is determined by its genus and the space of ends (see B. K\' er\' ekj\' arto \cite{Kerekjarto} and I. Richards \cite{Richards}). The geometric structure of an infinite Riemann surface equipped with a conformal hyperbolic metric is described by Alvarez-Rodriguez \cite{AR} (see also \cite{BasmajianSaric} and Section 2).

An arbitrary infinite hyperbolic surface $X$ can be obtained by isometrically gluing countably many geodesic pairs of pants along their cuffs and by adding an at most countably many hyperbolic funnels and an at most countably many closed half-planes (see \cite{AR}, \cite{BasmajianSaric} and Section 2).  
A {\it geodesic lamination} $\lambda$ on $X$ is a closed subset of $X$ equipped with a foliation by complete geodesics. Unlike for closed surfaces, $\lambda$ could possibly foliate a subset with non-empty interior. This happens when $X$ has an end that is a hyperbolic funnel or a half-plane. We first prove that the existence  of such ends is the only reason for this phenomenon (see Proposition \ref{prop:support-gl}). 

Indeed, let $\Lambda (\pi_1(X))$ be the limit set of $\pi_1(X)$ on the boundary $\partial_{\infty}\tilde{X}$ and let $\mathcal{C}(\Lambda (\pi_1(X)))$ be the convex core of $\Lambda (\pi_1(X))$ in $\tilde{X}$. Then $\mathcal{C}(X):=\mathcal{C}(\Lambda (\pi_1(X)))/\pi_1(X)$ is the convex core of $X$. Note that $\pi_1(X)$ is of the first kind if and only if $X=\mathcal{C}(X)$. We prove

\begin{theorem}
Let $X$ be an infinite Riemann surface equipped with its conformal hyperbolic metric and let $\lambda$ be a geodesic lamination contained in the convex core $\mathcal{C}(X)$ of $X$. Then $\lambda$ is nowhere dense in $X$.

In particular, if $X=\mathcal{C}(X)$ then any geodesic lamination $\lambda$ in $X$ is nowhere dense.
\end{theorem}

The set of (unoriented) geodesics $G(\tilde{X})$ of $\tilde{X}$ is identified with $(\partial_{\infty}\tilde{X}\times \partial_{\infty}\tilde{X}-\Delta )/(\mathbb{Z}/2\mathbb{Z})$, where $\Delta$ is the diagonal (see Bonahon \cite{Bonahon}). The topology on $G(\tilde{X})$ is given by the product topology. A geodesic lamination $\lambda$ on $X$ lifts to a $\pi_1(X)$-invariant geodesic lamination $\tilde{\lambda}$ of $\tilde{X}$. Conversely, a $\pi_1(X)$-invariant geodesic lamination of $\tilde{X}$ projects to a geodesic lamination of $X$ (see \cite{Bonahon}).   

A homeomorphism $h:S_1\to S_2$ of two compact surfaces $S_1$ and $S_2$ with genus $g>1$
induces a natural homeomorphism $\widetilde{h}:G(\tilde{S}_1)\to G(\tilde{S}_2)$ of the spaces of geodesics of their universal coverings that is $\pi_1(S_1)$- and $\pi_1(S_2)$-equivariant (for example, see \cite{Bonahon}). For infinite hyperbolic surfaces a homeomorphism between two surfaces does not necessarily induce a (natural) homeomorphism between the spaces of geodesics of their universal covers. In fact, one surface can have a funnel end or a closed half-plane end while the other surface might not. Then  the two spaces of geodesics of the universal coverings are not naturally homeomorphic. However, we show that this is the only reason why a homeomorphism of two surfaces might not induce  a homeomorphism of spaces of geodesics of their universal coverings (see Theorem \ref{thm:home_geodesics}). 

\begin{theorem}
Let $X_1$ and $X_2$ be two infinite Riemann surfaces equipped with their conformal hyperbolic metrics such that $X_1=\mathcal{C}(X_1)$ and $X_2=\mathcal{C}(X_2)$. Let $G(\tilde{X}_i)$ be the space of geodesics of the universal covering $\tilde{X}_i$ of $X_i$, for $i=1,2$.  A homeomorphism
$$
h:X_1\to X_2
$$
induces a $\pi_1(X_1)$- and $\pi_1(X_2)$-equivariant homeomorphism
$$
\widetilde{h}:G(\tilde{X}_1)\to G(\tilde{X}_2).
$$
\end{theorem}

In the case of a closed surface $S$ of genus greater than $1$, train tracks on $S$ were used to give local coordinates to the space of measured (geodesic) laminations $ML(S)$ (see \cite{Thurston}, \cite{PennerHarer}, \cite{Bonahon}). Our main result is an extension of the idea of a train track to infinite surfaces in order to better understand the space of measured laminations on these surfaces.

From now on we assume that $X$ is an infinite hyperbolic surface with $X=\mathcal{C}(X)$ which is equivalent to $\pi_1(X)$ being of the first kind. In this case $X$ does not contain funnels or closed half-planes. (For completeness, we note here that a surface with funnels is homeomorphic to a surface where each funnel is replaced by a cusp. Each half-plane can be ``eliminated'' from a hyperbolic surface by choosing appropriate twists on the cuffs of the fixed pants decomposition, see \cite{BasmajianSaric}. Thus any hyperbolic surface $X$ is homeomorphic to a surface $X'$ without funnels and half-plane ends, i.e. $X'=\mathcal{C}(X')$.) 

 Let $\{ P_n\}$ be a fixed geodesic pants decomposition of $X$.
We choose a standard Dehn-Thurston train track in each $P_n$ which meets cuffs at fixed basepoints. The complementary regions of the standard train tracks in the pairs of pants are triangles and punctured monogons. A {\it pants train track} $\Theta$ on $X$ is obtained by taking choices of the standard train tracks in each pair of pants with cuffs being additional edges of the train track $\Theta$. Different choices of standard train tracks in $P_n$ and different choices of smoothing at the basepoints on the cuffs give rise to a whole family of train tracks starting from a fixed pants decomposition $\{ P_n\}$. A bi-infinite edge path $\gamma$ in a pants train track $\Theta$ determines a unique simple geodesic $g(\gamma )$ of $X$.  We will say that $g(\gamma )$ is {\it weakly carried} by $\Theta$ (see Section 4). Let $\tilde{\Theta}$ be the lift of $\Theta$ to the universal covering $\tilde{X}$.

Given an edge $e$ of $\tilde{\Theta}$, denote by $G(e)$ the set of geodesics in $\tilde{X}$ whose corresponding bi-infinite edge paths contain the edge $e$. Let $\mu$ be a measured lamination on $X$ and $\tilde{\mu}$ its lift to $\tilde{X}$. Let $E(\Theta )$ be the set of edges of $\Theta$ and $E(\tilde{\Theta})$ the set of edges of $\tilde{\Theta}$. We define an edge weight system 
$$
f_{\tilde{\mu}}:E(\tilde{\Theta})\to\mathbb{R}
$$
by
$$
f_{\tilde{\mu}}(e):=\tilde{\mu}(G(e)).
$$
The set $G(e)$ is a pre-compact subset of the space $G(\tilde{X})$ of geodesics of $\tilde{X}$ and thus $\tilde{\mu}(G(e))<\infty$ (see Section 5). 
At each vertex of $\tilde{\Theta}$ the edge weight system $f_{\tilde{\mu}}$ satisfies the switch relation as for closed surfaces (see \cite{Bonahon}, \cite{PennerHarer} and Section 5). The edge weight system $f_{\tilde{\mu}}$ is $\pi_1(X)$-invariant and it projects to an edge weight system $f_{{\mu}}:E(\Theta )\to\mathbb{R}$. 

We prove that (see Theorems 5.3 and 6.5).

\begin{theorem}
Let $X$ be an infinite hyperbolic surface such that $X=\mathcal{C}(X)$. Let $\Theta$ be a pants train track. Then the space $ML(X,\Theta )$ of all measured laminations (equipped with the weak* topology) that are weakly carried by $\Theta$ is homeomorphic to the space $\mathcal{W}({\Theta} ,[0,\infty ))$ of all edge weight systems on $\Theta$, when $\mathcal{W}({\Theta} ,[0,\infty ))$ is given the topology of pointwise convergence.
\end{theorem}

The train tracks are used to give local coordinates for measured laminations on infinite surfaces. We also establish that each measured lamination is weakly carried by at least one pants train track constructed from a fixed geodesic pants decomposition of $X$ (see Proposition \ref{prop:carrying}).

Thurston \cite{Thurston1} proved that any homeomorphic deformation of a hyperbolic surface $X$ can be obtained by an earthquake map along a measured lamination.  
The Teichm\"uller space $T(X)$ of a Riemann surface $X$ is the space of quasiconformal deformations of $X$ modulo post-compositions by conformal maps and homotopy relative ideal endpoints. It is known that an earthquake map induces a quasiconformal deformation of a hyperbolic surface $X$ if and only if it is performed along a bounded measured lamination on $X$ (see \cite{Sa1}, \cite{Sa2}, \cite{GHL}, \cite{Thurston1}). Moreover, the bijective correspondence between the space $ML_b(X)$ of bounded measured laminations and the Teichm\"uller space $T(X)$ given by earthquake maps is a homeomorphism for the uniform weak* topology on $ML_b(X)$ and the topology introduced by the Teichm\"uller distance on $T(X)$ (see \cite{MS}). For the definition of the uniform weak* topology see \cite{MS}, \cite{BonahonSaric}, \cite{Sa} and Section 5. In view of the action of the mapping class group it is perhaps even more important  to mention that the Thurston boundary of $T(X)$ is identified with the space $PML_b(X)$ of projective bounded measured laminations on $X$ (see \cite{BonahonSaric} and \cite{Sa}). 

It is therefore of interest to understand the space $ML_b(X)$ of bounded measured lamination on the hyperbolic surface $X$. We restrict our attention to surfaces $X$ with  bounded pants decompositions (for the definition see \cite{Shi}, \cite{ALPS}, \cite{Sa4} and Section 7) where the question  is more tractable. We prove (see Theorems 7.4, 8.6 and 9.4)

\begin{theorem}
Let $X$ be an infinite Riemann surface with a bounded pants decomposition such that $\mathcal{C}(X)=X$ and $\Theta$ a pants train track constructed from the pants decomposition. The space of all bounded edge weight systems $\mathcal{W}_b(\Theta ,[0,\infty ))$ is parametrizing the space $ML_b(X,\Theta )$ of all bounded measured laminations on $X$ that are weakly carried by $\Theta$, where $f\in\mathcal{W}_b(\Theta ,[0,\infty ))$ if $f\in\mathcal{W}(\Theta ,[0,\infty ))$ and $\|f\|_{\infty}<\infty$. 

In addition, the bijective correspondence
$$
ML_b(X,\Theta )\to \mathcal{W }_b(\Theta ,[0,\infty ))
$$
is a homeomorphism when $ML_b(X,\Theta )$ is endowed with the uniform weak* topology and $\mathcal{W}_b(\Theta ,[0,\infty ))$ with the topology induced by the supremum norm.
\end{theorem}

\vskip .2 cm

\noindent {\it Acknowledgements.} I am greatly indebted to Francis Bonahon for our illuminating conversations. I am also grateful to an anonymous referee for various suggestions and questions, and for giving us a short proof of Lemma \ref{lem:dist_bdd} which significantly improved the paper.

\section{Pants decompositions of infinite conformally hyperbolic surfaces}

A {\it topological pair of pants} is a bordered surface homeomorphic to a sphere 
minus three open disks. A {\it geodesic pair of pants} is a topological pair of pants equipped with a metric of constant curvature $-1$ such that the boundary consist of 3 closed geodesics (called {\it cuffs}) with possibly
1 or 2 geodesics degenerated to have length $0$--i.e., a cusp.

A topological surface is said to be {\it infinite } if its fundamental group is infinitely generated.
A Riemann surface is {\it conformally hyperbolic } if it supports a unique metric of constant curvature $-1$ in its conformal class called the {\it hyperbolic metric}. By the Uniformization Theorem every infinite Riemann surface is conformally hyperbolic. 

Given an infinite Riemann surface $X$ we endow it with the unique hyperbolic metric in its conformal class which makes $X$ a complete Riemannian two manifold without boundary. The universal covering $\tilde{X}$ of $X$ is conformally equivalent to the unit disk and the hyperbolic metric on $X$ induces a hyperbolic metric on $\tilde{X}$; in this metric $\tilde{X}$ is isometric to the hyperbolic plane. The boundary at infinity $\partial_{\infty}\tilde{X}$ of the universal covering $\tilde{X}$ is identified with the unit circle. The fundamental group $\pi_1(X)$ acts by isometries on $\tilde{X}$ and $\tilde{X}/\pi_1(X)$ is isometrically identified with $X$. Denote by $\Lambda (\pi_1(X))$ the limit set on $\partial_{\infty}\tilde{X}$ of the action of $\pi_1(X)$. The convex core $\mathcal{C}(\Lambda (\pi_1(X)))$ of $\Lambda (\pi_1(X))$ is the smallest convex subset of $\tilde{X}$ that has $\Lambda (\pi_1(X))$ as its ideal boundary. The convex core $\mathcal{C}(X)$ of $X$ is the smallest convex subset of $X$ that has the same homotopy type--equivalently, $\mathcal{C}(X)=\mathcal{C}(\Lambda (\pi_1(X)))/\pi_1(X)$ (see Maskit \cite{Maskit}).

Alvarez and Rodriguez \cite{AR} (see also \cite{BasmajianSaric}) proved that each infinite conformally hyperbolic surface $X$ can be constructed by isometrically gluing a countable set $\{ P_n\}_n$ of geodesic pairs of pants along their cuffs (boundary geodesics) $\{\alpha_k\}_k$ and by attaching to this union an at most countable set of hyperbolic funnels $\{ F_i\}_i$ and an at most countable set of closed hyperbolic half-planes $\{H_j\}_j$. The hyperbolic funnels $\{ F_i\}$ are attached to the boundary geodesics of $\{P_n\}$  that end up not being glued to any other boundary geodesics. The half-planes are attached to the open geodesics that are accumulated by the cuffs $\{\alpha_k\}_k$. Note that the open geodesics are not in the union of the pairs of pants but rather in the closure of the union.  

A first example of completing a countable union of geodesic pairs of pants by attaching a  hyperbolic half-plane was given by Basmajian \cite{Bas}. To better understand the situation, we fix a point $x\in\cup_nP_n$ and consider the set of all geodesic rays in $(\cup_nP_n)\cup (\cup_i F_i)$ starting at $x$. This set of geodesic rays is naturally identified with the unit circle since each tangent vector at $x$ defines a unique geodesic ray tangent to it. If a geodesic ray $g_x$ starting at $x$ has a finite length then there is an open interval $I_j$ of geodesic rays containing $g_x$ that have finite length (see \cite{BasmajianSaric}). The geodesic rays corresponding to the endpoints of the open interval $I_j$ of finite length geodesic rays have infinite length. To each such open interval $I_j$ a closed hyperbolic half-plane $H_j$ is added and the two boundary geodesic rays of $I_j$ are asymptotic to the boundary geodesic of $H_j$ (see \cite{BasmajianSaric}).  Each closed half-plane $H_j$ in $X$ lifts to countably many closed half-planes in the universal covering $\tilde{X}$ and the ideal boundary of each half-plane in $\tilde{X}$ is an open interval on $\partial_{\infty}\tilde{X}$. 
Since $\partial_{\infty}\tilde{X}$ is homeomorphic to  the unit circle it follows that $\partial_{\infty}\tilde{X}$ can contain at most countably many disjoint open subintervals. We conclude that an at most countably many closed hyperbolic half-planes are added to $X$.

A {\it locally finite topological pants decomposition} of an infinite Riemann surface $X$ is a decomposition of $X$ into topological pairs of pants such that any two pairs of pants are either disjoint or meet along a common  boundary component and each compact subset of $X$ meets at most finitely many pairs of pants. A theorem of K\' er\' ekj\' arto \cite{Kerekjarto}  and Richards \cite{Richards} implies that any infinite surface has a locally finite topological pants decomposition. By replacing each boundary curve of a locally finite topological pants decomposition with the simple closed geodesic in its homotopy class we obtain a locally finite geodesic pants decomposition of the convex core $\mathcal{C}(X)$ minus the set of open geodesics on its boundary (see \cite{BasmajianSaric}). The hyperbolic surface $X$ is obtained  from its convex core $\mathcal{C}(X)$ by attaching hyperbolic funnels to the closed geodesics of the boundary and by attaching half-planes to the infinite (open) geodesics of the boundary of its convex core (see \cite{BasmajianSaric}).

A conformal structure of an infinite Riemann surface $X$ can be changed by twisting along the boundary geodesics $\{\alpha_k\}_k$ of a locally finite geodesic pants decomposition $\{P_n\}_n$ (for example, see Alessandrini-Liu-Papadopoulos-Su \cite{ALPS}). In fact, there is a choice of real twists along  $\{\alpha_k\}_k$ such that the new conformally hyperbolic Riemann surface $X'$ is equal to its convex core union of an at most countably many hyperbolic funnels and no half-planes (see \cite{BasmajianSaric}).  Since $X'$ is obtained from $X$ by twisting along $\{\alpha_k\}$ it follows that $X'$ is homeomorphic to $X$. However, $X'$ is not quasiconformal to $X$. Indeed, a quasiconformal mapping between $X$ and $X'$ lifts to a quasiconformal mapping of their universal covers and it extends to a homeomorphism of the limit sets of the covering groups. The limit set of $\pi_1(X')$ is homeomorphic to the unit circle while the limit set of $\pi_1(X)$ is not connected which is a contradiction.

For the most part we will be interested in infinite Riemann surfaces whose hyperbolic metrics are such that they are equal to their convex cores. In this case we do not have  funnels and half-planes, and any locally finite topological pants decomposition straightens to a locally finite geodesic pants decomposition of the whole surface. The geodesic pairs of pants can have at most two cusps since the cusps are not glued to other pairs of pants. In the terminology of \cite{BasmajianSaric} such conformally hyperbolic Riemann surfaces are said to have no visible ends.  The discussion above shows that any infinite Riemann surface is homeomorphic to a Riemann surface with no visible ends by performing twists along a geodesic pants decomposition and by replacing funnels with cusps (see \cite{BasmajianSaric}).

\section{Geodesic laminations on infinite Riemann surfaces}

Let $X$ be a conformally hyperbolic Riemann surface equipped with its unique hyperbolic metric. Let $\tilde{X}$ be the universal covering of $X$ with the hyperbolic metric induced by the hyperbolic metric on $X$. Each oriented geodesic $g$ of $\tilde{X}$ is uniquely determined by the pair of its endpoints $(a,b)\in\partial_{\infty}\tilde{X}\times \partial_{\infty}\tilde{X}-\Delta$, where $\Delta$ is the diagonal of $\partial_{\infty}\tilde{X}\times \partial_{\infty}\tilde{X}$, $a$ is the initial point of $g$ and $b$ is the end point of $g$. The space of oriented geodesics of $\tilde{X}$ is identified with $\partial_{\infty}\tilde{X}\times \partial_{\infty}\tilde{X}-\Delta$ and the topology is given by the product topology.

The {\it space of (unoriented) geodesics} $G(\tilde{X})$ on the universal covering $\tilde{X}$ is identified with $(\partial_{\infty}\tilde{X}\times \partial_{\infty}\tilde{X}-\Delta)/(\mathbb{Z}/2\mathbb{Z})$, where the action of $\mathbb{Z}/2\mathbb{Z}$ permutes the components. The topology (and, in particular, the convergence) on $G(\tilde{X})$ is the quotient topology of the product topology on $\partial_{\infty}\tilde{X}\times \partial_{\infty}\tilde{X}-\Delta$. Note that $G(\tilde{X})$ is not compact; a sequence of  pairs of points that approaches the diagonal $\Delta$ has no accumulation points (see Bonahon \cite{Bonahon} and \cite{BonahonSaric}). 

\begin{remark} \label{rem:local}
All constructions in the paper are for the unoriented geodesics. We will use subsets of $\partial_{\infty}\tilde{X}\times \partial_{\infty}\tilde{X}-\Delta$ as local coordinates for the space of unoriented geodesics.
\end{remark}

We define a geodesic lamination on a conformally hyperbolic surface $X$.

\begin{definition}
\label{def:g_lamination}
Let $X$ be a conformally hyperbolic surface. A {\it geodesic lamination} $\lambda$ on $X$ consists of a  closed subset of $X$ together with its foliation by simple, pairwise disjoint complete geodesics of $X$. By a {\it foliation of a closed subset} $\lambda$ of $X$ by geodesics we mean a decomposition of $\lambda$ into a pairwise disjoint  simple complete geodesics such that each point $x\in\lambda$ has a neighborhood homeomorphic to $T\times I$ where $T$ is  homeomorphic to a closed subset of a compact geodesic arc and $I$ is an open interval corresponding to open arcs on geodesics. 
\end{definition}

\begin{remark}
The lift $\tilde{\lambda}$ of $\lambda$ to the universal covering $\tilde{X}$ is a $\pi_1(X)$-invariant closed subset of $\tilde{X}$ that is foliated by pairwise disjoint complete geodesics. 
Note that $\tilde{\lambda}$ being closed as a subset of $\tilde{X}$ is equivalent to it being closed as a subset of $G(\tilde{\lambda})$. \end{remark}

Unlike for compact hyperbolic surfaces, a geodesic lamination of a conformally hyperbolic infinite Riemann surface $X$ can foliate a subset of $X$ with non-empty interior. Indeed, since hyperbolic half-planes can be foliated by geodesic laminations a Riemann surface $X$ that contains a hyperbolic half-plane has this property. The same is true for Riemann surfaces with funnels because a funnel contains a hyperbolic half-plane.

 On the other hand, we show that no  subset of $X$ with non-empty interior can be foliated by a 
geodesic lamination if $X$ is equal to its convex core $\mathcal{C}(X)$. 
This follows by proving a result for the convex core of $X$.

\begin{proposition}
\label{prop:support-gl}
Let $X$ be an infinite Riemann surface equipped with its conformal hyperbolic metric. Any geodesic lamination ${\lambda}$ of the convex core $\mathcal{C}(X)$ is nowhere dense.  \end{proposition}

\begin{remark}
If $X=\mathcal{C}(X)$ then any geodesic lamination $\lambda$ of $X$ is nowhere dense in $X$. If $X\neq \mathcal{C}(X)$ then there are geodesic laminations of $X$ whose supports can have non-empty interiors. \end{remark}

\begin{proof}

Let $\tilde{\lambda}$ be the lift of ${\lambda}$. By our assumption, $\tilde{\lambda}$ is a subset of the convex core $\mathcal{C}(\Lambda (\pi_1(X)))$ of the limit set $\Lambda (\pi_1(X))\subset \partial_{\infty}\tilde{X}$.
We assume on the contrary that there exists a closed hyperbolic disk $U\subset \mathcal{C}(X)$ which is covered by the geodesics of ${\lambda}$ and seek a contradiction.  Let $\tilde{U}$ be a single component of the lift of $U$ to $\tilde{X}$. Then $\tilde{U}$ is a closed hyperbolic disk in $\tilde{X}$ that is covered by $\tilde{\lambda}$.

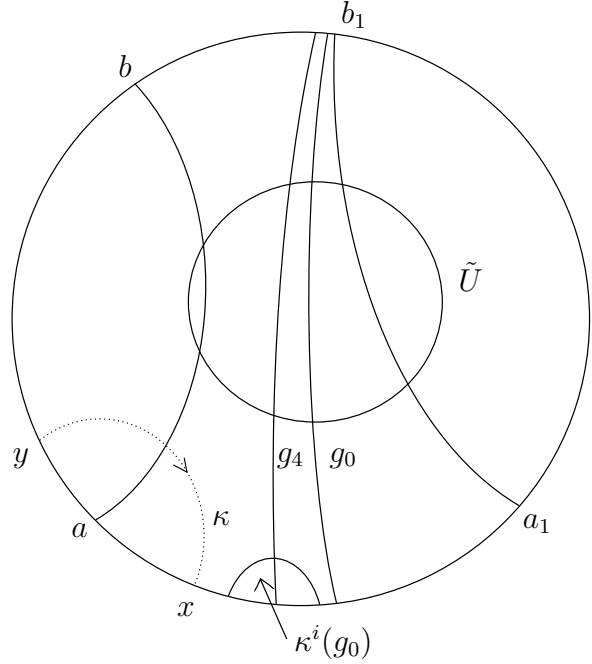
\begin{figure}[h]
\leavevmode \SetLabels
\L(.3*.18) $a$\\
\L(.39*.055) $x$\\
\L(.25*.3) $y$\\
\L(.49*.0) $\kappa^{i}(g_0)$\\
\L(.685*.195) $a_1$\\
\L(.52*.3) $g_0$\\
\L(.475*.3) $g_4$\\
\L(.42*.2) $\kappa$\\
\L(.63*.57) $\tilde{U}$\\
\L(.53*.995) $b_1$\\
\L(.34*.91) $b$\\
\endSetLabels
\begin{center}
\AffixLabels{\centerline{\epsfig{file =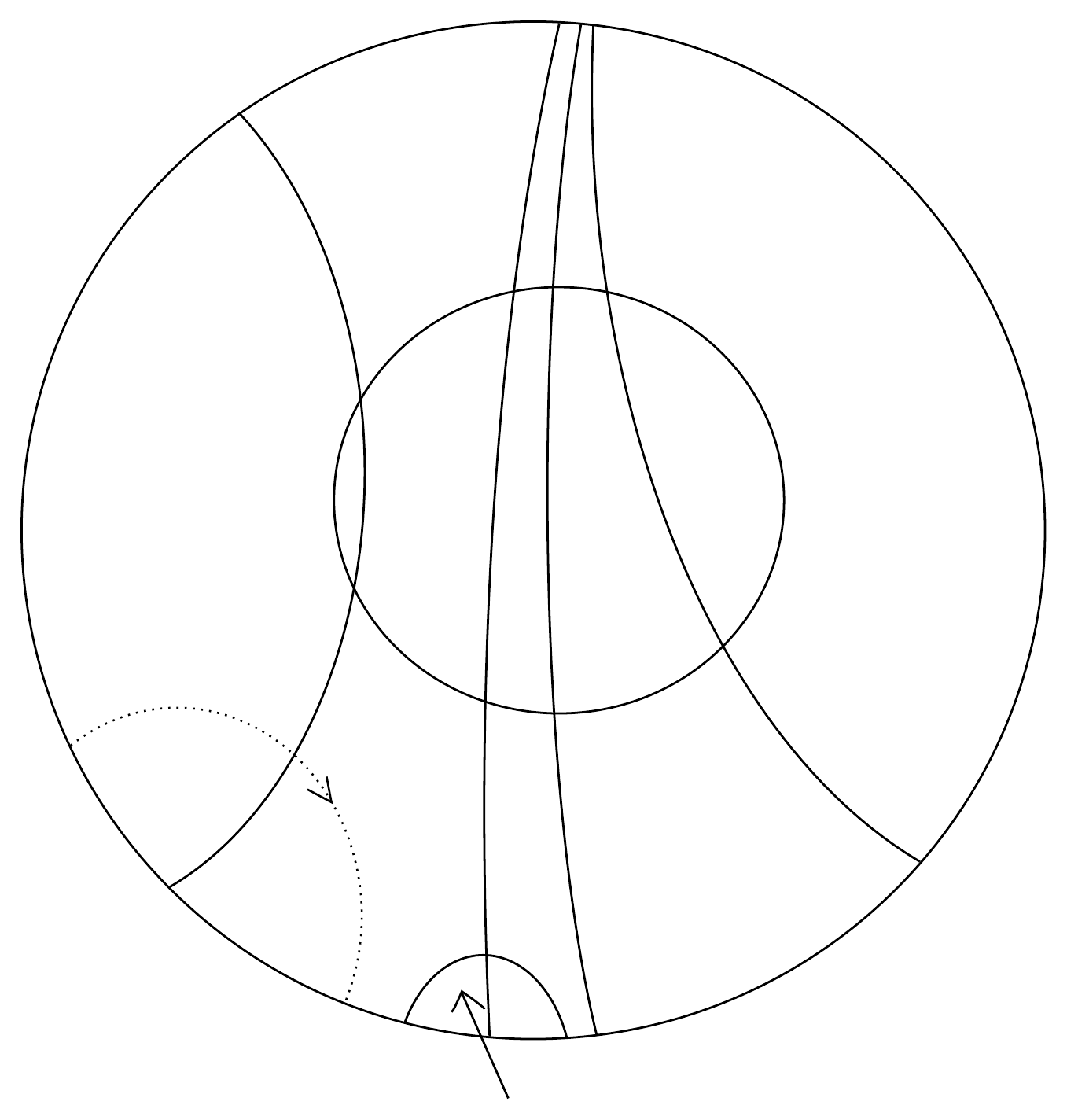,width=8.0cm,angle=0} }}
\vspace{-20pt}
\end{center}
\caption{The dotted geodesic is an axis of $\kappa\in\pi_1(X)$ and, in the general case, does not belong to $\tilde{\lambda}$. $\kappa^i(g_0)$ belongs to $\tilde{\lambda}$.} 
\end{figure}

A closed hyperbolic disk in the hyperbolic plane is said to be {\it regular} if for any three disjoint complete geodesics intersecting it one geodesic separates the other two. It is an elementary fact that there exists $r>0$ such that any closed disk of radius at most $r$ is regular. By decreasing $\tilde{U}$ if necessary, we can assume that it is regular.

Note that $\partial_{\infty}\tilde{X}$ has a natural orientation induced by the orientation of $\tilde{X}$. By definition $]a,b[\subset \partial_{\infty}\tilde{X}$ is the set of all $x\in \partial_{\infty}\tilde{X}$ such that $a$, $x$ and $b$ are in the order of this orientation, and $[a,b]=]a,b[\cup\{ a,b\}$.

Let $g=(a,b)$ and  $g_1=(a_1,b_1)$ be two geodesics of $\tilde{\lambda}$ that intersect $\tilde{U}$. Without loss of generality we can assume that they are in the relative position as in Figure 1. Since $\tilde{U}$ is regular then any geodesic of $\tilde{\lambda}$ between $g$ and $g_1$ separates them. In other words, it has one endpoint in $[a,a_1]$ and other endpoint in $[b_1,b]$. 

Let $K$ be the union of all endpoints of the geodesics of $\tilde{\lambda}$ that intersect $\tilde{U}$. Since $\tilde{\lambda}$ is a closed set and $\tilde{U}$ is a compact set it follows that $K$ is a closed subset of $\partial_{\infty}\tilde{X}$. 
Therefore each component of either $[a,a_1]\setminus K$ or $[b_1,b]\setminus K$ is an open arc if either of the two sets is non-empty.

Assume that $]a_2,a_3[$ is a component of $[a,a_1]\setminus K$. Let $g_2$ and $g_3$ be  geodesics of $\tilde{\lambda}$ with endpoints $a_2$ and $a_3$ such that their other two endpoints in $[b_1,b]$ are the closest to each other even possibly equal to each other. (We require this since either $a_2$ or $a_3$ could be endpoints of more than one geodesic of $      \tilde{\lambda}$.) Then the part of $\tilde{X}$ between $g_2$ and $g_3$ does not contain geodesics of $\tilde{\lambda}\cap [a,a_1]\times [b_1,b]$  and yet $\tilde{\lambda}$ covers $\tilde{U}$. This is a contradiction with $\tilde{U}$ being regular. Therefore the set $K$ covers $[a,a_1]$ and an analogous argument  shows that it covers $[b_1,b]$ as well.

Since $\tilde{\lambda}\subset \mathcal{C}(\Lambda (\pi_1(X)))$ it follows that $$[a,a_1], [b_1,b]\subset \Lambda (\pi_1(X)).$$
It is possible that either $[a,a_1]$ or $[b_1,b]$ is a single point. However not both intervals can be points since $\tilde{U}$ is covered by $\tilde{\lambda}$.

Assume first that $[a,a_1]$ is not a point. 
The set of endpoints of the lifts of closed geodesics of $X$ is dense in the limit set $\Lambda (\pi_1(X))$ (for example, see Maskit \cite{Maskit}). 
Let $\kappa\in \pi_1(X)$ be a hyperbolic translation with the attracting fixed point $x\in ]a,a_1[$ and the  repelling point $y\notin [b_1,b]$. (A hyperbolic element with these properties exist since $\pi_1(X)$ is not cyclic.) Let $g_0$ be a geodesic of $\tilde{\lambda}$ in $]a,a_1[\times [b_1,b]$ such that both its endpoints are different from $x$ and $y$. A high enough iterate $\kappa^{i}$ of $\kappa$ will attract both endpoint of $g_0$ into the interval $]a,a_1[$ (see Figure 1). Then $\kappa^{i}(g_0)$ transversely intersects a geodesic $g_4$ of $ \tilde{\lambda}$ which is in $[a,a_1]\times [b_1,b]$. However since  $\tilde{\lambda}$ is a geodesic lamination that is invariant under $\kappa$ we get a contradiction.
Thus $\tilde{\lambda}$ cannot cover a subset of $\mathcal{C}(\Lambda (\pi_1(X)))$ with non-empty interior and ${\lambda}$ cannot cover a subset of $\mathcal{C}(X)$ with non-empty interior. The argument is analogous when $[b_1,b]$ is not a point.
\end{proof}

A homeomorphism between two closed hyperbolic surfaces induces a natural homeomorphism between the spaces of geodesics of their universal coverings that is equivariant under the covering groups (see Bonahon \cite{Bonahon}, \cite{Bonahon-book}). 
In general,
a homeomorphism between two infinite Riemann surfaces does not induce an equivariant  homeomorphism between the spaces of geodesics of their universal coverings. This is easily seen by considering a homeomorphism which sends a cusp onto a funnel. However, we prove that when the hyperbolic metrics are such that the surfaces are equal to their convex cores then a homeomorphism of surfaces induces an equivariant homeomorphism of the spaces of geodesics of their universal coverings.

\begin{theorem}
\label{thm:home_geodesics}
Let $X$ and $X_1$ be two infinite Riemann surfaces which are equal to their convex cores. A homeomorphisms $f:X\to X_1$ induces a natural homeomorphism $\tilde{f}: G(\tilde{X})\to G(\tilde{X}_1)$ which projects to a well-defined map $f:G(X)\to G(X_1)$. Furthermore,  simple closed geodesics of $X$ are mapped onto simple closed geodesics of $X_1$ in the homotopy classes of their image curves under $f$.
\end{theorem}

\begin{proof}
A lift $\tilde{f}:\tilde{X}\to\tilde{X}_1$ of the homeomorphism $f:X\to X_1$ conjugates the action of $\pi_1(X)$ onto the action of $\pi_1(X_1)$. Since the sets of fixed points of hyperbolic elements of $\pi_1(X)$ and $\pi_1(X_1)$ are dense in $\partial_{\infty}\tilde{X}$ and $\partial_{\infty}\tilde{X}_1$ it follows that $\tilde{f}$ extends to an order preserving injective map $h$ from a dense subset of $\partial_{\infty}\tilde{X}$ onto a dense subset of $\partial_{\infty}\tilde{X}_1$. 

The universal covers  $\mathbb{R}\to \partial_{\infty}\tilde{X}$ and $\mathbb{R}\to\partial_{\infty}\tilde{X}_1$ are given by the exponential maps.
We lift the map $h$ to an increasing map $\tilde{h}$ from a dense subset of $\mathbb{R}$ onto a dense subset of $\mathbb{R}$. We claim that $\tilde{h}$ can be 
extended to a homeomorphism of $\mathbb{R}$. Indeed, let $x\in\mathbb{R}$ be a point where $\tilde{h}$ is not defined. Then there exists an increasing sequence 
$x_n$ that converges to $x$ such that $\tilde{h}$ is defined on $x_n$. 
Let $a$ and $b$ be the points of $\mathbb{R}$ on which $\tilde{h}$ is defined such that $a<x<b$. We can assume that $a<x_n<b$. 
Since $\tilde{h}$ is an increasing map we
have $\tilde{h}(a)<\tilde{h}(x_n)<\tilde{h}(b)$. Thus $\{\tilde{h}(x_n)\}$ is a bounded increasing sequence in $\mathbb{R}$ and therefore it has a limit $y\in\mathbb{R}$.
If $x_n'$ is another increasing sequence that converges to $x$ on whose elements $\tilde{h}$ is defined then $\tilde{h}(x_n')$ 
converges to some $y'$. For every $x_n'$ there exists $x_{k(n)}$ such that $x_n'<x_{k(n)}$ which implies $\tilde{h}(x_n')<\tilde{h}(x_{k(n)})<y$. By letting $n\to\infty$ we obtain $y'\leq y$ and by changing the roles of $x_n$ and $x_n'$ we obtain $y\leq y'$. Thus $y=y'$ and we have a well-defined extension. 

Therefore we extended $\tilde{h}$ to a map of $\mathbb{R}$ into $\mathbb{R}$ which can easily be seen to be an increasing map. It remains to be proved that $\tilde{h}$ is onto. Let $z\in\mathbb{R}$. Since $\tilde{h}(\mathbb{R})$ is dense in $\mathbb{R}$ there exists an increasing sequence $y_n=\tilde{h}(x_n)$ that converges to $z$. Since $\tilde{h}$ is increasing it follows that the sequence $x_n$ is increasing and let $x$ be its limit. Then $\tilde{h}(x)=z$ by the definition of $\tilde{h}$ and we established that $\tilde{h}$ is onto. Therefore $\tilde{h}$ is a continuous bijection. Finally $\tilde{h}$ is invariant under the covering transformations on a dense subset of $\mathbb{R}$ and therefore it is invariant under covering transformations on all of $\mathbb{R}$. Therefore $\tilde{h}$ projects to a continuous bijection between  $\partial_{\infty}\tilde{X}$ and $\partial_{\infty}\tilde{X}_1$ which is a homeomorphism by the invariance of domains theorem.

Finally a homeomorphism of boundaries $\partial_{\infty}\tilde{X}$ and $\partial_{\infty}\tilde{X}_1$  extends to a homeomorphism $\tilde{f}:G(\tilde{X})\to G(\tilde{X}_1)$ which is invariant under the actions of $\pi_1(X)$ and $\pi_1(X_1)$. Thus $f$ induces a natural bijection between $G(X)$ and $G(X_1)$. Since simple closed geodesics on $X$ are mapped by $f$ to simple closed homotopically non-trivial curves on $X_1$ the last statement of the theorem follows.
\end{proof}

\section{Train tracks from the pants decompositions}

The theory of train tracks on closed surfaces was started by Thurston \cite{Thurston} and further developed by various authors: Penner-Harer \cite{PennerHarer}, Bonahon \cite{Bonahon-book} to name a few. 
In this section we introduce train tracks for infinite hyperbolic surfaces (that are equal to their convex cores) and in doing so we use the ideas which are developed for closed surfaces. Unlike for closed surfaces, the Milnor-\v Svarc lemma does not hold for infinite surfaces. Our results rely upon a fact that edge paths of the constructed train track when lifted to the universal covering converge to well-defined ideal boundary points of the universal covering. This fact needs to hold for an arbitrary surface without a priori control of the geometry. For this reason we adopt standard Dehn-Thurston train tracks on closed surfaces to infinite hyperbolic surfaces starting from a fixed geodesic pants decomposition. When developing basic facts for these train tracks we mostly use the approach and terminology of Bonahon's book (see \cite{Bonahon-book}) and Penner-Harer's book (see \cite{PennerHarer}). 

Throughout this section we assume that $X$ is an infinite Riemann surface equipped with its conformal hyperbolic metric such that $\mathcal{C}(X)=X$. Therefore the Riemann surface $X$ has no funnels and no half-plane ends.  
We fix  a locally finite geodesic pants decomposition $\{ P_n\}_n$ of $X$. Denote by $\{\alpha_k\}_k$ the collection of cuffs  of the pants decomposition $\{ P_n\}_n$. 
Any simple geodesic of $X$ either intersects a cuff, or it is a cuff, or it belongs to a single pair of pants and accumulates at the cuffs at both of its ends.
We introduce an uncountable family of train tracks on $X$ starting from the fixed geodesics pants decomposition $\{ P_n\}$ by connecting the boundary geodesics (cuffs)
$\{\alpha_k\}_k$ with ``standard train tracks'' inside the pairs of pants as follows.

On each cuff $\alpha_k$ we choose a base point $a_k\in\alpha_k$. The base points $a_k$ will belong to the set of vertices of the train track on $X$ that we will construct below. There will be additional vertices in each pair of pants but no additional vertices on $\alpha_k$. Each cuff $\alpha_k$ is an edge of the train track that has both ends equal to $a_k$ called a {\it cuff edge}. 
If a pair of pants $P_n$ has $3$ cuffs $\alpha_1^n$, $\alpha_2^n$ and $\alpha_3^n$, then we have a choice of four standard train tracks inside $P_n$ that meet cuffs at $a_1^n$, $a_2^n$ and $a_3^n$ as in Figure 2 (see also \cite{PennerHarer}).
In addition, at each vertex $a_i^n$ we have two possible choices of smoothing and Figure 2 represents one such choice. This gives a total of $32$ choices on a single pair of pants with $3$ cuffs. 

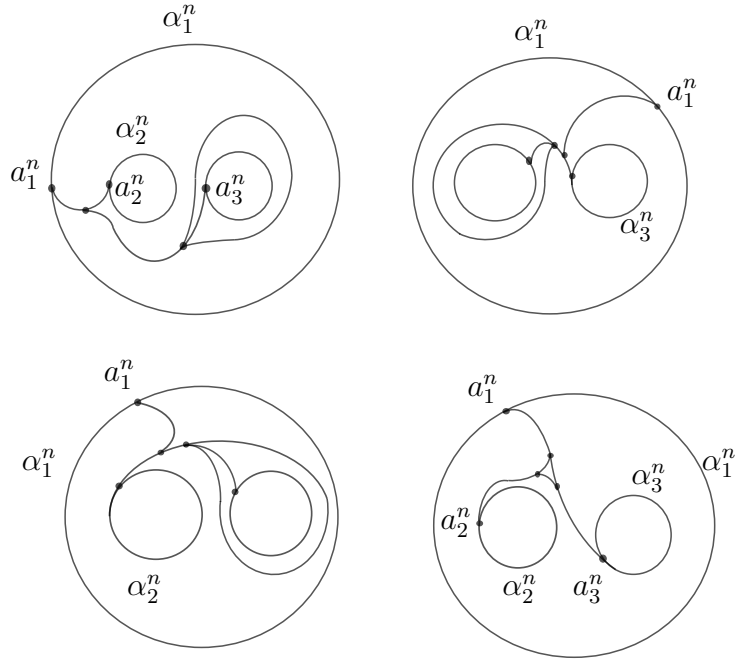
\begin{figure}[h]
\leavevmode \SetLabels
\L(.3*.9) $\alpha^n_1$\\
\L(.26*.75) $\alpha^n_2$\\
\L(.69*.62) $\alpha^n_3$\\
\L(.17*.69) $a^n_1$\\
\L(.259*.67) $a^n_2$\\
\L(.345*.67) $a^n_3$\\
\L(.6*.88) $\alpha^n_1$\\
\L(.73*.8) $a^n_1$\\
\L(.18*.3) $\alpha^n_1$\\
\L(.25*.42) $a^n_1$\\
\L(.76*.3) $\alpha^n_1$\\
\L(.56*.4) $a^n_1$\\
\L(.7*.28) $\alpha^n_3$\\
\L(.65*.13) $a^n_3$\\
\L(.59*.13) $\alpha^n_2$\\
\L(.537*.22) $a^n_2$\\
\L(.27*.13) $\alpha^n_2$\\
\endSetLabels
\begin{center}
\AffixLabels{\centerline{\epsfig{file =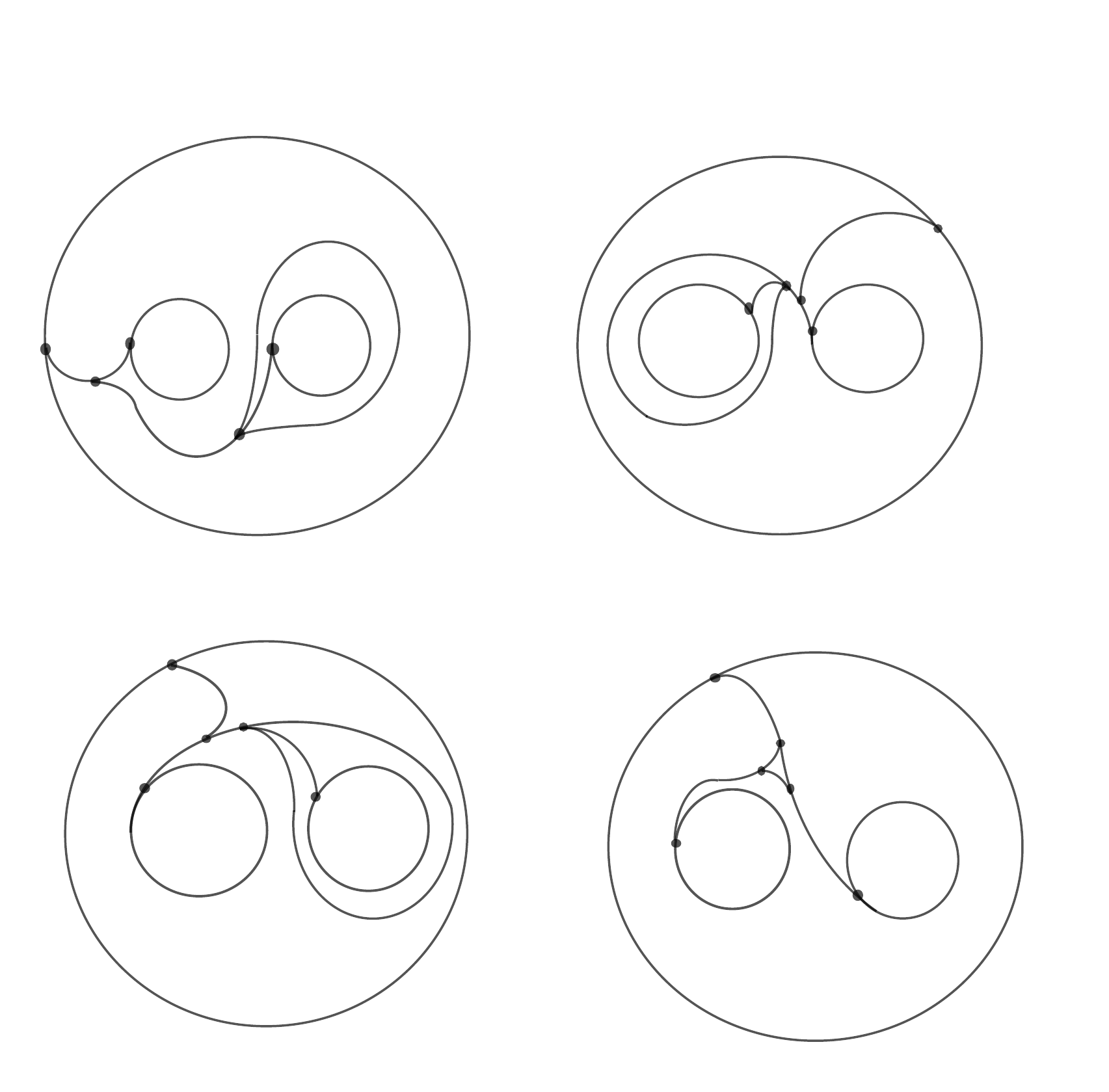,width=10.0cm,angle=0}}}
\vspace{-20pt}
\end{center}
\caption{The four standard train tracks on a pair of pants with $3$ cuffs. The smoothing at each cuff is chosen arbitrarily.} 
\end{figure}

If a pair of pants $P_n$ has $2$ cuffs $\alpha_1^n$ and $\alpha_2^n$, then there are two possible configurations of the ``standard train tracks'' in $P_n$ illustrated in the first two cases of  Figure 3. In addition, at each vertex $a_i^n$ we have two possible choices of smoothing which gives a total of $4$ choices. Finally if a pair of pants $P_n$ has $1$ cuff $\alpha^n_1$, then there is one possible configurations of the ``standard train track'' in $P_n$ illustrated in the last case of Figure 3 and we have two possible choices of smoothing at $a^n_1$ which gives a total of $2$ choices.

\begin{figure}[h]
\leavevmode \SetLabels
\L(.26*.9) $\alpha^n_1$\\
\L(.18*.7) $a^n_1$\\
\L(.29*.72) $\alpha^n_2$\\
\L(.61*.5) $a^n_1$\\
\L(.78*.8) $\alpha^n_1$\\
\L(.62*.787) $\alpha^n_2$\\
\L(.35*.2) $\alpha^n_1$\\
\L(.42*.4) $a^n_1$\\
\endSetLabels
\begin{center}
\AffixLabels{\centerline{\epsfig{file =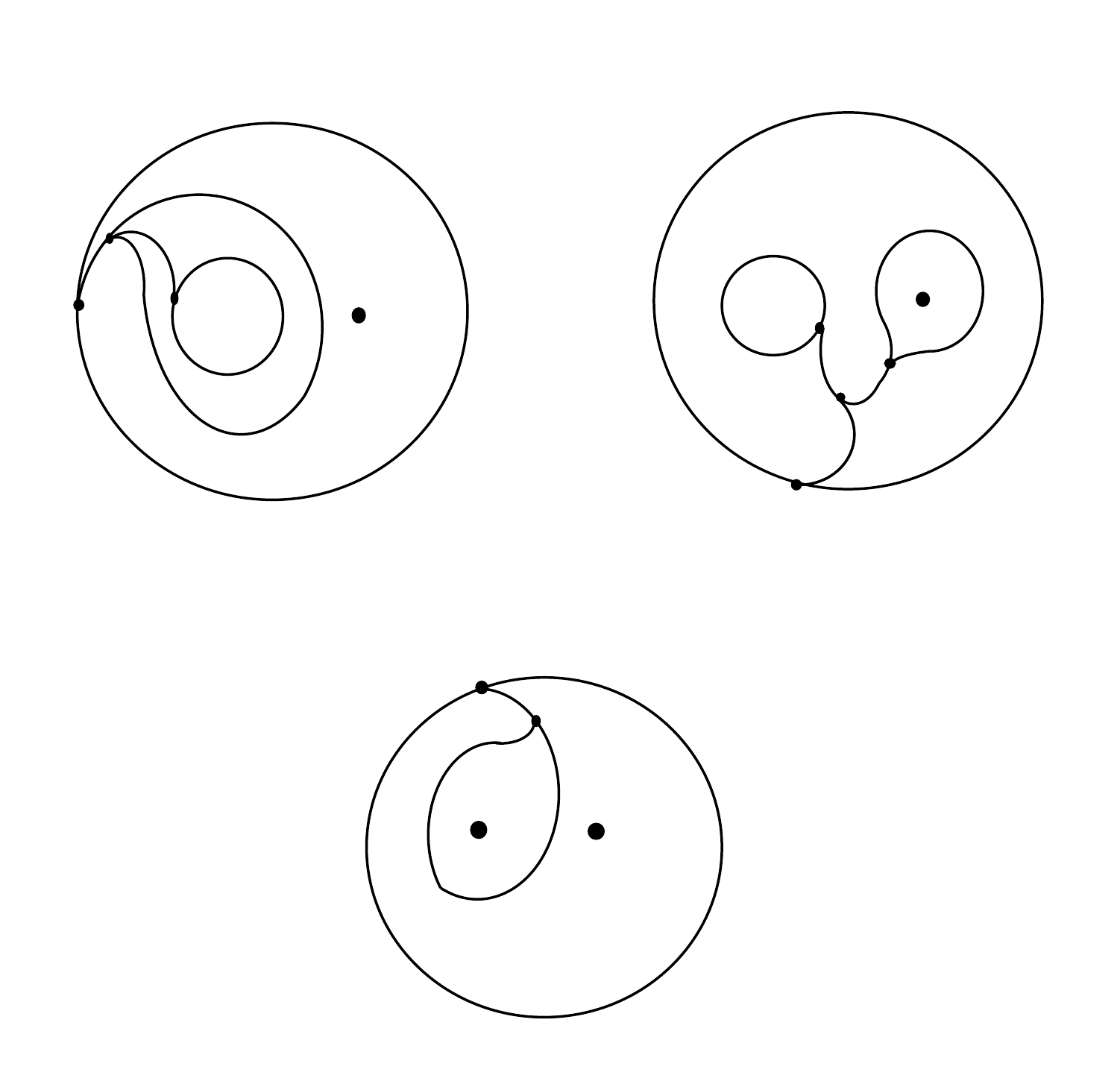,width=10.0cm,angle=0} }}
\vspace{-20pt}
\end{center}
\caption{The standard train tracks when pairs of pants have $2$ or $1$ cuff.} 
\end{figure}
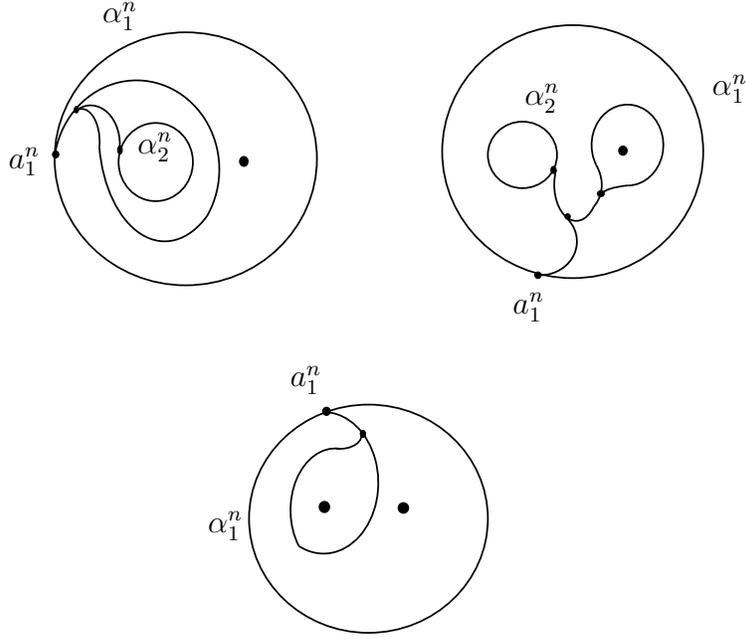

We form a train track $\Theta$ on $X$ as follows.
In the interior of each pair of pants $P_n$ we choose a standard train track as in Figures 2 and 3. The two standard train tracks can meet only at a base point of a cuffs on the boundary of the two pairs of pants. The edges of the standard train tracks in the interior of $P_n$  are called  {\it connector} edges and each cuff gives exactly one edge called a { cuff} edge. The vertices of $\Theta$ are formed by all the basepoints $a_k$ on the cuffs $\alpha_k$ and up to three vertices from the standard train tracks in the interior of each pair of pants (see Figures 2 and 3).

Let $e$ be a connector edge with one vertex $v$ at a cuff $\alpha_k$. We orient $e$ such that $v$ is the end point for the orientation. We further orient $\alpha_k$ such that the unit tangent vectors to $\alpha_k$ and $e$ at the vertex $v$ agree.
Then we will say that the connector edge $e$ is {\it left tangent} to $\alpha_k$ if $e$ is on the left of the oriented cuff $\alpha_k$. Otherwise the connector edge $e$ is {\it right tangent} to $\alpha_k$. Note that the notions of left and right tangent to a cuff are independent of any a priori orientation of the cuff and depends only on the orientation of the surface.

For each cuff $\alpha_k$, 
we introduce a requirement that the two connector edges (on opposite sides and) meeting $\alpha_k$ are either both left or both right tangent to $\alpha_k$.
By the definition of standard train tracks in geodesic pairs of pants, each complementary region of $\Theta$ is either a Jordan domain with piecewise smooth boundary and exactly three non-smooth points (called a triangle) or a  Jordan domain minus a point with smooth boundary except at one point (called a punctured monogon).

A train track $\Theta$ on $X$ obtained by making consistent choices of smoothing at each vertex on a cuff is called a {\it pants train track}. In fact, 
there are uncountably many pants train tracks on $X$  defined using  a fixed pants decomposition $\{ P_n\}$ of $X$.  

We fix one pants train track $\Theta$ on $X$ and denote by $\tilde{\Theta}$ the lift to $\tilde{X}$. The set of edges of $\Theta$ is denoted by $E(\Theta )$ and the set of edges of $\tilde{\Theta}$ is denoted by $E(\tilde{\Theta})$.  An {\it edge path} in $\tilde{\Theta}$ is a finite or infinite or bi-infinite sequence of edges of $\tilde{\Theta}$ such that the consecutive edges meet smoothly  at each vertex.

 We will need the following lemma. 
 
\begin{lemma}
\label{lem:lift_standard}
Let $\tilde{\alpha}_k$ be a single lift of a cuff and $\tilde{a}_k\in \tilde{\alpha}_k$ a lift of the base point $a_k$. Then the number of lifts of cuffs on one side and connected to $\tilde{\alpha}_k$ by finite edge paths starting at $\tilde{a}_k$ which are not crossing lifts of cuffs is at most four.
\end{lemma}

\begin{proof}
Note that the finite edge paths that we consider are obtained by lifting of the standard train tracks in the pair of pants adjacent to $\alpha_k$ corresponding to the given side of $\tilde{\alpha}_k$. We note that some standard train tracks have closed loops and a single lift of a closed loop can connect $\tilde{\alpha}_k$ to infinitely many lifts of cuffs. By considering how different smoothing effectively restrict closed curves in the standard train tracks to correspond to more than two edge paths in the lift, we conclude that the total number is at most four (see Figure 4).
\end{proof}

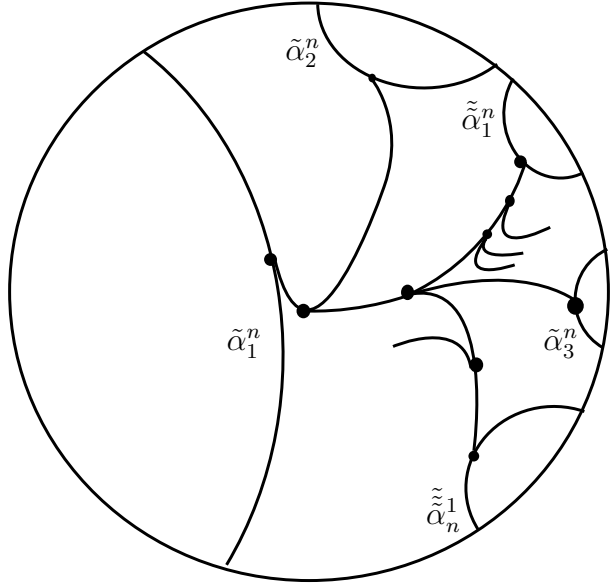
\begin{figure}[h]
\leavevmode \SetLabels
\L(.43*.4) $\tilde{\alpha}^n_1$\\
\L(.48*.9) $\tilde{\alpha}^n_2$\\
\L(.63*.78) $\tilde{\tilde{\alpha}}^n_1$\\
\L(.7*.4) $\tilde{\alpha}^n_3$\\
\L(.6*.1) $\tilde{\tilde{\tilde{\alpha}}}_n^1$\\
\endSetLabels
\begin{center}
\AffixLabels{\centerline{\epsfig{file =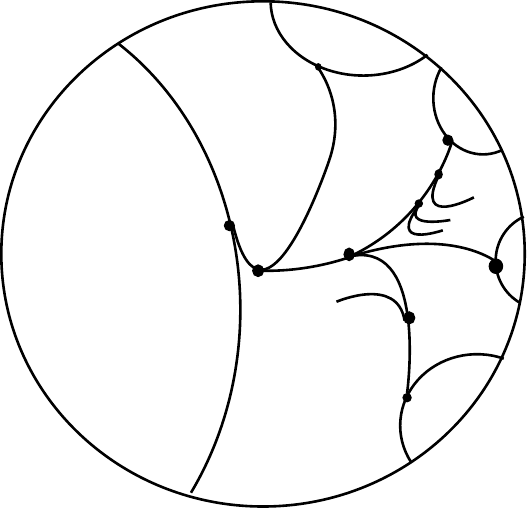,width=8.0cm,angle=0}}}
\vspace{-20pt}
\end{center}
\caption{The lift of a standard train track from upper left of Figure 2. There are four lifts of cuffs connected by finite edge paths to a single vertex of the lift of a cuff on the left side of the figure.} 
\end{figure}

\begin{remark} 
\label{rem:leave} By our choice connector edges of $\Theta$ that meet a cuff are either all left tangent or all right tangent to the cuff and the same is true for the lifted train track $\tilde{\Theta}$ and lifts of cuffs. Let $\tilde{\gamma} =(e_1,e_2,\ldots )$ be an edge path of $\tilde{\Theta}$. We orient edges $e_i$ such that the end point of $e_i$ is the initial point of $e_{i+1}$.  Assume that the connector edge $e_1$ starts on a lift $\tilde{\alpha}_1$ of a cuff and let $e_i$ be
the next connector edge of $\tilde{\gamma}$ that intersects a lift of a cuff
$\tilde{\alpha}_2$. The lifts of cuffs $\tilde{\alpha}_1$ and $\tilde{\alpha}_2$ are on the boundary of a single component $\tilde{P}$ of the lift of a pair of pants $P$ from the fixed pants decomposition of $X$. Then either all other edges of $\tilde{\gamma}$ after $e_i$  remain on $\tilde{\alpha}_2$ or, an at most finitely many edges $(e_i,\ldots ,e_k)$ of $\tilde{\gamma}$ remain on $\tilde{\alpha}_2$ and the edge $e_{k+1}$ leaves $\tilde{P}$. This follows because an edge path in $\tilde{\Theta}$ cannot come in at a lift of a cuff and leave it on the same side of the lift of a cuff since connector edges meeting $\tilde{\alpha}_2$  are either all left tangent or all right tangent to $\tilde{\alpha}_2$. In other words, an edge path $\tilde{\gamma}$ cannot ``bounce off'' from one boundary of $\tilde{P}$ back inside $\tilde{P}$ to meet another boundary component. 
\end{remark}

Let $\tilde{\gamma} =(e_1,e_2,e_3,\ldots , e_n,\ldots )$ be an infinite edge path of the train track $\tilde{\Theta}$. We will say that $\tilde{\gamma}$ {\it crosses} a lift $\tilde{\alpha}_k$ of a cuff $\alpha_k$ of the fixed pants decomposition $\{ P_n\}$ if $\tilde{\gamma}$ has edges in both complementary half-planes of $\tilde{\alpha}_k$ in $\tilde{X}$. We prove the following technical lemma for the later use.

\begin{lemma}
\label{lem:crossing}
Let $\tilde{\gamma}$ be an infinite edge path on the train track $\tilde{\Theta}$ such that no tail of $\tilde{\gamma}$ lies on a single lift $\tilde{\alpha}_k$ of a cuff $\alpha_k$. Then $\tilde{\gamma}$ crosses infinitely many lifts of cuffs and if $\tilde{\gamma}$ crosses $\tilde{\alpha}_k$ it cannot return to it later.
\end{lemma}

\begin{proof}
We first prove the second claim.
Orient each edge $e_n$ of $\tilde{\gamma}$ such that its end point coincide with the initial point of $e_{n+1}$.  Let $e_{n_0}\in \tilde{\gamma}$ be a connector edge with the initial point on a lift $\tilde{\alpha}_k$ of a cuff $\alpha_k$ (necessarily $e_{n_0}\nsubseteq\tilde{\alpha}_k$) and assume on the contrary that   $\tilde{\gamma}$ meets $\tilde{\alpha}_k$ in an another connector edge $e_{n_1}$ for some $n_1>n_0$. Consider the sequence $\tilde{\alpha}_k,\tilde{\alpha}_{k+1},\ldots \tilde{\alpha}_l$ of lifts of cuffs that have points in common with the edge path $(e_{n_0},e_{n_0+1},\ldots ,e_{n_1})$. By Remark \ref{rem:leave} each $\tilde{\alpha}_{j+1}$ separates $\tilde{\alpha}_j$ and $\tilde{\alpha}_{j+2}$ and thus the edge path $(e_{n_0},e_{n_0+1},\ldots ,e_{n_1})$ cannot return to $\tilde{\alpha}_k$ which proves the second claim.

We prove the first claim next. The edge path $\tilde{\gamma}$ does not cross infinitely many lifts of cuffs if and only if its tail remains in the closure of a single component $\tilde{P}_n$ of the lift of a pair of pants $P_n$. If the tail stays on a lift of the same cuff then the assumptions in the lemma are violated. Therefore we can assume that the tail has edges on the infinitely many lifts of cuffs which are on the boundary of $\tilde{P}_n$. Thus the tail contains a finite path of connector  edges $\{ e_i,e_{i+1},\ldots ,e_j\}$ connecting a lift $\tilde{\alpha}_k$ of a cuff to the lift $\tilde{\alpha}_{k+1}$ of a cuff.
By Remark \ref{rem:leave} the tail either stays forever on $\tilde{\alpha}_{k+1}$ or it leaves $\tilde{P}_n$ through $\tilde{\alpha}_{k+1}$ which is a contradiction. Thus the first claim follows.
\end{proof}

The above proof gives more detailed information on the nested sequence of lifts of cuffs that are crossed by an infinite edge path of $\tilde{\Theta}$ which we state as a separate lemma.

\begin{lemma}
\label{lem:nested}
Let $\tilde{\gamma}=(e_1,e_2,e_3,\ldots ,e_n,\ldots )$ be an infinite edge path in $\tilde{\Theta}$ and let $\{\tilde{\alpha}_k\}_{k=1}^{\infty}$ be the nested sequence of lifts of cuffs that are crossed by $\tilde{\gamma}$ in the given order. Each consecutive pair $(\tilde{\alpha}_k,\tilde{\alpha}_{k+1})$ of lifts of cuffs is on the boundary of a single component of a  lift of a pair of pants. Moreover $\tilde{\alpha}_k$ is connected to $\tilde{\alpha}_{k+1}$  by a unique finite edge path of $\tilde{\gamma}$ which consists of only connector edges that are lifts of a standard train track in a single pair of pants.
\end{lemma} 

We prove that two lifts of cuffs in $\tilde{X}$ (which need not be adjacent) are connected by an at most one edge path of $\tilde{\Theta}$. We will say that a sequence of lifts $\{\tilde{\alpha}_j\}_j$ is {\it nested} if for each $j_0$ a half-plane with $\tilde{\alpha}_{j_0}$ on its boundary contains $\tilde{\alpha}_j$ for all $j>j_0$.

\begin{lemma}
\label{lem:unique_connection}
Let $\tilde{\Theta}$ be the lift to the universal cover $\tilde{X}$ of the above constructed pants train track $\Theta$. Then for any two different lifts $\tilde{\alpha}_1$ and $\tilde{\alpha}_2$ of cuffs there is an at most one finite edge path of $\tilde{\Theta}$ that connects them.
\end{lemma}

\begin{proof}
Any two lifts of cuffs on a single component of a lift of a pair of pants are connected by an at most one path of connector edges. Indeed, if they are connected by two  paths of connector edges then these two paths together with two subarcs of the lifts of cuffs form a Jordan domain with smooth boundary except at two points with zero angle. This domain is union of the complementary triangles of $\tilde{\Theta}$. A standard application of Poincar\'e-Hopf theorem gives a contradiction (see \cite[Page 24]{Bonahon-book} or \cite[Page 7]{PennerHarer}). Therefore an at most one path of connector edges connects two lifts of cuffs in a single component of the lift of a pair of pants. 

Assume that lifts of cuffs $\tilde{\alpha}_1$ and $\tilde{\alpha}_2$ are connected by two finite edge paths $\tilde{\gamma}_1$ and $\tilde{\gamma}_2$ in $\tilde{\Theta}$. The nested families of lifts of cuffs that $\tilde{\gamma}_1$ and $\tilde{\gamma}_2$ are crossing are identical. Indeed, if they are not identical then one edge path would have two different subpaths connecting two lifts of cuffs on a single component of the lift of a pair of pants which is impossible by the above paragraph.

Thus the connector edges of $\tilde{\gamma}_1$ and $\tilde{\gamma}_2$ are identical. The   cuff edges of $\tilde{\gamma}_1$ and $\tilde{\gamma}_2$ are determined uniquely by the connector edges and therefore $\tilde{\gamma}_1=\tilde{\gamma}_2$.
\end{proof}

We prove that each infinite edge path on $\tilde{\Theta}$ accumulates to a unique point on the boundary $\partial_{\infty} \tilde{X}$ of $\tilde{X}$ which is the key result for encoding simple geodesics using bi-infinite edge paths.

\begin{proposition}
\label{prop:infinite_edge_paths}
An infinite edge path $\tilde{\gamma}=(e_1,e_2,\ldots ,e_n,\ldots )$ in $\tilde{\Theta}$ has a unique accumulation point on $\partial_{\infty}\tilde{X}$. Moreover two infinite edge paths $\tilde{\gamma}$ and $\tilde{\gamma}_1$ in $\tilde{\Theta}$ have the same accumulation point if and only if the edge paths have the same tails up to renumbering. 
Thus a  bi-infinite edge path $\tilde{\gamma}=(\ldots ,e_{-n},\ldots ,e_{-1},e_0,e_1,\ldots ,e_n,\ldots )$ has two distinct accumulation points on $\partial_{\infty}\tilde{X}$. 
\end{proposition}

\begin{proof}
Consider an infinite edge path $\tilde{\gamma}=(e_1,e_2,e_3,\ldots ,e_n,\ldots )$ in $\tilde{\Theta}$ and we orient each edge $e_n$ such that its endpoint is the initial point of $e_{n+1}$. If there is $n_0$ such that all edges $e_n$ for $n\geq n_0$ belong to a single lift $\tilde{\alpha}_k$ of a cuff $\alpha_k$ then the unique accumulation point of $\tilde{\gamma}$ is the appropriate endpoint of $\tilde{\alpha}_k$.

 A single lift $\tilde{\alpha}_k$ of a cuff $\alpha_k$ of the fixed pants decomposition of $X$ is a bi-infinite geodesic in $\tilde{X}$ and it divides $\tilde{X}$ into two hyperbolic half-planes. Since $\mathcal{C}(\Lambda (\pi_1(X)))=\tilde{X}$ and the fixed geodesic pants decomposition is locally finite, each nested sequence of lifts of cuffs accumulates to a single point on $\partial_{\infty}\tilde{X}$. 
 
By Lemma \ref{lem:crossing} an infinite edge path $\tilde{\gamma}$ whose tail does not lie on a single lift of a cuff intersects a nested sequence of lifts $\tilde{\alpha}_j$ of the cuffs without backtracking and therefore it accumulates to a single point on $\partial_{\infty}\tilde{X}$.  This proves the first statement. 
 
We prove the second statement in the proposition. If $\tilde{\gamma}$ and $\tilde{\gamma}_1$ have the same tails then they converge to the same point. 

Conversely assume that $\tilde{\gamma}$ and $\tilde{\gamma}_1$ accumulate to the same point  $x\in\partial_{\infty}\tilde{X}$. The point $x$ is either an endpoint of a lift of a cuff or the accumulation point of a nested sequence of lifts of cuffs. In the former case both $\tilde{\gamma}$ and $\tilde{\gamma}_1$ must eventually lie on the lift of the cuff because otherwise the separation and no backtracking  properties from Lemmas \ref{lem:crossing} and \ref{lem:unique_connection} would not allow this convergence. Thus $\tilde{\gamma}$ and $\tilde{\gamma}_1$ agree on their tails. 

Assume now that $x$ is the accumulation of a nested sequence of consecutive lifts $\{\tilde{\alpha}_k\}_{k=1}^{\infty}$ of cuffs. By Lemma \ref{lem:crossing}, $\tilde{\gamma}$ and $\tilde{\gamma}_1$ intersect the family $\{\tilde{\alpha}_k\}_{k=k_0}^{\infty}$ for some $k_0\geq 0$ because they have the same endpoint.
Then Lemma \ref{lem:unique_connection} implies that $\tilde{\gamma}$ and $\tilde{\gamma}_1$ agree on their tails. This finishes the proof of the second statement.

 Consider a bi-infinite edge path $\tilde{\gamma}=(\ldots ,e_{-n},\ldots ,e_{-1},e_0,e_1,\ldots ,e_n,\ldots )$ in $\tilde{\Theta}$. Divide it into two infinite edge paths $\tilde{\gamma}_1=(\ldots ,e_{-n},\ldots ,e_{-1},e_0)$ and $\tilde{\gamma}_2=(e_0,e_1,\ldots ,e_n,\ldots )$. The two paths have different tails and therefore they converge to different points on $\partial_{\infty}\tilde{X}$.
\end{proof}

Given a bi-infinite edge path $\tilde{\gamma}$ of $\tilde{\Theta}$ we denote by $G(\tilde{\gamma})$ the geodesic of $\tilde{X}$ whose endpoints on $\partial_{\infty}\tilde{X}$ are the two accumulation points of $\tilde{\gamma}$. We will say that a geodesic $\tilde{g}$ of $\tilde{X}$ is {\it weakly carried} by $\tilde{\Theta}$ if there exists a bi-infinite edge path $\tilde{\gamma}$ in $\tilde{\Theta}$ such that $G(\tilde{\gamma})=\tilde{g}$. 

\begin{proposition}
\label{prop:geodesics-edge_paths}
There is a one to one correspondence between bi-infinite edge paths of $\tilde{\Theta}$ and geodesics weakly carried by $\tilde{\Theta}$.
\end{proposition}

\begin{proof}
Assume that $G(\tilde{\gamma})=G(\tilde{\gamma}_1)=g$ for two bi-infinite edge path in $\tilde{\Theta}$. Then $\tilde{\gamma}$ and $\tilde{\gamma}_1$ have to agree on both of their tails by Proposition \ref{prop:infinite_edge_paths}. Therefore we have at most a finite subpath of $\tilde{\gamma}$ and a finite subpath of $\tilde{\gamma}_1$ that have the same end edges but might not agree on the interior edges. This is not possible by Lemma \ref{lem:unique_connection}. Thus $\tilde{\gamma}=\tilde{\gamma}_1$ after possible renumbering the sequences. 
\end{proof}

Let ${\gamma}$ be an edge path in $\Theta$ and $\tilde{\gamma}$ a single component lift of ${\gamma}$ to the universal covering. Recall that $G(\tilde{\gamma} )$ is a geodesic whose endpoints are equal to the endpoints of $\tilde{\gamma}$. For any $\kappa\in\pi_1(X)$ we have that $\kappa (G(\tilde{\gamma} ))=G(\kappa (\tilde{\gamma} ))$. Define $G({\gamma})$ to be the projection of $G(\tilde{\gamma} )$ onto $X$. We will say that a geodesic $g=G(\gamma )$ is {\it weakly carried} by $\Theta$. Let $G(\Theta )$ denote the set of all geodesics weakly carried by $\Theta$.
Proposition \ref{prop:geodesics-edge_paths} immediately gives

\begin{proposition}
\label{prop:geodesics-edge_paths_on_X}
There is a one to one correspondence between bi-infinite edge paths of ${\Theta}$ and geodesics weakly carried by ${\Theta}$.
\end{proposition}

We describe the convergence of geodesics in terms of the corresponding bi-infinite edge paths in $\tilde{\Theta}$. Denote by $G(\tilde{\Theta})$ the set of all geodesics of $\tilde{X}$ that are weakly carried by $\tilde{\Theta}$. 

\begin{proposition}
\label{prop:convergence_edge_path}
Let $g_n, g\in G(\tilde{X})$ be weakly carried by a train track $\tilde{\Theta}$. Denote by $\tilde{\gamma}_n,\tilde{\gamma}$ the corresponding bi-infinite edge paths in $\tilde{\Theta}$. Then $g_n$ converges to $g$ as $n\to\infty$ if and only if for each finite subpath $\tilde{\gamma}'$ of $\tilde{\gamma}$ there is $n_0\geq 0$ such that $\tilde{\gamma}’$ is contained in the path $\tilde{\gamma}_n$ for all $n\geq n_0$.
\end{proposition}

\begin{proof}
Assume $g_n\to g$ as $n\to\infty$. Each endpoint of $g$ corresponds to an infinite tail of $\tilde{\gamma}$. 
We fix a finite subpath $\tilde{\gamma}’$ of $\gamma$ and need prove that there exist $n_0\geq 0$ such that $\tilde{\gamma}_n$ contains $\tilde{\gamma}’$ for $n\geq n_0$.
The proof is divided into several cases.

\vskip .2 cm
\noindent {\it Case 1.} The first case is when both tails of $\tilde{\gamma}$ are crossing infinitely many lifts of cuffs. Denote by $\{\tilde{\alpha}_k\}_{k=-\infty}^{\infty}$ the lifts of cuffs that $\tilde{\gamma}$ intersects in the given order. Let $x\in\partial_{\infty}\tilde{X}$ be the limit of the nested sequence of cuffs $\{\tilde{\alpha}_k\}_{k=-\infty}^0$ and $y$ the limit  of the nested sequence of cuffs $\{\tilde{\alpha}_k\}_{k=0}^{\infty}$. 

We can assume without loss of generality that the finite subpath $\tilde{\gamma}'$ connects the lifts $\tilde{\alpha}_{-k_0}$ and $\tilde{\alpha}_{k_0}$ of cuffs.  Since $G(\tilde{\gamma}_n)=g_n$ converges to $G(\tilde{\gamma} )=g$ we have that the endpoints $x_n,y_n$ of $G(\tilde{\gamma}_n)$ converge to $x,y$. Thus  
there is $n_0\geq 0$ such that 
$x_n$ and $x$ are separated from $y$ by a half-plane with  boundary $\tilde{\alpha}_{-k_0}$, and points $y_n$ and $y$ are separated from $x$ by a half-plane with boundary $\tilde{\alpha}_{k_0}$ for $n\geq n_0$. Therefore the edge path $\tilde{\gamma}_n$ connects $\tilde{\alpha}_{-k_0}$ and $\tilde{\alpha}_{k_0}$ for $n\geq n_0$. This implies that $\tilde{\gamma}_n$ contains $\tilde{\gamma}'$ for $n\geq n_0$ by Lemma \ref{lem:unique_connection}.

\vskip .2 cm

\noindent {\it Case 2.}
Assume that one tail of $\tilde{\gamma}$ is on a single lift of a cuff $\tilde{\alpha}$ and the other tail crosses infinitely many lifts of cuffs. Then $g$ has one endpoint equal to an appropriate endpoint $x\in\partial_{\infty}\tilde{X}$ of $\tilde{\alpha}$ and the other endpoint $y$ of $g$ is accumulated by a  sequence of nested lifts of cuffs $\{\tilde{\alpha}_k\}_{k=1}^{\infty}$ such that  $\tilde{\alpha}_1$  and $\tilde{\alpha}$ are on the boundary of a component of the lift of a pair of pants from the decomposition. Since $g_n\to g$  endpoints $x_n,y_n\in\partial_{\infty}\tilde{X}$ of $g_n$ converge to the endpoints $x,y$ of $g$ as $n\to\infty$. 

If $x_n=x$ for $n\geq n_0\geq 0$ then the proof is similar to the Case 1. We assume now that $x_{n_k}\neq x$ for an infinite subsequence $x_{n_k}$ of $x_n$.
Denote by $\tilde{P}_x^1$ and $\tilde{P}_x^2$ the lifts of pairs of pants from the fixed pants decomposition that have $\tilde{\alpha}$ on their boundaries where $\tilde{P}_x^1$ contains $\tilde{\alpha}_1$ on its boundary and $\tilde{P}_x^2$ is on the other side of $\tilde{\alpha}$.
There are two 
possibilities: either $x_{n_k}$ contains an infinite  subsequence $x_{n_{k_j}}$ that is on the side of $\tilde{\alpha}$ that contains $\tilde{P}_x^1$ or the whole sequence $x_{n_k}$ is on the side of $\tilde{\alpha}$ that contains $\tilde{P}_x^2$. 

Assume first that $x_{n_{k_j}}$ is on the side of $\tilde{\alpha}$ that contains $\tilde{P}_x^1$. 
Since $x_{n_{k_j}}\to x$ and $x_{n_{k_j}}\neq x$
the edge path $\tilde{\gamma}_{n_{k_j}}$ intersects a boundary geodesic (a lift of a cuff) $\tilde{\beta}_{n_{k_j}}$ of  $\tilde{P}_x^1$ such that $x_{n_{k_j}}$ is separated from $x$ by the closed half-plane whose boundary in $\tilde{X}$ is $\tilde{\beta}_{n_{k_j}}$ and which does not contain $\tilde{P}_x^1$ (this includes the possibility that $x_{n_{k_j}}$ is an endpoint of $\beta_{n_{k_j}}$). 

For the convenience of the argument, we identify $\tilde{X}$ with the unit disk $\mathbb{D}$. We arrange the components of the complement of $\tilde{P}_x^1$ in the unit disk into a sequence $\{ C_l\}$ and note that their Euclidean size goes to zero as $l\to\infty$. Since $x_{n_{k_j}}\to x$ and $x_{n_{k_j}}\neq x$ it follows that $x_{n_{k_j}}\in C_{l(j)}$ such that $l(j)\to\infty$ as $j\to\infty$. Thus the Euclidean size of $\tilde{\beta}_{n_{k_j}}$ goes to zero and 
the endpoints of $\tilde{\beta}_{n_{k_j}}$ converge to $x$ as $j\to\infty$. 

Since $\tilde{\alpha}_1\neq \tilde{\alpha}$, the endpoints of $\tilde{\beta}_{n_{k_j}}$ do not converge to an endpoint of $\tilde{\alpha}_1$. By Lemma \ref{lem:lift_standard}, the set of all lifts of cuffs on the boundary of $\tilde{P}_x^1$ that are connected to $\tilde{\alpha}_1$ by finite edge paths is the orbit of finitely many lifts of cuffs under the cyclic stabilizer of $\tilde{\alpha}_1$ in the group $\pi_1(X)$. Consequently, any infinite sequence of distinct lifts of cuffs on the boundary of $\tilde{P}_x^1$  connected to $\tilde{\alpha}_1$ by finite edge paths consisting of connector edges converges to endpoints of $\tilde{\alpha}_1$.
Since the endpoints of $\tilde{\beta}_{n_{k_j}}$ do not converge to endpoints of $\tilde{\alpha}_1$, it follows that $\tilde{\beta}_{n_{k_j}}$ and $\tilde{\alpha}_1$ are not connected by a connector edge for $j$ large enough. Given that $\tilde{\beta}_{n_{k_j}}$ and $\tilde{\alpha}_1$ are on the boundary of $\tilde{P}_x^1$, Remark \ref{rem:leave} implies that no edge path in $\tilde{\Theta}$ connects them. Thus the geodesic $g_{n_{k_j}}$ for $j$ large is not weakly carried by $\tilde{\Theta}$ which is a contradiction and no subsequence $x_{n_{k_j}}$ is on the side of $\tilde{\alpha}$ that contains $\tilde{P}_x^1$.

Assume next that $x_{n_k}$ is on the same side of $\tilde{\alpha}$ as $\tilde{P}_x^2$. Let $\tilde{\beta}_{n_k}$ be the boundary side of $\tilde{P}_x^2$ different from $\tilde{\alpha}$ that $g_{n_k}$ (or equivalently the corresponding edge path $\tilde{\gamma}_{n_k}$) intersects. Then both endpoints of $\tilde{\beta}_{n_k}$ converge to $x$ by the same method as above. Since $\tilde{\gamma}_{n_k}$ enters $\tilde{P}_x^2$ through $\tilde{\alpha}$ it can leave it only through a boundary connected to $\tilde{\alpha}$ by a finite edge path of connector edges  of $\tilde{\Theta}$ by Lemma \ref{lem:nested}. Since the endpoints of $\tilde{\beta}_{n_k}$ converge to an endpoint $x$ of $\tilde{\alpha}$ it follows that the edge path $\tilde{\gamma}_{n_k}$ contains a finite edge path of $\tilde{\Theta}$ connecting $\tilde{\alpha}_1$ and $\tilde{\alpha}$ followed by a large number of cuff edges of $\tilde{\Theta}$ that lie on $\tilde{\alpha}$ and then followed by a finite edge path connecting $\tilde{\alpha}$ and $\tilde{\beta}_{n_k}$. Given a finite subpath $\tilde{\gamma}'$ of $\tilde{\gamma}$, we choose $k_0\geq 0$  such that all cuff edges of $\tilde{\gamma}'$ that are on $\tilde{\alpha}$ are contained in $\tilde{\gamma}_{n_k}$ for $k\geq k_0$. By choosing larger $k_0$ if necessary, the method in Case 1 gives that all edges on $\tilde{\gamma}'$ that are not on $\tilde{\alpha}$ are contained in $\tilde{\gamma}_{n_k}$ for $k\geq k_0$. Thus $\tilde{\gamma}_{n_k}$ contains a finite subpath of $\tilde{\gamma}'$ for $k\geq k_0$.

\vskip .2 cm

\noindent
{\it Case 3.} Assume that $\tilde{\gamma}$ consists of edges on a single lift $\tilde{\alpha}$ of a cuff. Let $x,y$ be the endpoints of $\tilde{\alpha}$ and let $x_n,y_n$ be the endpoints of $\tilde{\gamma}_n$ where $x_n\to x$ and $y_n\to y$ as $n\to\infty$.  Then the points $x_n$ and $y_n$ have to be on the opposite sides of $\tilde{\alpha}$ or we would have a contradiction similar to the Case 2. If they are on the opposite sides of $\tilde{\alpha}$ then the argument of the Case 2 shows that $\tilde{\gamma}_n$ contains a large number of cuff edges on 
 $\tilde{\alpha}$ which finishes the proof.
 
 \vskip .2 cm

\noindent 
{\it Case 4.} Assume that $\tilde{\gamma}$ has two tails in two different lifts of cuffs $\tilde{\alpha}_1$ and $\tilde{\alpha}_2$. Let $\tilde{\gamma}'$ be a finite subpath with an initial edge on $\tilde{\alpha}_1$ and a terminal edge on $\tilde{\alpha}_2$. By applying the method of Case 2 to $\tilde{\gamma}'$ we find $n_0\geq 0$ such that $\tilde{\gamma}_n$ contains the initial and terminal edges of $\tilde{\gamma}'$ for $n\geq n_0$. By the uniqueness of edge paths, $\tilde{\gamma}_n$ contains $\tilde{\gamma}'$ for $n\geq n_0$.
\end{proof}

\begin{definition}
Two bi-infinite edge paths $\tilde{\gamma} =(\ldots ,e_{-n},\ldots ,e_{-1},e_0,e_1\ldots ,e_n,\ldots )$ and $\tilde{\gamma}' =(\ldots ,e_{-n}',\ldots ,e_{-1}',e_0',e_1'\ldots ,e_n',\ldots )$ {\it cross each other} if, after possible reversing the orientation and renumbering, there exists $p<q$ such that $e_s=e_s'$ for $p< s< q$, and that  $e_p'$ and $e_q'$ lie on the opposite sides of $\tilde{\gamma}$. 
\end{definition}

It is clear that the geodesics represented by bi-infinite edge paths $\tilde{\gamma}$ and $\tilde{\gamma}'$ are intersecting if and only if $\tilde{\gamma}$ and $\tilde{\gamma}'$ cross each other. 
Then two geodesics on $X$ represented by edge paths $\gamma$ and $\gamma'$ of the train track $\Theta$  intersect each other if and only if their corresponding edge paths cross each other.
This also holds true when ${\gamma} ={\gamma}'$ and allow us to characterize simple geodesics as corresponding to bi-infinite edge paths that do not cross themselves.

 A geodesic lamination $\lambda$ of $X$ is {\it weakly carried} by $\Theta$ if every geodesic of $\lambda$ is weakly carried by $\Theta$. If $\lambda$ is weakly carried by $\Theta$ then its lift $\tilde{\lambda}$ is weakly carried by $\tilde{\Theta}$. The proposition below follows directly from previous discussions.

\begin{proposition}
\label{prop:geodesic_laminations}
The set of geodesic laminations on $X$ that are weakly carried by $\Theta$ is in a one to one correspondence with the families $\Gamma$ of bi-infinite edge paths that satisfy:
\begin{itemize}
\item Any two bi-infinite edge paths $\gamma$ and $\gamma'$ in $\Gamma$ do not cross, and
\item If $\gamma$ is a bi-infinite edge path such that for any finite edge subpath there is a bi-infinite edge path in $\Gamma$ that contains it, then $\gamma\in\Gamma$.
\end{itemize}
\end{proposition}

We prove that our special construction of pants train tracks is general enough for our purpose of giving a local parametrization of the space of measured geodesic laminations on $X$. Hence we will assume that geodesic laminations do not have leaves which spiral to the cuffs of the pants decomposition.

\begin{proposition}
\label{prop:carrying}
Fix a locally finite geodesic pants decomposition of a Riemann surface $X$ whose fundamental group $\pi_1(X)$ is of the first kind.
Given a geodesic lamination $\lambda$ on $X$ which does not have leaves spiraling to a cuff, there is a pants train track $\Theta$ such that $\lambda$ is weakly carried by $\Theta$. The pants train track $\Theta$ is constructed using the fixed pants decomposition.
\end{proposition}

\begin{proof}
Let $\{ P_n\}_n$ be a fixed locally finite geodesic pants decomposition of $X$ and $\{\alpha_k\}_k$ the family of all cuffs of $\{ P_n\}_n$. Consider a geodesic lamination $\lambda$ on $X$ which does not have a leaf spiraling to a cuff of $\{ P_n\}_n$. 
In each pair of pants $P_n$ the set of arcs ${\lambda}\cap P_n$ is divided into an at most three homotopy classes setwise fixing the cuffs. We choose a standard train track in $P_n$ that has finite edge paths homotopic to each arc in ${\lambda}\cap P_n$ setwise fixing the cuffs as in Figures 2 and 3 (see also \cite[Section 2.6]{PennerHarer}). The choice of a standard train track is unique in each $P_n$ where the number of homotopy classes of the arcs $\lambda\cap P_n$ is maximal. If the number of homotopy classes in $P_n$ is not maximal then there is a more than one choice of a standard train track and we fix one such choice. It remains to choose the smoothing at the cuffs.

The smoothing is done based on the standard construction of Dehn-Thurston coordinates for multi-curves and measured laminations on closed surfaces (see Penner-Harer \cite[Page 15]{PennerHarer}). We fix a closed arc $I_k$ on each cuff $\alpha_k$ which we call a {\it window} on $\alpha_k$ and denote its midpoint by $a_k$. Choose a regular neighborhood $U_k$ of each cuff $\alpha_k$ which is a hyperbolic collar of width equal half the standard collar width so that the closures of two different $U_k$ are disjoint.  On each boundary component $\partial_i U_k$ for $i=1,2$ of $U_k$ we denote by $I_k^i$ the closed arc that orthogonally project to $I_k$. The arcs $I_k^i$ are called {\it windows} of the regular neighborhood $U_k$ and denote their midpoints by $a_k^i$. Denote by $P_n^{*}$ the pair of pants obtained by deleting the neighborhoods $U_k$ of the cuffs of $P_n$.

In each $P_n$ we find an isotopy $H$ pointwise fixing the cuffs  that moves the family of arcs $\lambda\cap P_n$ into a family of arcs $\Lambda_n$ that enter one-sided regular neighborhoods of cuffs through the interior of their windows and once they enter these neighborhoods they can only leave them through the cuffs. In addition, we require that there exists a standard train track in $P_n^{*}$ with endpoints at $a_k^i$ such that for each arc $l'$ in $\Lambda_n\cap P_n^{*}$ there is a homotopy modulo windows $I_k^i$ onto a simple edge path of the standard track in $P_n^{*}$ and we choose this  standard train track on each $P_n^{*}$. 

 It remains to connect the standard train tracks in $P_n^{*}$ from their endpoints $a_k^i$ to the points $a_k\in\alpha_k$ by simple arcs inside one-sided neighborhoods (the half of $U_k$)  such that  they are either right or left tangent to $\alpha_k$ on both sides of $\alpha_k$.

Consider two-sided regular neighborhood $U_k$ of $\alpha_k$ and the family $\hat{\Lambda}_k$ of arcs in $U_k$  connecting the windows $I_k^1$ and $I_k^2$ obtained by the above isotopy $H$ of the arcs $\lambda\cap (P_n'\cup P_n'')$, where $P_n',P_n''$ are the two (geodesic) pairs of pants that have $\alpha_k$ on their boundaries. Let $c_k^i\in\partial I_k^i$ for $i=1,2$ be endpoints of the windows $I_k^i$ such that the shortest closed geodesic arc $b_k$ in $U_k$ that connects them is orthogonal to $\alpha_k$. For an arc $l\in\hat{\Lambda}_k$, the absolute value of the twist number equals the essential number of its intersections with $b_k$. Let $\hat{l}\subset U_k$ be an arc homotopic to $l$ modulo endpoints that realizes the twist number. The twist number of $l$ is positive if $\hat{l}$ goes to the right once it starts from the window $I_k^1$ and it is negative if it goes to the left. The twist is zero if there is no essential intersections. The twist number of any other $l'\in\hat{\Lambda}_k$ differs by at most one since arcs in $\hat{\Lambda}_k$ are disjoint.
  
The absolute value of the twist number around $\alpha_k$ is the maximum of the absolute value of the twist numbers over all arcs in $\tilde{\Lambda}_k$. The sign is positive if at least one arc has a positive twist number and it is negative if at least one arc has negative twist.

Since $I_k^1$ and $I_k^2$ orthogonally project to $I_k$, the set of points in $U_k$ whose orthogonal projection on $\alpha_k$ is inside $I_k$ forms a quadrilateral $Q_k$ whose two opposite sides are $I_k^1$ and $I_k^2$, and the other two opposite sides are geodesic arc $b_k$ and the geodesic arc connecting two endpoints of $I_k^1$ and $I_k^2$ different from $c_k^1$ and $c_k^2$. 
We connect the two standard train tracks in the pairs of pants $(P_n')^{*}$ and $(P_n'')^{*}$ at the point $a_k\in\alpha_k$ by connecting the points $a_k^1$ and $a_k^2$ inside $Q_k$. It is possible that $\alpha_k$ is on the boundary of a single pair of pants in which case we connect the standard train track to itself at the point $a_k$.  
The smoothing at $a_k$ is chosen such that the connector edge is coming from the right of $\alpha_k$ if the twist is positive and it is from the left if the twist is negative. If the twist is zero then the smoothing is arbitrary. Note that the incoming edges from both sides of $\alpha_k$ are simultaneously either right or left tangent. Since points $a_k^i$ became bivalent edges we can erase them and consider standard tracks in $P_n$ which have $a_k$ vertices on $\alpha_k$. 

The choices of isotopies in the pairs of pants minus regular neighborhoods of cuffs followed by the choices of homotopies in the regular neighborhood of cuffs guarantee that each geodesic of $\lambda$ is homotopic to an edge path in $\Theta$. In fact, we can arrange that this homotopy setwise fixes the cuffs $\alpha_k$. 
We need to show that each geodesic $g$ of the lift $\tilde{\lambda}$ to the universal cover $\tilde{X}$ of the geodesic lamination $\lambda$ has the same endpoints on $\partial_{\infty}\tilde{X}$ as an edge path in $\tilde{\Theta}$.

Let $\bar{g}$ be the geodesic of $\lambda$ whose one lift is $g$ and let $\bar{\gamma}$ be an edge path in $\Theta$ that is homotopic to $\bar{g}$. By lifting the homotopy between $\bar{g}$ and $\bar{\gamma}$ to the universal cover $\tilde{X}$ starting from the geodesic $g$ we obtain an edge path $\gamma$ in $\tilde{\Theta}$ which is a lift of $\bar{\gamma}$ and is homotopic to $g$. The homotopy of $g$ and $\gamma$ is not necessarily bounded in the hyperbolic metric of $\tilde{X}$.

Since $\bar{g}$ cannot accumulate to a cuff it follows that either both ends of $\bar{g}$ intersect cuffs infinitely many times or $\bar{g}$ is a cuff. Therefore $g$ is either a lift of a cuff or both ends of $g$ intersect two sequences of nested cuffs that accumulate to the ideal endpoints of $g$. If $g=\tilde{\alpha}_k$ then the construction of $\tilde{\Theta}$ allows us to take $\gamma$ to be the bi-infinte sequence of cuff edges of $\tilde{\alpha}_k$.

Assume that both ends of $g$ intersect nested sequences of lifts of cuffs. Since the homotopy between $\bar{g}$ and $\bar{\gamma}$ can be chosen such that it setwise fixes the cuffs, it follows that the homotopy between $g$ and $\gamma$ setwise preserves each lift of a cuff. Therefore $g$ and $\gamma$ intersect the same sequence of nested lifts of cuffs at both of their ends and they have the same endpoints on $\partial_{\infty}\tilde{X}$.
Thus we obtained that $g$ is weakly carried by $\tilde{\Theta}$. Since $g$ is a lift of an arbitrary geodesic of $\lambda$, it follows that $\lambda$ is weakly carried by $\Theta$.
\end{proof}

\section{Measured laminations carried by train tracks}

In this section we parametrize the set of all measured laminations on $X$ that are weakly carried by pants train tracks using the edge weight systems(see \cite{PennerHarer}, \cite{Bonahon} for the case of  a closed surface). 
We first introduce some standard definitions regarding measured laminations and train tracks analogous to the closed surfaces (see Bonahon \cite{Bonahon-book}).

A geodesic on the universal covering $\tilde{X}$ is uniquely determined by its two  ideal endpoints on $\partial_{\infty}\tilde{X}$. In Section 3, we identified the space of unoriented geodesics $G(\tilde{X})$ of the universal covering with $(\partial_{\infty}\tilde{X}\times\partial_{\infty}\tilde{X})/(\mathbb{Z}/2\mathbb{Z})$, where the action of $\mathbb{Z}/2\mathbb{Z}$ sends $(a,b)\in \partial_{\infty}\tilde{X}\times\partial_{\infty}\tilde{X}$ to $(b,a)$. 

\begin{definition}
A {\it geodesic current} $\mu$ on $X$ is a Radon measure $\tilde{\mu}$ on the space $G(\tilde{X})$ of unoriented  geodesics on the universal covering $\tilde{X}$ of $X$ that is invariant under the action of the covering group $\pi_1(X)$. Here a Radon measure is a Borel measure that satisfies $\tilde{\mu} (K)<\infty$ for any compact $K\subset G(\tilde{X})$.
\end{definition}

The space of geodesic currents carries a natural weak* topology(see \cite{BonahonSaric}). Namely, let $\xi :G(\tilde{X})\to\mathbb{R}$ be a continuous function with compact support.  The semi-norm $|\cdot |_{\xi}$ induced by the function $\xi$ is
$$
|\tilde{\mu} |_{\xi}=|\int_{G(\tilde{X})}\xi d\tilde{\mu}|
$$
for any geodesic current $\mu$. The {\it weak* topology} on the geodesic currents of $X$ is induced by the family of semi-norms $|\cdot |_{\xi}$, where $\xi$ runs through all continuous functions with compact support. We also note that the weak* topology on the space of geodesic currents is metrizable (see \cite{BonahonSaric}).

Denote by $\mathbb{H}(\tilde{X})$ the set of all isometries of $\tilde{X}$ for the hyperbolic metric induced from $X$. Given a continuous function $\xi :G(\tilde{X})\to\mathbb{R}$ with compact support, we define another semi-norm $\|\cdot\|_{\xi}$ by
$$
\|\tilde{\mu} \|_{\xi}=\sup_{\varphi\in\mathbb{H}(\tilde{X})}|\int_{G(\tilde{X})}\xi\circ\varphi d\tilde{\mu}|
$$
for any geodesic current $\mu$ (see \cite{BonahonSaric}). The space of {\it bounded geodesic currents} $\mathcal{C}_{b}(X)$ consists of all geodesic currents $\mu$ on $X$ such that $\|\tilde{\mu}\|_{\xi}<\infty$ for all continuous $\xi :G(\tilde{X})\to\mathbb{R}$ with compact support (see \cite{BonahonSaric}).  
The {\it uniform weak* topology} on $\mathcal{C}_{b}(X)$ is induced by the family of seminorms $\|\cdot \|_{\xi}$, where $\xi$ runs through all continuous functions with compact support. Note that the uniform weak* topology  is metrizable (see \cite{BonahonSaric}).

\begin{definition}
A {\it measured lamination} $\mu$ on $X$ {\it with support the geodesic lamination} ${\lambda}$ on $X$ is a geodesic current $\tilde{\mu}$ whose support is the lift $\tilde{\lambda}$ of  ${\lambda}$. A measured lamination $\mu$ with support ${\lambda}$ is {\it weakly carried} by $\Theta$ if the geodesic lamination $\tilde{\lambda}$ is weakly carried by $\tilde{\Theta}$. 
\end{definition}

The set $G(\tilde{\Theta})$ consists of all geodesics of $\tilde{X}$ that are weakly carried by $\tilde{\Theta}$.  
Given a finite edge path ${\gamma}$ of $\tilde{\Theta}$, we denote by $G({\gamma})$ the set of geodesics in $G(\tilde{\Theta})$ whose bi-infinite edge paths contain ${\gamma}$ as a subpath. For an edge ${e}\in E(\tilde{\Theta})$, denote by $G({e})$ the set of geodesics in $G(\tilde{\Theta})$ that contain ${e}$ in their bi-infinite edge paths on $\tilde{\Theta}$. 
If ${e}$ is an edge in ${\gamma }$ 
then 
$G({\gamma})\subset G({e})$. 

The following two propositions are technical tools needed in the rest of the paper. Before stating the propositions we define a special type of subsets of $G(\tilde{X})$.

\begin{definition}
Let $I$ and $J$ be two disjoint closed subarcs of $\partial_{\infty}\tilde{X}$. A {\it box of geodesics} $I\times J$ is the set of unoriented geodesics of $G(\tilde{X})$ with one endpoint in $I$ and the other endpoint in $J$.
\end{definition}

\begin{proposition}
\label{prop:carrying_box_connector}
Let $\gamma$ be a finite edge path in $\tilde{\Theta}$ that has only its initial and end vertex on lifts of cuffs $\tilde{\alpha}$ and $\tilde{\beta}$. Then there exist two boxes of geodesics $Q'=I'\times J'$ and $Q=I\times J$ such that $G(\gamma )=G(\tilde{\Theta})\cap Q'=G(\tilde{\Theta})\cap Q$ and $Q'$ is contained in the interior $Q^{\circ}$ of $Q$.

Let $I_1$ and $I_2$ be two components of $I\setminus I'$, and let $J_1$ and $J_2$ be the two components of $J\setminus J'$.  
If the lengths of the cuffs of the pants decomposition of $X$ are between $1/M$ and $M$, 
and the number of edges in $\gamma$ is at most $d>0$, then there exists $m=m(d,M)>0$ and a choice of $Q$ and $Q'$ such that the Liouville measure of each of the boxes of geodesics $I_1\times J$, $I_2\times J$, $I\times J_1$ and $I\times J_2$ is between $1/m$ and $m$.
\end{proposition}

\begin{proof}
The lifts of cuffs $\tilde{\alpha}$ and $\tilde{\beta}$ are on the boundaries of two disjoint half-planes. Let $I'\subset\tilde{X}$ be the closed interval which is  the ideal boundary of the half-plane bounded by $\tilde{\alpha}$ and let $J'$ be the closed interval which is the ideal boundary of the half-plane bounded by $\tilde{\beta}$ (see Figure 5). 

Let $v_1$ be the vertex of $\gamma$ on $\tilde{\alpha}$ and $v_2$ the vertex of $\gamma$ on $\tilde{\beta}$. Let $v_1'$ and $v_1''$ be two vertices of $\tilde{\Theta}$ on $\tilde{\alpha}$ on the opposite sides of $v_1$ that are adjacent to $v_1$.  Let $v_2'$ and $v_2''$ be two vertices of $\tilde{\Theta}$ on $\tilde{\beta}$ on the opposite sides of $v_2$ that are adjacent to $v_2$. Let $\tilde{\alpha}'$ be a lift of a cuff between $\tilde{\alpha}$ and $\tilde{\beta}$ that is connected to $v_1'$ by a finite edge path not crossing any lift of a cuff. Similarly let $\tilde{\alpha}''$ be a lift of a cuff connected to $v_1''$ by a finite edge path not crossing any lift of a cuff and being in between $\tilde{\alpha}$ and $\tilde{\beta}$. In a similar fashion we define lifts of cuff $\tilde{\beta}'$ and $\tilde{\beta}''$ that are connected to $v_2'$ and $v_2''$ (see Figure 5). 

\begin{figure}[h]
\leavevmode \SetLabels
\L(.16*.4) $I$\\
\L(.24*.6) $I'$\\
\L(.72*.3) $J'$\\
\L(.82*.6) $J$\\
\L(.337*.5) $v_1$\\
\L(.325*.38) $v_1''$\\
\L(.326*.625) $v_1'$\\
\L(.315*.7) $\tilde{\alpha}$\\
\L(.4*.82) $\tilde{\alpha}'$\\
\L(.39*.2) $\tilde{\alpha}''$\\
\L(.62*.48) $v_2$\\
\L(.635*.31) $v_2''$\\
\L(.64*.6) $v_2'$\\
\L(.58*.2) $\tilde{\beta}''$\\
\L(.567*.8) $\tilde{\beta}'$\\
\L(.67*.656) $\tilde{\beta}$\\
\L(.49*.57) $\gamma$\\
\endSetLabels
\begin{center}
\AffixLabels{\centerline{\epsfig{file =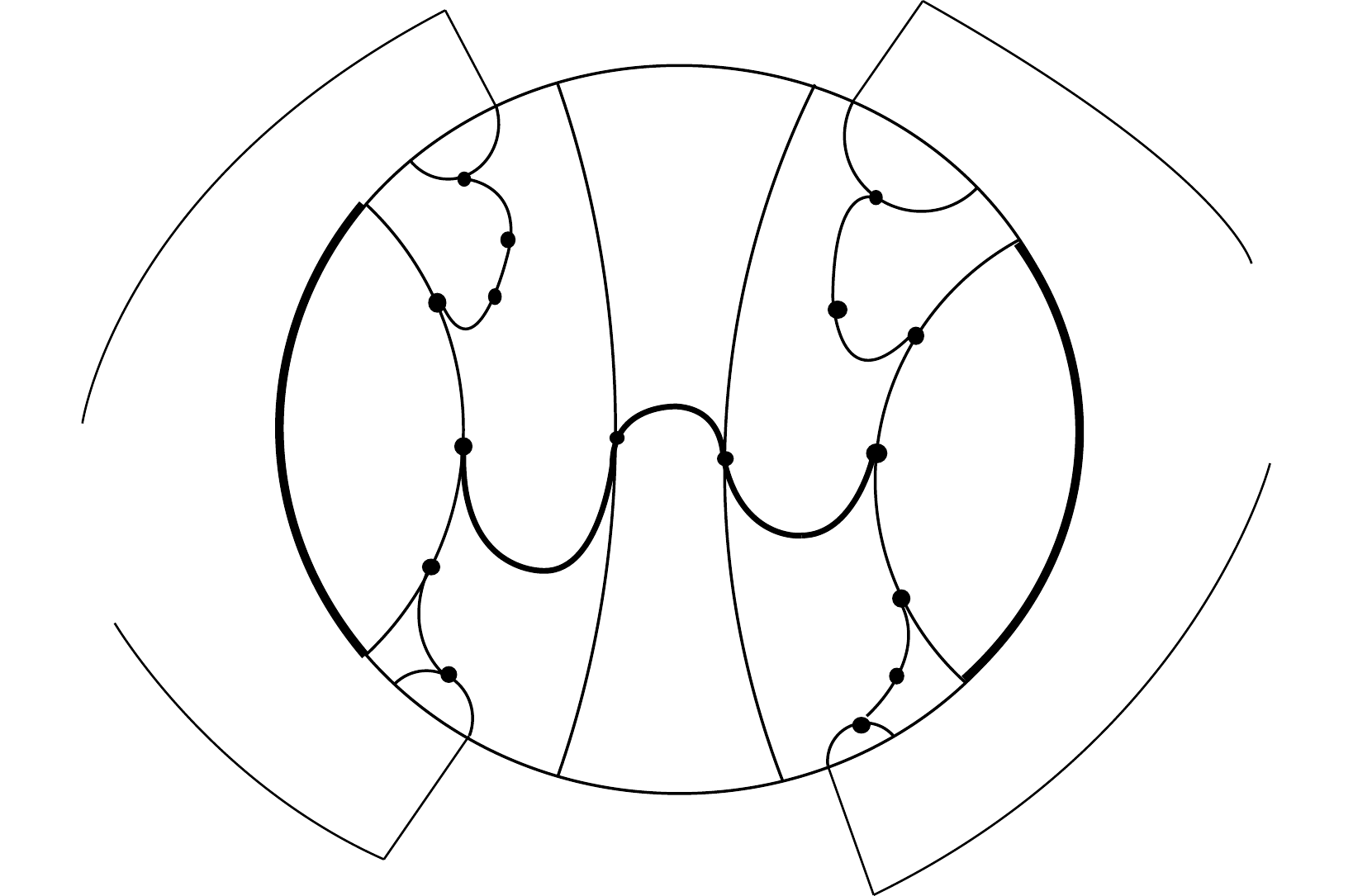,width=12.0cm,angle=0} }}
\vspace{-20pt}
\end{center}
\caption{The set $G(\gamma )$ inside a pair of nested boxes of geodesics when $\gamma$ connects two lifts of cuffs.} 
\end{figure}
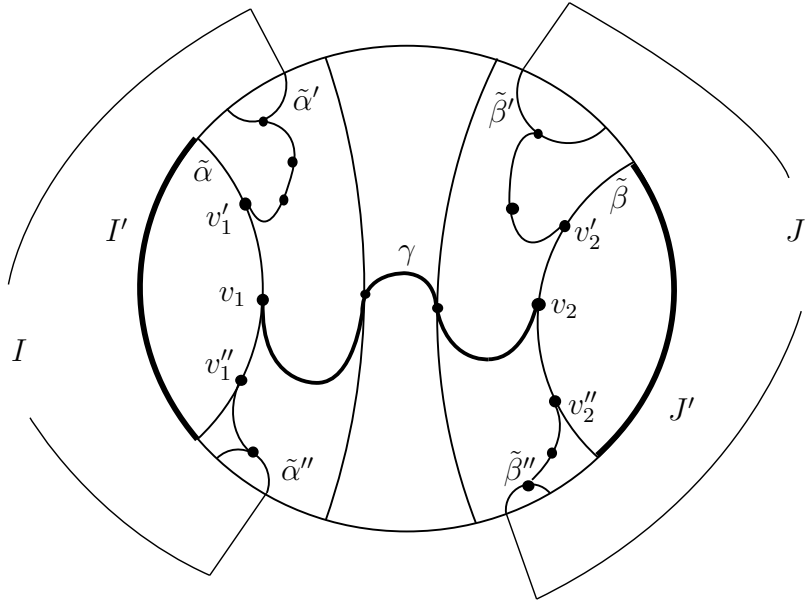

We choose $I\supset I'$ to be the smallest closed interval on $\partial_{\infty}\tilde{X}$ that contains endpoints of $\tilde{\alpha}'$ and $\tilde{\alpha}''$ and we choose $J\supset J'$ to be the smallest closed interval on $\partial_{\infty}\tilde{X}$ that contains endpoints of $\tilde{\beta}'$ and $\tilde{\beta}''$. The obtained boxes of geodesics $Q=I\times J$ and $Q'=I'\times J'$ satisfy
$G(\gamma )=G(\tilde{\Theta})\cap Q'=G(\tilde{\Theta})\cap Q$ by the construction.

Assume that the lengths of the cuffs of the pants decomposition of $X$  are in between $1/M$ and $M$. Then the distance between adjacent lifts of cuffs and the lengths of cuffs edges of $\tilde{\Theta}$ are bounded between $1/M_1$ and $M_1$. Thus the distance between $\tilde{\alpha}$ and $\tilde{\beta}$ is at most $dM_1$. Further, the distance between $\tilde{\alpha}'$ and $\tilde{\beta}$ is at most $(d+6)M_1$. The same bound bound holds for the distance between $\tilde{\alpha}''$ and $\tilde{\beta}$, the distance between $\tilde{\beta}'$ and $\tilde{\alpha}$, and the distance between $\tilde{\beta}''$ and $\tilde{\alpha}$. This gives the bounds on the Liouville measures of $I_1\times J$, $I_2\times J$, $I\times J_1$ and $I\times J_2$. 
\end{proof}

\begin{proposition}
\label{prop:carrying_box_cuff}
Let $\gamma$ be a finite edge path in $\tilde{\Theta}$ that lies on a single lift of a cuff $\tilde{\alpha}$. Then there exist two boxes of geodesics $Q'=I'\times J'$ and $Q=I\times J$ such that $G(\gamma' )=G(\tilde{\Theta})\cap Q'=G(\tilde{\Theta})\cap Q$ and $Q'$ is contained in the interior of $Q$.

Let $I_1$ and $I_2$ be two components of $I\setminus I'$, and let $J_1$ and $J_2$ be the two components of $J\setminus J'$.  
If the lengths of the cuffs of $X$ are between $1/M$ and $M$, 
and the distance between the end vertices of $\gamma'$ is at most $d>0$, then there exists $m=m(d,M)>0$ and a choice of $Q$ and $Q'$ such that the Liouville measure of each of the boxes of geodesics $I_1\times J$, $I_2\times J$, $I\times J_1$ and $I\times J_2$ is between $1/m$ and $m$.
\end{proposition}

\begin{proof} 
Assume first that the connector edges are right tangent to $\alpha$. 
Fix an arbitrary orientation of $\tilde{\alpha}$ and assume that $\gamma$ is given the induced orientation. 
Let $v_1$ and $v_2$ be the initial and end vertex of $\gamma$. There is finitely many lifts of cuffs connected to $\tilde{\alpha}$ by a finite edge path with initial vertex $v_1$ that does not cross any lifts of cuffs and is on the right of $\tilde{\alpha}$ for the fixed orientation. Denote by $\tilde{\beta}_1'$ the lift of the cuff that is farthest away from the initial point of $\tilde{\alpha}$ out of the above finite set of lifts of cuffs (see Figure 6). 
The interval $I'$ is the smallest interval that contains the initial point of $\tilde{\alpha}$ and the endpoints of $\tilde{\beta}_1'$ but does not contain the end point of $\tilde{\alpha}$. We define $\tilde{\beta}_2'$ with respect to $v_2$ on the left side of $\tilde{\alpha}$ in analogous manner (see Figure 6). The interval $J'$ is defined using the end point of $\tilde{\alpha}$ and the endpoints of $\tilde{\beta}_2'$ analogously.

\begin{figure}[h]
\leavevmode \SetLabels
\L(.2*.9) $I$\\
\L(.38*.91) $I'$\\
\L(.76*.08) $J$\\
\L(.55*.1) $J'$\\
\L(.52*.5) $\gamma$\\
\L(.505*.63) $v_1$\\
\L(.45*.4) $v_2$\\
\L(.32*.8) $\tilde{\beta}_1'$\\
\L(.205*.627) $\tilde{\beta}_1$\\
\L(.7*.3) $\tilde{\beta}_2$\\
\L(.55*.17) $\tilde{\beta}_2'$\\
\L(.5*.8) $\tilde{\alpha}$\\
\L(.68*.74) $\tilde{\beta}_1''$\\
\L(.33*.22) $\beta_1''$\\
\endSetLabels
\begin{center}
\AffixLabels{\centerline{\epsfig{file =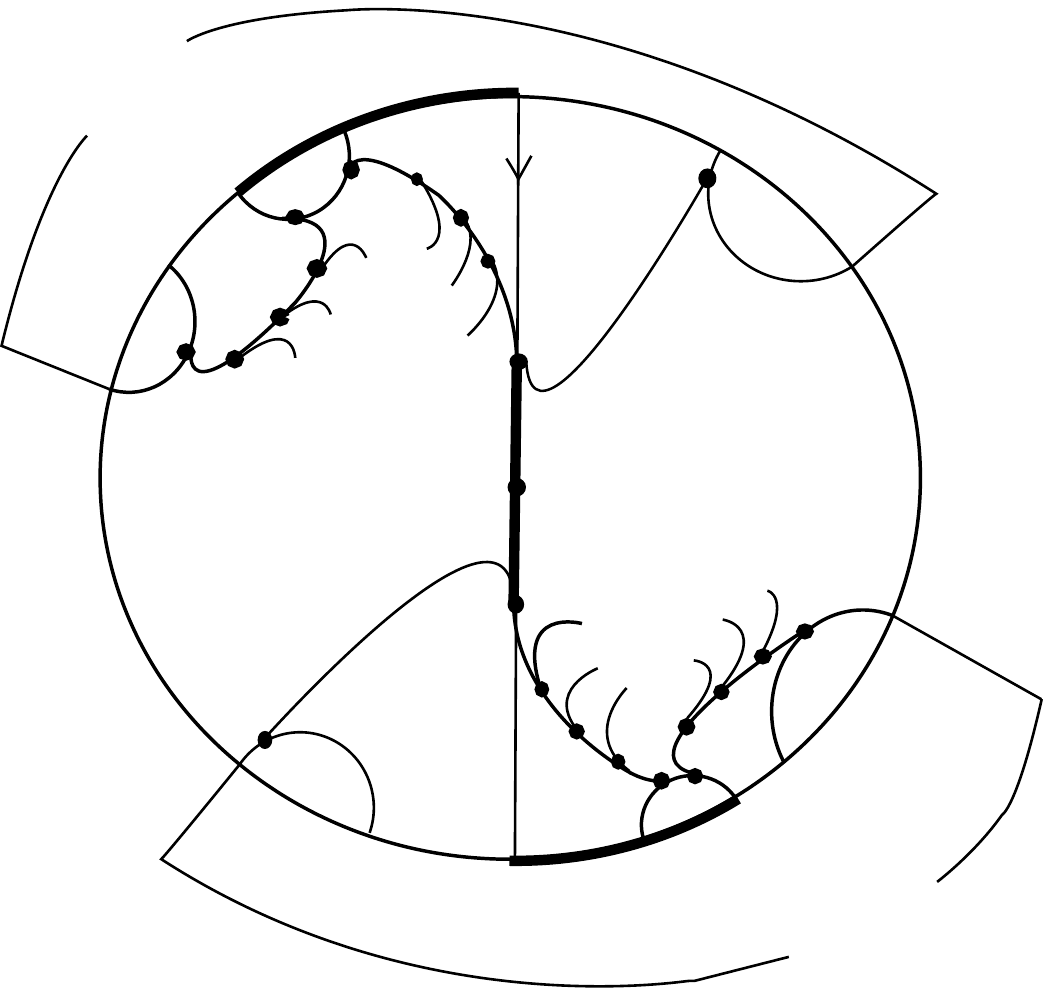,width=12.0cm,angle=0} }}
\vspace{-20pt}
\end{center}
\caption{The set $G(\gamma )$ inside a pair of nested boxes of geodesics when $\gamma$ consists of cuff edges.} 
\end{figure}
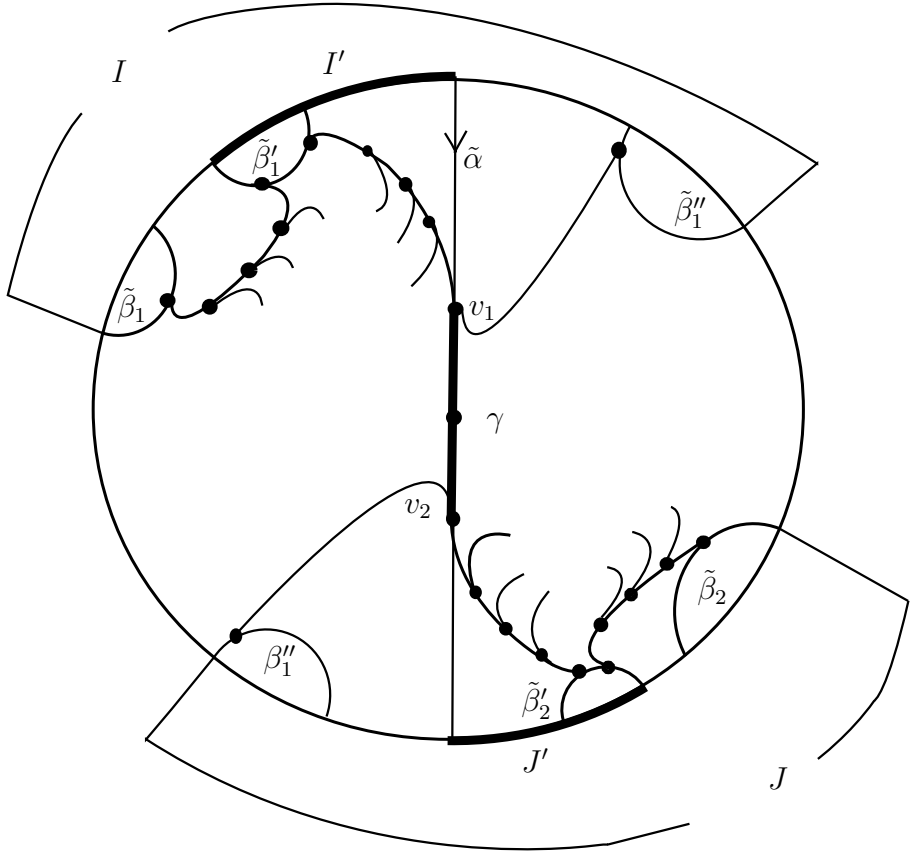

There is countably many lifts of cuffs connected to $\tilde{\beta}_1'$ by a finite edge path not crossing other lifts of cuffs that have endpoints in $\partial_{\infty}\tilde{X}\setminus (I'\cup J')$. Let $\tilde{\beta}_1$ be the lift of a cuff whose endpoints are farthest from $I'$.  

There is countably many lifts of cuffs connected to $\tilde{\beta}_2'$ by a finite edge path not crossing other lifts of cuffs that have endpoints in $\partial_{\infty}\tilde{X}\setminus (I'\cup J')$. Let $\tilde{\beta}_2$ be the lift of a cuff whose endpoints are farthest from $J'$.

Let $\beta_1''$ be a lift of a cuff on the left side of $\tilde{\alpha}$ that is connected to it by a finite edge path not crossing a lift of a cuff and has a vertex $v_1$. Let $\beta_2''$ be a lift of a cuff on the right side of $\tilde{\alpha}$ that is connected to it by a finite edge path not crossing a lift of a cuff and has a vertex $v_2$. 

Let $I$ be the smallest interval containing $I'$ and endpoints of $\tilde{\beta}_1$ and $\tilde{\beta}_1''$ but not containing $J'$. Let $J$ be the smallest interval containing $J'$ and endpoints of $\tilde{\beta}_2$ and $\tilde{\beta}_2''$ but not containing $J'$. (see Figure 6). 

Then $I'\times J'\subset I\times J$ satisfy the desired properties with proofs similar to the previous proposition. Finally, if the connector edges are left tangent the proof is analogous.
\end{proof}

Let $\mu$ be a measured lamination on $X$ that is weakly carried by $\Theta$ and $\tilde{\mu}$ its lift to a geodesic current on $\tilde{X}$. 
Let $\bar{e}$ be a connector edge on $\Theta$ and  ${e}$  a single connector  edge of $\tilde{\Theta}$ that projects to $\bar{e}$. 
Then there exits a finite set of finite connector edge paths $\{ \gamma_1,\ldots ,\gamma_j\}$ such that each $\gamma_i$ contains ${e}$, connects two lifts of cuffs without crossing any lifts of cuffs and
$$
G({e})=\cup_{i=1}^jG(\gamma_i).
$$
By Proposition \ref{prop:carrying_box_connector} and additivity of $\tilde{\mu}$ we have that
$$
\tilde{\mu}(G({e}))<\infty .
$$
When ${e}$ is a cuff edge we obtain $\tilde{\mu}(G({e}))<\infty$ directly from Proposition \ref{prop:carrying_box_cuff}.

We define $\tilde{\mu} ({\gamma} )$ to be the quantity $\tilde{\mu}(G({\gamma}))$. 
By the invariance of $\tilde{\mu}$ under the action of $\pi_1(X)$, this gives a well-defined number $\mu (\bar{\gamma} )$ where $\bar{\gamma}$ is the projection of ${\gamma}$ on $\Theta$.
If $\bar{e}$ is an edge of ${\Theta}$ then we define ${\mu} (e):=\tilde{\mu}(G({e}))$ where 
${e}\in E(\tilde{\Theta})$ projects to $\bar{e}\in E(\Theta )$. 

Fix a vertex $\bar{v}$ of the train track $\Theta$. An edge with a vertex $\bar{v}$ is given the orientation such that $\bar{v}$ is its end point. If an edge has both of its vertices equal to $\bar{v}$ then we have two copies of $\bar{e}$ with the opposite orientations.
We divide the oriented edges of ${\Theta}$ with a common  vertex $\bar{v}$ into two sets $\bar{e}_i$ for $i=1,2,\ldots n$ and $\bar{e}_j'$ for $j=1,2,\ldots ,k$ such that edges in the same set have equal unit tangent vectors at $\bar{v}$.
A {\it switch relation} at $\bar{v}$ is given
$$
\sum _{i=1}^n{\mu} (\bar{e}_i)=\sum_{j=1}^k{\mu} (\bar{e}_j')
$$
and we have a switch relation at every vertex of ${\Theta}$ (see Bonahon \cite{Bonahon-book}). A corresponding switch relation holds true for the lift $\tilde{\mu}$ to the universal covering $\tilde{X}$ of the measured lamination $\mu$. In fact, the values $\tilde{\mu}({e})$ are invariant under the action of $\pi_1(X)$.

A function $a:E({\Theta} )\to [0,\infty )$ which satisfies the {\it switch relation}
$$
\sum _{i=1}^na (\bar{e}_i)=\sum_{j=1}^ka (\bar{e}_j')
$$
 at each switch of ${\Theta}$ is called an {\it edge weight system}. The set of all edge weight systems is denoted by $\mathcal{W}({\Theta} ,[0,\infty ))$. Similarly each edge weight system on $E(\Theta )$  can be lifted to an edge weight system on $E(\tilde{\Theta})$. 
 
 Given a measured lamination ${\mu}$ that is weakly carried by ${\Theta}$, we obtained an edge weight system $f_{{\mu}}:E({\Theta} )\to\mathbb{R}$ given by $f_{{\mu}}(\bar{e}):={\mu} (\bar{e})$. We prove the converse.

\begin{theorem}
Let $a:E({\Theta} )\to\mathbb{R}_{\geq 0}$ be an edge weight system. Then there exists a unique measured lamination ${\mu}$ on ${X}$ which realizes $a$.
\end{theorem}

\begin{proof}
Analogous to the closed case (see Bonahon \cite[Section 3.3, page 56]{Bonahon-book}). We choose a regular neighborhood of $\Theta$ that has the same homotopy type and obtain a fattened train track $\Psi$ corresponding to $\Theta$. The edges of $\Theta$ correspond to long rectangle in $\Psi$ and the vertices of $\Theta$ correspond to the union of several short sides of rectangles that are  connected (see Bonahon \cite{Bonahon-book}). We foliate the long rectangles defining the edges of $\Psi$  by arcs connecting short sides and identify the rectangles with Euclidean rectangles such that leaves of the foliations correspond to the horizontal lines in the Euclidean rectangles and the widths of the Euclidean rectangles are given by the value of the edge weight system $a$ on these edges of $\Theta$. Then each leaf of the obtained foliation of $\Psi$ naturally corresponds to a bi-infinite edge path in $\Theta$ which is naturally identified with a simple geodesic of $X$ by Proposition \ref{prop:geodesic_laminations}. In this way we obtain a geodesic lamination $\lambda$ of $X$ corresponding to an edge weight system $a$ and we push-forward the Euclidean transverse measure of the foliation of $\Psi$ to a measured lamination on $X$ with support $\lambda$. 
\end{proof}

\begin{remark}
The measured laminations considered above are not necessarily bounded. In general, they are only locally bounded--i.e., the measure of any compact subset of $G(\tilde{X})$ is finite. We obtained a one to one correspondence between the space of measured laminations $ML(\Theta ,X)$ whose support geodesics are weakly carried by $\Theta$ and the edge weight systems on $E(\Theta )$. 
\end{remark}

\begin{remark}
In Proposition \ref{prop:carrying} we obtained that every geodesic lamination (without leaves spiraling to cuffs of the fixed pants decomposition) is weakly carried by some pants train track $\Theta$. However we do not claim that a neighborhood of a measured lamination $\mu$ consists of all measured laminations that are weakly carried by a single pants train track. This is true for closed surfaces (see Bonahon \cite{Bonahon-book}) but it is not true for infinite surfaces. 
Indeed, if the support $\lambda$ of a measured lamination $\mu$ consists of a single closed geodesic then it can intersect at most finitely many cuffs of a locally finite pants decomposition. Then any pants train track constructed from the pants decomposition will have multiple choices of standard train tracks in the pants not intersecting the support of $\lambda$. The neighborhood of $\mu$  consists of measured laminations whose support geodesics induce different choices of standard train tracks in the pairs of pants not intersecting the support $\lambda$. Thus no single train track weakly carries a neighborhood of $\mu$.
\end{remark}

\section{The edge weight systems and weak* topology on measured laminations}

The family of all edge weight systems is a subset of $\mathbb{R}^{E(\tilde{\Theta})}$ which can be equipped with the product topology.
In this section we prove that the correspondence between the edge weight systems and measured laminations weakly carried by the train track is a homeomorphism, when the topology on the edge weight systems is given by the restriction of the product topology and the topology on the measured laminations is the weak* topology. This result holds for compact surfaces as well, however the proof for infinite surfaces is more involved and it also leads to the extension of this result for the uniform weak* topology in Section 9. The results in this section and in the next three sections is the core of the paper.

We first state a classical fact about the weak* convergence of measures which will help us to work more efficiently when proving various convergence statements.

\begin{lemma} \cite[chap. IV, §5, no 12]{Bourbaki}
\label{lem:box_conv}
Let $\tilde{\mu}_n$ be a sequence of geodesic currents on $\tilde{X}$ that converges to $\tilde{\mu}$ in the weak* topology. Then for any measurable set $B$ that satisfies 
$\tilde{\mu} (\delta B)=0$, where $\delta B$ is the topological boundary, we have
$$
\lim_{n\to\infty}\tilde{\mu}_n(B)=\tilde{\mu} (B).
$$
\end{lemma}

Recall a standard fact that a sequence of measures on a compact set whose total mass is bounded is weak* compact (see Bourbaki \cite[chap. III, Section 1, no 9]{Bourbaki}). We will often use an elementary consequence of that fact given by Lemma \ref{lem:weak*compact}. 

\begin{lemma}
\label{lem:weak*compact}
Let $\tilde{\mu}_n$ be a sequence of measured laminations on the universal covering $\tilde{X}$ that are lifts of a sequence of measured laminations $\mu_n$ on $X$ . If
$$
\sup_n\tilde{\mu}_n(K)<\infty
$$
for each compact set $K\subset G(\tilde{X})$ then the sequence $\{\tilde{\mu}_n\}_n$ has a subsequence that converges to a measured lamination $\tilde{\mu}$ in the weak* topology which projects to a measured lamination $\mu$ on $X$.
\end{lemma}

\begin{proof}
Let $\{K_j\}_j$ be a compact exhaustion of the space of geodesics $G(\tilde{X})$. By Bourbaki \cite[chap. III, Section 1, no 9]{Bourbaki} and the assumption of the lemma, the sequence $\{\tilde{\mu}_n\}_n$ is weak* compact on each $K_j$. Using a Cantor diagonal process, we find a subsequence $\{\tilde{\mu}_{n_k}\}_k$ that is weak* convergent on each $K_j$. The limit $\tilde{\mu}$ is a geodesic current on $\tilde{X}$. 

To see that $\tilde{\mu}$ is a measured lamination, assume on the contrary that two geodesics $g_1$ and $g_2$ of the support of $\tilde{\mu}$ intersect in $\tilde{X}$. Then there exist two boxes of geodesics $I_1\times I_2$ and $J_1\times J_2$ that contain $g_1$ and $g_2$ in their interiors   and each geodesic of one box intersect each geodesic of the other box. By the countable additivity of $\tilde{\mu}$ we can also arrange that $\tilde{\mu}(\delta (I_1\times I_2))= \tilde{\mu}(\delta (J_1\times J_2))=0$. Since $\tilde{\mu}$-measure of the boundary is zero, the weak* convergence implies that $\tilde{\mu}_n( I_1\times I_2)\to \tilde{\mu}( I_1\times I_2)$ and $\tilde{\mu}_n( J_1\times J_2)\to \tilde{\mu}( J_1\times J_2)$ as $n\to\infty$ (see Lemma \ref{lem:box_conv}). Then $\tilde{\mu}_n(I_1\times I_2)\neq 0$ and $\tilde{\mu}_n(J_1\times J_2)\neq 0$ for $n$ large enough which contradicts the assumption that $\tilde{\mu}_n$ is a measured lamination. 

Thus the support of $\tilde{\mu}$ is a geodesic lamination and $\tilde{\mu}$ is a measured lamination on $\tilde{X}$. Since $\tilde{\mu}_n$ are invariant under the action of $\pi_1(X)$ so is $\tilde{\mu}$. Therefore $\tilde{\mu}$ projects to a measured lamination on $X$. 
\end{proof}

Since the weak* topology and the product topology on $\mathbb{R}^{E(\tilde{\Theta})}$ are metrizable to prove the continuity it is enough to prove the sequential continuity. 
In order to do so, we will need the following lemma
that proves  convergence of measured laminations on special subsets of $G(\tilde{X})$ if  their edge weight systems converge.

\begin{lemma}
\label{lem:convergence_paths}
Let $\tilde{\mu}_n,\tilde{\mu}$ be measured laminations on $\tilde{X}$ weakly carried by the train track $\tilde{\Theta}$ such that  $f_{\tilde{\mu}_n}(e)\to f_{\tilde{\mu}}(e)$ as $n\to\infty$ for each $e\in E(\tilde{\Theta})$. If ${\gamma}$ is a finite edge path in $\tilde{\Theta}$, then
$$
\tilde{\mu}_n(G({\gamma}))\to\tilde{\mu}(G({\gamma}))
$$
as $n\to\infty$.
\end{lemma} 

\begin{proof}
For a fixed ${\gamma}$, there is a unique piecewise linear function $\mathcal{L}_{{\gamma}}$
such that the quantities $\tilde{\mu}_n(G({\gamma}))$ and $\tilde{\mu}(G({\gamma}))$ are obtained by evaluating $\mathcal{L}_{{\gamma}}$  on the edge weight systems $f_{\tilde{\mu}_n}$ and $f_{\tilde{\mu}}$ for the edges of $\tilde{\Theta}$ that are in ${\gamma}$ or that have a vertex in common with ${\gamma}$ (see Bonahon \cite[Section 3.2, page 53, the second lemma]{Bonahon-book}). Since $f_{\tilde{\mu}_n}(e)\to f_{\tilde{\mu}}(e)$ for all $e\in E(\tilde{\Theta})$ as $n\to\infty$, the lemma follows.
\end{proof}

In the following lemma we show that each box of geodesics $I\times J$ can contain at most finitely many lifts of cuffs. 

\begin{lemma}
\label{lem:cuffs_in_box}
If $Q=I\times J$ is a box of geodesics in $G(\tilde{X})$ then the number of lifts of cuffs in $Q$ is finite.
\end{lemma}

\begin{proof}
Indeed, since the pants decomposition is locally finite and $Q$ is a compact set in $G(\tilde{X})$ there can be only finitely many lifts of cuffs in $Q$.
\end{proof}

The following lemma is the key ingredient for proving the weak* convergence from the convergence of the edge weight systems.

\begin{lemma}
\label{lem:conv_general_box}
Let $\tilde{\mu},\tilde{\mu}_n$ be measured laminations on $\tilde{X}$ weakly carried by $\tilde{\Theta}$ that are lifts of measured laminations $\mu ,\mu_n$ on $X$. Let $Q=I\times J$, where $I=[a,b]$ and $J=[c,d]$, be a box of geodesics such that $a,b,c$ and $d$ are endpoints of lifts of cuffs and no lift of a cuff is on the topological boundary $\delta Q$ of $Q$. If $f_{\tilde{\mu}_n}(e)\to f_{\tilde{\mu}}(e)$ as $n\to\infty$ for each $e\in E(\tilde{\Theta})$ then
$$
\tilde{\mu}_n(Q)\to \tilde{\mu}(Q)
$$
as $n\to\infty$.
\end{lemma}

\begin{proof}
Note that $\tilde{\mu} (\delta Q)=0$ and $\tilde{\mu}_n (\delta Q)=0$ because a geodesic in $\partial Q$ that is weakly carried by $\tilde{\Theta}$ has at least one endpoint in common with a lift of a cuff but it is not a lift of a cuff. Therefore this geodesic projects to a geodesic on $X$ that spirals around a closed geodesic and it cannot be in the support of any measured lamination on $X$.

Assume on the contrary that there exists an infinite subsequence of $\tilde{\mu}_n$, which is for the simplicity of the notation denoted by $\tilde{\mu}_n$, such that
\begin{equation}
\label{eq:m-dist}
|\tilde{\mu}_n(Q)-\tilde{\mu}(Q)|\geq c>0.
\end{equation}

Let $Q_1=[a_1,b_1]\times [c_1,d_1]$ be an arbitrary box of geodesics in $G(\tilde{X})$. We first prove that  $Q_1 \cap G(\tilde{\Theta})$ can be covered by finitely many sets $G(e)$ for $e\in E(\tilde{\Theta})$. 
Since $\pi_1(X)$ is of the first kind it follows that there exist two lifts of cuffs: $\tilde{\alpha}_1$, 
 whose ideal boundary points are in $(b_1,c_1)$, and $\tilde{\alpha}_2$,
 whose ideal endpoints are in $(d_1,a_1)$. If a bi-infinite edge path in $\tilde{\Theta}$ intersects either $\tilde{\alpha}_1$ or $\tilde{\alpha}_2$ then its endpoints cannot be in $Q_1$ by the no backtracking property of edge paths in $\tilde{\Theta}$ (see Lemma \ref{lem:crossing}).
 Thus any geodesic in $Q_1\cap G(\tilde{\Theta})$ is represented by a bi-infinite geodesic path that separates $\tilde{\alpha}_1$ and $\tilde{\alpha}_2$. Let $\omega$ be a compact geodesic arc that connects $\tilde{\alpha}_1$ and $\tilde{\alpha}_2$. Since $\tilde{\alpha}_1\cup \omega\cup \tilde{\alpha}_2$ separates $I_1$ and $J_1$, it follows that any bi-infinite edge path representing a geodesic of $Q_1\cap G(\tilde{\Theta})$ intersects $\omega$. Given that the pants decomposition is locally finite and that $\omega$ is compact, it follows that $\omega$ intersects only finitely many edges of $\tilde{\Theta}$. 
Thus $Q_1\cap G(\tilde{\Theta})$ is covered by finitely many $G(e)$, where $e$ is an edge intersecting $\omega$.

Since $f_{\tilde{\mu}_n}(e)\to f_{\tilde{\mu}}(e)$ we have that the sequence $\tilde{\mu}_n$ has uniformly bounded mass on $Q_1$, where the bound depends only on $Q_1$. 
By Lemma \ref{lem:weak*compact}, a subsequence $\tilde{\mu}_{n_k}$ converges to a  measured lamination $\tilde{\nu}$ on $G(\tilde{X})$ and the measured lamination $\tilde{\nu}$ is the lift of a measured lamination $\nu$ on $X$.

Let $Q=I\times J$ be a box of geodesics from (\ref{eq:m-dist}).  
Denote by $\{\tilde{\alpha}_i'\}_i$ the maximal (an at most countable) family of lifts of cuffs with both endpoints in $I$ such that each $\tilde{\alpha}_i'$ is not separated from $J$ by another lift of a cuff that also has both endpoints in $I$. We denote by $\{\tilde{\alpha}_j''\}_j$ the corresponding family of lifts of cuffs for $ J$. 

Consider all pairs $(\tilde{\alpha}'_i,\tilde{\alpha}_j'')$ formed from the above two families 
such that there is a finite edge path $\gamma_{i,j}$ of $\tilde{\Theta}$ that connects them.
By Proposition \ref{prop:carrying_box_connector}, there exists a box of geodesics $Q_{i,j}$ such that $G(\gamma_{i,j})=G(\tilde{\Theta})\cap Q_{i,j}$ and $G(\tilde{\Theta})\cap\delta Q_{i,j}=\emptyset$.

By Lemma \ref{lem:cuffs_in_box} there is at most finitely many lifts of cuffs $\{\tilde{\alpha}^Q_r\}_r$ that are contained in $Q$. For a fixed $\tilde{\alpha}^Q_r$, let $P_r^1$ and $P_r^2$ be two components of lifts of pairs of pants that share a common boundary geodesic $\tilde{\alpha}^Q_r$.
From the set of all lifts of cuffs with both endpoints in $I$ on the boundary of $P_r^1$ that are connected to $\tilde{\alpha}_r^Q$ by finite edge paths, choose $\tilde{\alpha}_1'$ whose endpoints are the farthest from the endpoint of $\tilde{\alpha}_r^Q$ in $I$. Similarly, let $\tilde{\alpha}_1''$ be the lift of a cuff on the boundary of $P_r^1$ with both endpoints in $J$ and connected by a finite edge path to $\tilde{\alpha}_r^Q$ that is the farthest from the endpoint of $\tilde{\alpha}_r^Q$. The choices of $\tilde{\alpha}_1'$ and $\tilde{\alpha}_1''$ imply that they belongs to the families $\{\tilde{\alpha}'\}$ and $\{\tilde{\alpha}''\}$, respectively. In an analogous manner we choose a lift of cuff $\tilde{\alpha}_2'$ with endpoints in $I$ and on the boundary of $P_r^2$, and a lift of cuff $\tilde{\alpha}_2''$ with endpoints in $J$ and on the boundary of $P_r^2$. We also have that $\tilde{\alpha}_2'\in\{\tilde{\alpha}'_i\}$ and $\tilde{\alpha}_2''\in\{\tilde{\alpha}''_j\}$. 
Out of the set of four possible pairs $\{(\tilde{\alpha}_t',\tilde{\alpha}_s'')\}_{t,s=1}^4$ there is a unique pair  $(\tilde{\alpha}_{i_r}',\tilde{\alpha}_{j_r}'')$ that is connected by a finite edge path $\gamma_r$ of $\tilde{\Theta}$. 

There are two possibilities: either $\gamma_r$ has a subpath $\gamma_r^Q$ of cuff edges on $\tilde{\alpha}$ or it crosses $\tilde{\alpha}$ at a vertex $v$ without having an edge on $\tilde{\alpha}$. 

If $\gamma_r$ has no edges on $\tilde{\alpha}$ then we modify the construction of $\gamma_r^Q$ as follows. Take $\gamma_r^Q$ to consists of a single cuff edge $e_r$ on $\tilde{\alpha}$ with one vertex $v$ and the other vertex $v_1$ such that the lift of a cuff on $P_r^{i_r}$ that is connected to $\tilde{\alpha}$ through the vertex $v_1$ (and farthest from the endpoint of $\tilde{\alpha}$ in $I$) is closer to the endpoint of $\tilde{\alpha}$ in $I$ than the endpoints of $\tilde{\alpha}_{i_r}'$.

We have that $G(\gamma_r^Q)\subset Q$.  By Proposition \ref{prop:carrying_box_cuff} there is a box of geodesics $Q_r$ such that $G(\gamma_{r}^Q)=G(\tilde{\Theta})\cap Q_{r}$, $\tilde{\alpha}_r^Q\in Q_r$ and $G(\tilde{\Theta})\cap\delta Q_{r}=\emptyset$. The box of geodesics $Q_r=I_r\cup J_r$ is chosen such that each $\tilde{\alpha}'_i$ and each $\tilde{\alpha}''_j$ is completely contained in $I_r$ or in $J_r$ or it is disjoint from both $I_r$ and $J_r$. Therefore if $G(\gamma_{i,j})$ (corresponding to a pair $(\tilde{\alpha}'_i,\tilde{\alpha}''_j)$) intersects $Q_r$ then it is contained in $Q_r$ and we erase it from the family of pairs $(\tilde{\alpha}_i',\tilde{\alpha}'')$.

Therefore the family of all geodesics in $ Q=I \times  J $ weakly carried by $\tilde{\Theta}$ is the union of all $G({\gamma}_{i,j})$ and $G(\gamma^Q_r)$. Moreover, by construction $G({\gamma}_{i,j})\cap G({\gamma}_{i_1,j_1})=\emptyset$ for $(i,j)\neq (i_1,j_1)$ and $G({\gamma}_{i,j})\cap G(\gamma^Q_r)=\emptyset$ for all $(i,j)$ and $r$. Thus, the set of geodesics of $Q$ that are weakly carried by $\tilde{\Theta}$ is partitioned into at most countable disjoint union of the sets $G({\gamma}_{i,j})$ and $G(\gamma^Q_r)$.
  
By Lemma \ref{lem:convergence_paths} we have
\begin{equation}
\label{eq:limit_edge-path}
\lim_{k\to\infty}\tilde{\mu}_{n_k}(G(\tilde{\gamma}_{i,j}))=\tilde{\mu}(G(\tilde{\gamma}_{i,j})).
\end{equation}
and
\begin{equation}
\label{eq:limit_edge-path1}
\lim_{k\to\infty}\tilde{\mu}_{n_k}(G(\gamma^Q_r))=\tilde{\mu}(G(\gamma^Q_r)).
\end{equation}
The weak* convergence implies that $\tilde{\mu}_{n_k}(Q_r)\to\tilde{\nu}(Q_r)$ and $\tilde{\mu}_{n_k}(Q_{i,j})\to\tilde{\nu}(Q_{i,j})$ as $k\to\infty$. Since $G(\gamma_{r}^Q)=G(\tilde{\Theta})\cap Q_{r}$ and $G(\gamma_{i,j})=G(\tilde{\Theta})\cap Q_{i,j}$ we have $\tilde{\nu}(Q_r)=\tilde{\nu}(G(\gamma^Q_r))$ and $\tilde{\nu}(Q_{i,j})=\tilde{\nu}(G(\gamma_{i,j}))$. Then (\ref{eq:limit_edge-path}) and (\ref{eq:limit_edge-path1}) gives
\begin{equation}
\label{eq:cuffs}
\tilde{\mu}(G(\gamma^Q_r))=\tilde{\nu}(G(\gamma^Q_r))
\end{equation}
and
\begin{equation}
\label{eq:connectors}
\tilde{\mu}(G(\gamma^Q_{i,j}))=\tilde{\nu}(G(\gamma^Q_{i,j})).
\end{equation}

Since $G(\tilde{\Theta})\cap Q$ is a disjoint countable union of sets $G(\tilde{\gamma}_{i,j})$ and $G(\gamma^Q_r)$ then (\ref{eq:cuffs}) and (\ref{eq:connectors}) imply $\tilde{\mu}(Q)=\tilde{\nu}(Q)$. This is in a contradiction with (\ref{eq:m-dist}) and the lemma is proved.
\end{proof}

Let $ML(\Theta ,X)$ be the space of all measured laminations on $X$ that are weakly carried by $\Theta$. We prove the main theorem of this section.

\begin{theorem}
The bijective correspondence between the space $ML(\Theta ,X)$ of measured laminations weakly carried by $\Theta$ and the space  
$\mathcal{W}(\Theta ,[0,\infty ))$ of  edge weight systems  is a homeomorphism for the weak* topology on $ML(\Theta ,X)$ and the topology of pointwise convergence on $\mathcal{W}(\Theta ,[0,\infty ))$.
\end{theorem}

\begin{proof}
Let $\xi :G(\tilde{X})\to\mathbb{R}$ be a continuous function with compact support. Let $\tilde{\mu}_n$ and $\tilde{\mu}$ be lifts to $\tilde{X}$ of measured laminations on $X$ such that $f_{\tilde{\mu}_n}(e)\to f_{\tilde{\mu}}(e)$ as $n\to\infty$ for each $e\in E(\tilde{\Theta})$. We need to prove that
\begin{equation}
\label{eq:weak*_conv}
\int_{G(\tilde{X})}\xi d[\tilde{\mu}-\tilde{\mu}_n]\to 0
\end{equation}
as $n\to\infty$. By using a partition of unity, we can assume that the support of $\xi$ is contained in a box of geodesics. 

Since the lifts of cuffs are coming from a geodesic pants decomposition of $X$ it follows that the endpoints of the lifts of cuffs are dense in $\partial_{\infty}\tilde{X}$. 
By slightly increasing the size of the support box of $\xi$ we can assume that the four vertices of the box are the endpoints of the lifts of cuffs and that the boundary of the box does not contain lifts of cuffs. 
By the density of the endpoints of the lifts of cuffs in $\partial_{\infty}\tilde{X}$,
the support box is divided into small boxes whose vertices are also of the above type. We approximate the function $\xi$ with a step function whose steps are the boxes of the partition.  Lemma \ref{lem:conv_general_box} implies (\ref{eq:weak*_conv}). We obtained that $\tilde{\mu}_n\to\tilde{\mu}$ in the weak* topology as $n\to\infty$.

We assume now that $\tilde{\mu}_n\to\tilde{\mu}$ in the weak* topology as $n\to\infty$. 
Let $e\in\tilde{\Theta}$. If $e$ is a cuff edge then Proposition \ref{prop:carrying_box_cuff} gives a box of geodesics $Q_e$ such that $G(e)=Q_e\cap G(\tilde{\Theta})$ and $\delta Q_e\cap G(\tilde{\Theta})=\emptyset$. Then $\tilde{\mu}_n(G(e))=\tilde{\mu}_n(Q_e)\to\tilde{\mu}(Q_e)=\tilde{\mu}(G(e))$ by the weak* convergence and Lemma \ref{lem:box_conv}. Thus $
f_{\tilde{\mu}_n}(e)\to f_{\tilde{\mu}}(e)
$ when $e$ is a cuff edge.

If $e$ is a connector edge of $\tilde{\Theta}$ then there exists finitely many finite edge paths 
$\{\gamma_1,\ldots ,\gamma_j\}$ consisting of only connector edges and containing $e$ such that they connect two lifts of cuffs on the lift of a pair of pants containing $e$. Proposition \ref{prop:carrying_box_connector} implies that $\tilde{\mu}_n(G(\gamma_i))\to\tilde{\mu}(G(\gamma_i))$ for all $i$ as $n\to\infty$. Since $G(\gamma_i)\cap G(\gamma_{i_1})=\emptyset$ for all $i\neq i_1$ and $G(e)=\cup_{i=1}^jG(\gamma_i)$, we have that $\tilde{\mu}_n(G(e))\to\tilde{\mu}(G(e))$ which is the same as 
$
f_{\tilde{\mu}_n}(e)\to f_{\tilde{\mu}}(e)
$
as $n\to\infty$.
\end{proof}

\section{Hyperbolic surfaces with bounded pants decomposition}

Throughout this section $X$ is an infinite hyperbolic surface equipped with a fixed locally finite geodesic pants decomposition $\{ P_n\}$ such that the lengths of cuffs are bounded between $1/M$ and $M$ for some $M\geq 1$.  In addition we assume that $X$ has no cusp--i.e.,  each $P_n$ has three cuffs. The case when $X$ has cusps is considered in Section 8. We consider a fixed pants train track $\Theta$ on $X$.

In each pair of pants $P_n$, the train track $\Theta$ has exactly three edge paths that connects pairs of cuffs. We choose three geodesic arcs $o_i^n$, $i\in\{ 1,2,3\}$, with both 
endpoints orthogonal to the cuffs of $P_n$ that connect the three pairs of cuffs that are also connected by $\Theta$.
Then $o_i^n$
 divide $P_n$ into two 
 right angled hexagons. 
Let $\Theta_0$ be the union of cuffs and orthogonal geodesic arcs $o_i^n$ over all $P_n$ in $X$. The lift $\tilde{\Theta}_0$ of $\Theta_0$ to the universal covering $\tilde{X}$ is the union of the boundaries of right-angled hexagons and the hexagons tile $\tilde{X}$.

If the added geodesic arcs $o_i^n$ orthogonal to cuffs of $P_n$ are not pants seams then the two hexagons share the three arcs and consequently they are isometric because by the hexagon formula all sides have equal lengths(see Beardon \cite[page 161, Theorem 7.19.2]{Beardon}). The lengths of the hexagon sides that lie on the cuffs of $P_n$ are equal to half the cuffs lengths and thus are pinched between between $1/(2M)$ and $M/2$.
By the hexagon formula we get that the other three side lengths of the hexagons are pinched between $1/M_1$ and $M_1$ for some $M_1=M_1(M)>1$. Let $m=\max\{ 2M,M_1\}$. Then the lengths of sides of such hexagons are pinched between
 $1/m$ and $m$.

We also estimate the size of hexagons obtained from the pairs of pants when pants seams are used. The following lemma  controls the geometry of the complementary hexagons of $\tilde{\Theta}_0$ when one of the added orthogonal arcs in $P_n$ is a pants seam. 

\begin{lemma}
\label{lem:seams_bound}
Fix $M>1$. Let $P$ be a geodesic pair of pants with the cuffs $\alpha_1,\alpha_2,\alpha_3$ such that 
$$
\frac{1}{M}\leq l(\alpha_i)\leq M
$$
for $i=1,2,3$ where $l(\alpha_i)$ is the hyperbolic length of $\alpha_i$. Let $l_1$ be the length of the shortest geodesic arc connecting $\alpha_1$ to itself and separating $\alpha_2$ and $\alpha_3$. Let $l_2,l_3$ be the lengths of the shortest geodesic arcs connecting $\alpha_1$ to $\alpha_2,\alpha_3$ respectively. 

Then the arcs $l_1,l_2,l_3$ divide $P$ into two right angled hexagons whose side lengths are between $1/M_1$ and $M_1$ for some $M_1>1$ which depends only on $M$. 
\end{lemma}

\begin{proof}
Denote by $l_1'$ the shortest geodesic arc in $P_n$ that connects $\alpha_2$ and $\alpha_3$. Then $l_1'\cup l_2\cup l_3$ divides $P_n$ into two isometric right angled hexagons as above. There is a reflection of $P_n$ in $l_1'\cup l_2\cup l_3$ that isometrically sends one hexagon onto the other. The arc $l_1$ is orthogonal at both of its endpoints to $\alpha_1$ and the reflection of $P_n$ sends $l_1$ onto itself (by the uniqueness of the shortest arc $l_1$ in its homotopy class). It follows that the angle between $l_1'$ and $l_1$ is $\pi /2$ and that $l_1'$ bisects $l_1$ (see Figure 8).

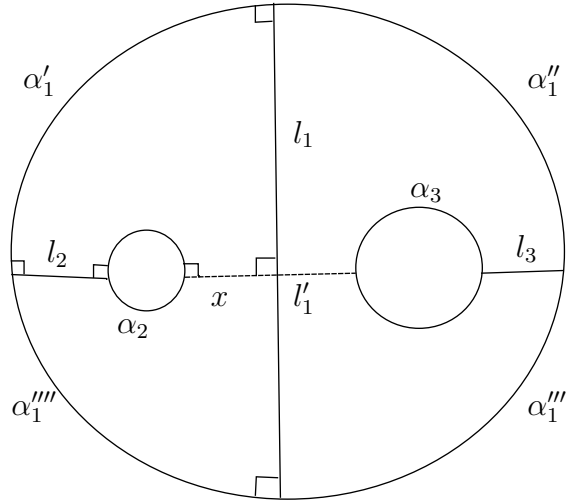
\begin{figure}[h]
\leavevmode \SetLabels
\L(.27*.8) $\alpha_1'$\\
\L(.7*.8) $\alpha_1''$\\
\L(.7*.2) $\alpha_1'''$\\
\L(.26*.2) $\alpha_1''''$\\
\L(.29*.48) $l_2$\\
\L(.35*.35) $\alpha_2$\\
\L(.43*.4) $x$\\
\L(.5*.7) $l_1$\\
\L(.5*.4) $l_1'$\\
\L(.6*.6) $\alpha_3$\\
\L(.69*.49) $l_3$\\
\endSetLabels
\begin{center}
\AffixLabels{\centerline{\epsfig{file =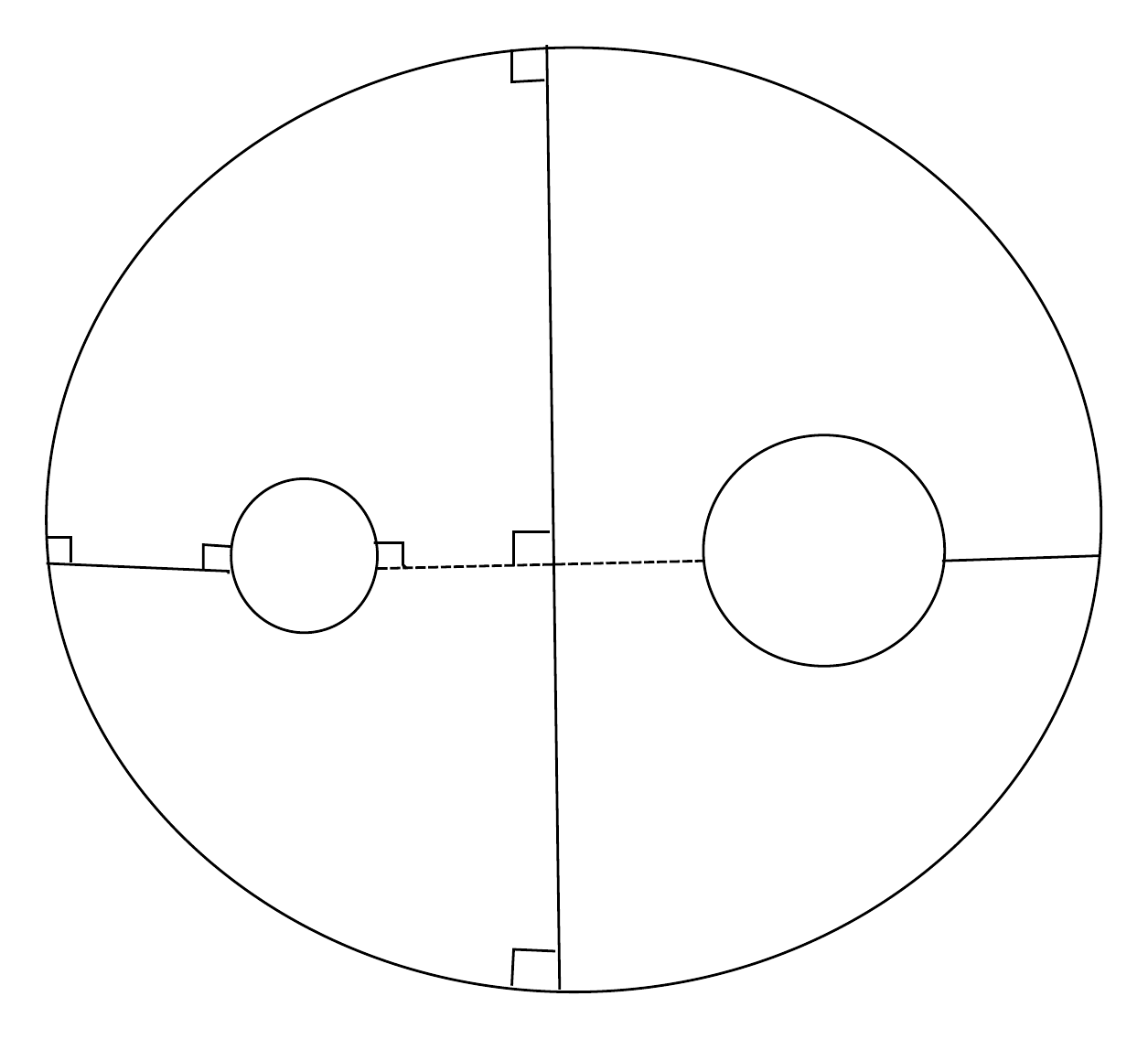,width=8.0cm,angle=0} }}
\vspace{-20pt}
\end{center}
\caption{Bounded geometry hexagons from pants seams. $l_1'$ is the dotted arc.} 
\end{figure}

The endpoints of the arcs $l_1,l_2,l_3$ divide the cuff $\alpha_1$ into four arcs 
$\alpha_1'$, $\alpha_1''$, $\alpha_1'''$ and $\alpha_1''''$. The arc $l_1'$ divides the two symmetric  hexagons into four right angled pentagons (see Figure 8). Consider the pentagon with sides $\alpha_1'$, $l_2$, $\frac{1}{2}l(\alpha_2)$, $x$ and $\frac{1}{2}l_1$, where $x$ is a part of $l_1'$ from $\alpha_2$ to the point $l_1'\cap l_1$. The pentagon formula (see Beardon \cite[page 159, Theorem 7.18.1]{Beardon}) gives
$$
\tanh l_2\cosh \frac{1}{2}l(\alpha_2)\tanh x=1
$$ 
and since $l_2$ and $\frac{1}{2}l(\alpha_2)$ are pinched between two positive constants, it follows that $x$ is also pinched between two positive constants. The pentagon formula also gives
$$
\tanh l(\alpha_1')\cosh l_2\tanh \frac{1}{2}l(\alpha_2)=1
$$ 
which implies that $l(\alpha_1')$ is pinched between two positive constants.

All other sides of the pentagons are similarly pinched between two positive constants.
 Therefore $P_n$ is divided into two right angled hexagons whose side lengths are pinched between two positive constants and the lemma is proved.
\end{proof}

For the definiteness let $M_1>1$ be such that the lengths of the sides of the two hexagons that $P_n$ is divided into are between $1/M_1$ and $M_1$ for all choices of the dividing arcs in $P_n$.   

\begin{lemma}
\label{lem:quasiisometry}
The lift $\tilde{\Theta}_0$ of $\Theta_0$ with the induced path metric coming from the hyperbolic metric on $\tilde{X}$ is quasi-isometric to $\tilde{X}$ under the inclusion map. 
\end{lemma}

\begin{proof}
Since $\tilde{X}$ is tiled by the hexagons whose boundary sides have lengths pinched between two positive constants (by Lemma \ref{lem:seams_bound} and the remark above it) it follows that every point of $\tilde{X}$ is on a bounded distance from $\tilde{\Theta}_0$.

Let $\rho (x,y)$ be the hyperbolic distance between $x,y\in\tilde{X}$. Given $x,y\in\tilde{\Theta}_0$, let $p(x,y)$ be the hyperbolic length of the shortest path on $\tilde{\Theta}_0$ connecting $x$ and $y$. We need to prove that $\frac{1}{A}\rho (x,y)-B\leq p(x,y)\leq A\rho (x,y)+B$ for some constants $A>0$ and $B\geq 0$. It is immediate that $\rho (x,y)\leq p(x,y)$ and it remains to prove the right hand side inequality.

Let $l$ be the hyperbolic geodesic arc in $\tilde{X}$ with endpoints $x$ and $y$. Let $l_i$, $i=1,2,\ldots , n$ be the subarcs of $l$ that are obtained by intersecting $l$ with the hexagon tiling of $\tilde{X}$. Then $\rho (x,y)=\sum_{i=1}^n\rho (l_i)$ where $\rho (l_i)$ is the hyperbolic length of $l_i$. Denote by $\Sigma_i$ the hyperbolic hexagon that contains $l_i$ and let $s,s'$ be the boundary sides that $l_i$ connects. 

If $s$ and $s'$ are adjacent then there are arcs $a\subset s$ and \
$b\subset s'$  such that $a,b,l_i$ form a geodesic right angled triangle with $l_i$ opposite the right angle which is at a vertex of $\Sigma_i$.  Since $\cosh \rho (l_i) =\cosh \rho (a)\cosh \rho (b)$ (see Beardon \cite[page 146, Theorem 7.11.1]{Beardon}) it follows that $\rho (l_i)>\rho (a)$ and $\rho (l_i)>\rho (b)$ which implies $\rho (l_i)>\frac{1}{2} (\rho (a)+\rho (b))$. Therefore we can replace $l_i$ with the path $a\cup b$ on $\tilde{\Theta}_0$ that has the same endpoints and whose length is less than twice the length of $l_i$.

Assume that $s$ and $s'$ are not adjacent boundary sides of $\Sigma_i$. 
Since $l_i$ is at least as long as a side, then $\rho (l_i)\geq \frac{1}{M_1}$ and since $a$ follows at most five sides, then $\rho (l_i) \leq 5M_1$.
 Thus 
$\rho (a)\leq 5M_1=\frac{5M_1^2}{M_1}\leq 5M_1^2\rho (l_i)$ and the arc $l_i$ can be replaced by an arc $a$ on $\tilde{\Theta}_0$ whose length is less than ${5M_1^2}$ of the length of $l_i$.

By concatenating the above paths, we obtain a path on $\tilde{\Theta}_0$ which connects $x$ and $y$ and whose length is at most $\max \{ 2, {5M_1^2}\}$ times the length of $l$. Thus the inclusion of $\tilde{\Theta}_0$ is a quasi-isometry.
\end{proof}

We use $\Theta_0$ in order to define a pants train track $\Theta'$ such that edge paths of its lift $\tilde{\Theta}'$ are on a bounded distance from paths in $\tilde{\Theta}_0$. 

\begin{lemma}
\label{lem:quasi2}
Given $\Theta$ as above, there is a choice of a homeomorphic pants train track $\Theta'$ and a constant $d>0$ such that each edge path of the lifted train track $\tilde{\Theta}'$ is on a distance at most $d$ from a path in $\tilde{\Theta}_0$.
\end{lemma}

\begin{proof}
The geodesic pairs of pants of $X$ have cuffs with lengths between $1/M$ and $M$ for some $M\geq 1$. We fix a ``model'' geodesic pair of pants $P$ whose  all cuffs have length $1$ and on each cuff $\alpha_j$ we fix a base point $a_j$. Then we realize all standard train tracks $\{ s_i\}_{i=1}^{i_0}$ on $P$ such that their edges are smooth rectifiable arcs and the vertices on the cuffs are at the base points $\{a_j\}_{j=1}^3$. By Bishop \cite{Bishop}, there exists a biLipschitz map $f_n$ from $P$ to any pairs of pants $P_n$ of $X$ which is affine on cuffs and maps orthogonals between pairs of different cuffs in $P$ to orthogonals between pairs of different cuffs in $P_n$. Inside each $P_n$ we have a standard train track induced from $\Theta$ which is homotopic to $f_n(s_{i(n)})$ for a standard train track $s_{i(n)}$ in the model $P$. The orthogonal arcs $o_i^n$ between cuffs of $P_n$ from the definition of $\Theta_0$ are chosen such that they are homotopic to the finite edge paths of the standard train tracks $s_{i(n)}$ of $\Theta$ in $P_n$. Since the biLipschitz constants of $f_n:P\to P_n$ are bounded above, it follows that the finite edge paths in $f_n(s_{i(n)})$ are on a bounded distance from the corresponding $o_i^n$ with the bound uniform in $n$. 

Let $P_n$ and $P_{n_1}$ be two pairs of pants (possibly equal) that are glued along a cuff $\alpha_k$ in $X$. Then $f_{n}(a_j)$ and $f_{n_1}(a_{j_1})$ on $\alpha_j$ might be different. We choose a small one-sided collar neighborhood of $\alpha_j$ in $P_n$ and perform a smooth twist setwise fixing the cuff $\alpha_k$ that moves $f_n(a_j)$ to $f_{n_1}(a_{j_1})$ by a uniformly bounded map. We modify all standard train tracks $f_n(s_{i(n)})$ in this fashion so that the basepoints coming from the two sides of each $\alpha_k$ agree. With this process we obtain a new pants train track $\Theta'$ in $X$ which is homotopic to the original pants train track $\Theta$. 
The finite edge paths in the new standard train tracks in each $P_n$ are on a bounded distance from the corresponding orthogonals $o_i^n$ by the construction. Thus any infinite edge path in $\Theta'$ is on a bounded distance from a path in $\Theta_0$.
\end{proof}

Using the above  lemma we obtain

\begin{theorem}
\label{thm:quasi-no-cusps}
Let $X$ be an infinite Riemann surface without cusps equipped with a geodesic pants decomposition $\{ P_n\}_n$ whose cuff lengths are pinched between two positive constants and let $\Theta$ be a corresponding pants train track. Then a measured lamination weakly carried by the train track $\Theta$ is bounded if and only if the corresponding edge weight system has a finite supremum norm. 
\end{theorem}

\begin{proof}
Lemma \ref{lem:quasi2} implies  that an edge path in $\tilde{\Theta}$ is on a bounded distance from a simple path in $\tilde{\Theta}_0$. A simple path in $\tilde{\Theta}_0$ is a quasigeodesic in $\tilde{X}$ by Lemma \ref{lem:quasiisometry}. Therefore there exists $d>0$ such that each bi-infinite edge path in $\tilde{\Theta}$ is at  distance at most $d$ from the corresponding geodesic of $\tilde{X}$.

Recall that a measured lamination $\tilde{\mu}$ is bounded if for every continuous $\xi :G(\tilde{X})\to\mathbb{R}$ with a compact support we have that $\|\tilde{\mu}\|_{\xi}<\infty$. For a closed hyperbolic ball $D$ in $\tilde{X}$ denote by $G(D)$ the set of geodesics of $\tilde{X}$ that  intersect $D$. The support of $\xi$ is covered by finitely many $G(D_k)$ for $k=1,\ldots ,n$, where $D_k$ is a closed ball of radius one. Since $\|\xi\|_{\infty}<\infty$ we easily conclude that
a measured lamination $\tilde{\mu}$ is {\it bounded} if and only if $\|\tilde{\mu}\|_{Th} :=\sup_{D}\tilde{\mu} (G(D))<\infty$ where the supremum is over all closed  hyperbolic balls $D\subset\tilde{X}$ of radius $1$ (see \cite{GHL}, \cite{Sar} and \cite{BonahonSaric}). 

Let $\tilde{\mu}$ be a bounded measured lamination and denote by $f_{\tilde{\mu}}:E(\tilde{\Theta} )\to\mathbb{R}$ the corresponding edge weight system. If $e\in E(\tilde{\Theta} )$ then a hyperbolic ball $D_1$ of radius $d$ with the center on $e$ intersects all geodesics in $G(e)$ by the above. Therefore $f_{\tilde{\mu}}(e)\leq C(d)\|\tilde{\mu}\|_{Th}$ where $C(d)$ is the smallest number of closed hyperbolic balls of radius $1$ that is needed to cover a hyperbolic ball of radius $d$.  This bound is uniform in all $e\in E(\tilde{\Theta})$ and we obtain $\| {f}_{\tilde{\mu}}\|_{\infty}\leq C(d)\|\tilde{\mu}\|_{Th}$. 

Assume now that $\| f_{\tilde{\mu}}\|_{\infty} <\infty$. If $D$ is a closed hyperbolic ball of radius $1$ denote by $G(\tilde{\mu} ,D)$ the set of geodesics of the support of $\tilde{\mu}$ that intersect $D$. Since each geodesic of $G(\tilde{\mu} ,D)$ is on a distance at most $d$ from the corresponding bi-infinite edge path in $\tilde{\Theta}$, it follows that a  ball $D_1$ of radius $d+1$ concentric to $D$ intersects $\tilde{\Theta}$ in a finite set of edges $\{ e_1,e_2,\ldots ,e_k\}$ such that $G(\tilde{\mu} ,D)\subset \cup_{i=1}^kG(e_i)$. Then $\tilde{\mu} (D)\leq\sum_{i=1}^k {f}_{\tilde{\mu}}(e_i)$. The number of edges $k$ is uniformly bounded by some constant $k'$ independently of the choice of $D$ by Lemma \ref{lem:quasi2}. Thus we obtain $\tilde{\mu} (D)\leq k'\| \tilde{f}_{\tilde{\mu}}\|_{\infty}$ for all $D$ and thus $\|\tilde{\mu}\|_{Th}\leq k'\| f_{\tilde{\mu}}\|_{\infty}$. 
\end{proof}

\section{The hyperbolic surfaces with cusps and bounded pants decompositions}

We assume that a Riemann surface $X$ has a bounded geodesic pants decomposition $\{ P_n\}$ and possibly infinitely many cusps. We define a train track $\Theta$ on $X$ starting from the bounded pants decomposition $\{ P_n\}$. In each pair of pants we introduce geodesic arcs orthogonal to its cuffs. In the case when a pair of pants has three cuffs we divide it into two right angled hexagons that have sides pinched between two positive constants as in the previous section. When a pair of pants has two cuffs and a cusp, then we draw two geodesic arcs that divide it into a right angled hexagon and a right angled bigon with a cusp (see Figure 9). We need the following lemma

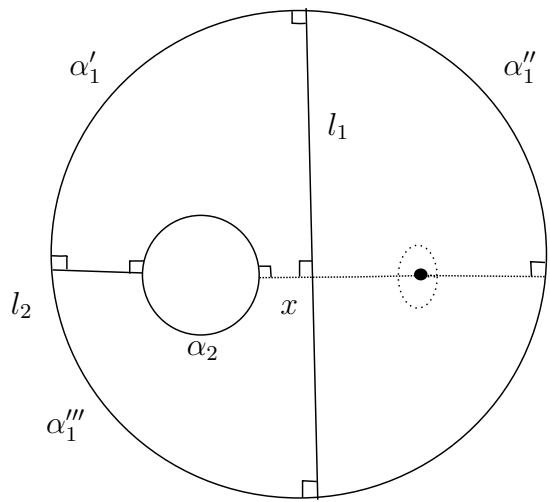
\begin{figure}[h]
\leavevmode \SetLabels
\L(.3*.8) $\alpha_1'$\\
\L(.67*.8) $\alpha_1''$\\
\L(.28*.2) $\alpha_1'''$\\
\L(.25*.4) $l_2$\\
\L(.4*.33) $\alpha_2$\\
\L(.48*.4) $x$\\
\L(.52*.7) $l_1$\\
\endSetLabels
\begin{center}
\AffixLabels{\centerline{\epsfig{file =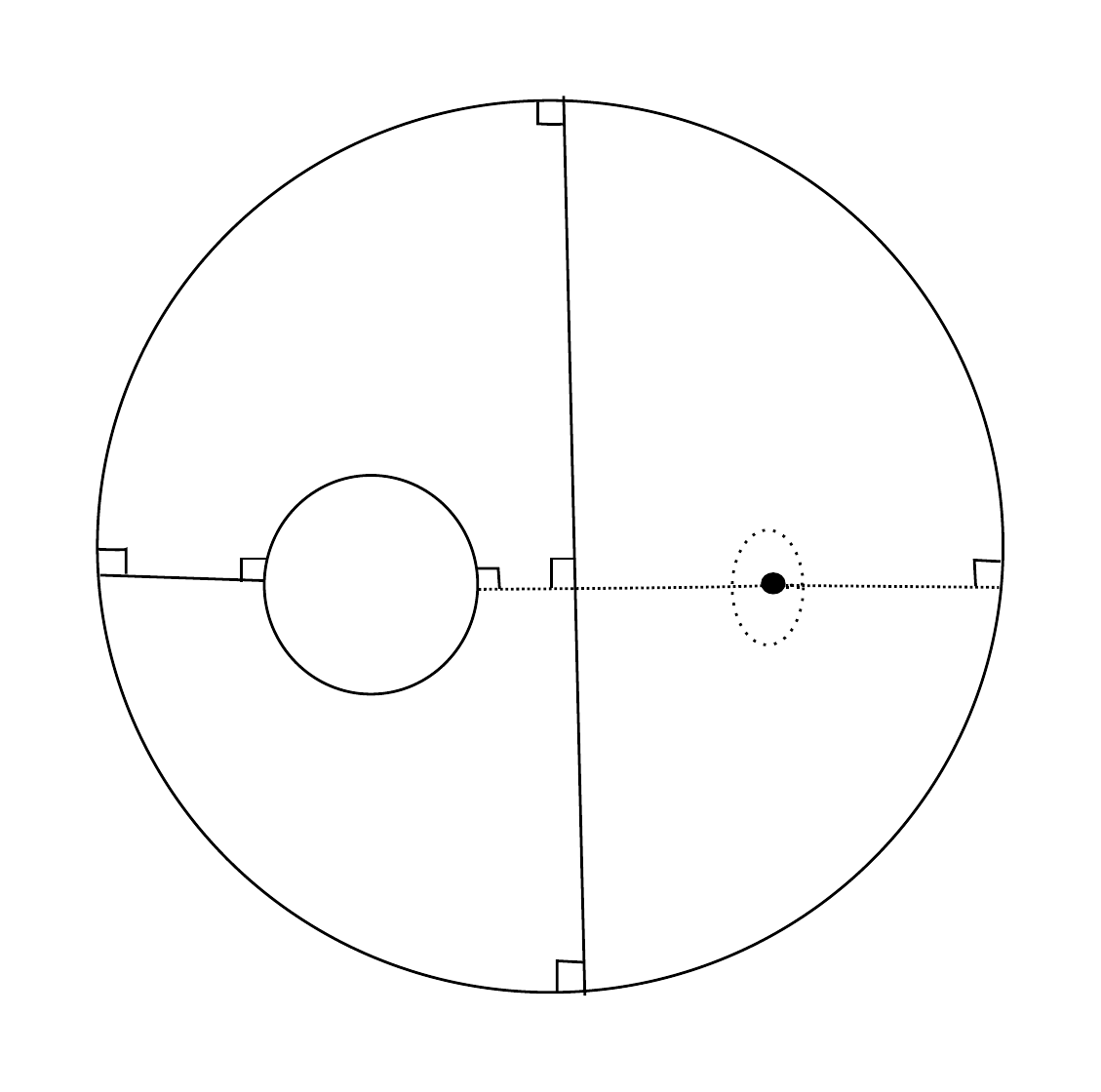,width=8.0cm,angle=0} }}
\vspace{-20pt}
\end{center}
\caption{The pair of pants with one cusp. The reflection in the dotted geodesics and $l_2$ is a symmetry of $P$. The dotted ellipse represents a horocycle of length $1$ around the puncture.} 
\end{figure}

\begin{lemma}
\label{lem:one_cusp}
Consider a pair of pants with two cuffs $\alpha_1,\alpha_2$ and one cusp such that $1/M\leq l(\alpha_i)\leq M$ for $i=1,2$. Let $l_2$ be the length of the shortest geodesic arc connecting $\alpha_1$ and $\alpha_2$, and let $l_1$ be the length of the shortest geodesic arc connecting $\alpha_1$ to itself. The cuff $\alpha_1$ is divided by the endpoints of $l_1$ and $l_2$ into arcs $\alpha_1'$, $\alpha_1''$ and $\alpha_1''$ as in Figure 9. Then there exists $M_1>1$ which depends only on $M$ such that
$$
1/M_1\leq l(\alpha_1'),l(\alpha_1''),l(\alpha_1''')\leq M_1.
$$
\end{lemma} 

\begin{proof}
The proof is a standard application of hyperbolic geometry similar to the proof of Lemma \ref{lem:seams_bound} (see Figure 8). 
\end{proof}

The last case is when a geodesic pair of pants has one cuff and two cusps. Then there is a single geodesic arc which connects the cuff to itself that is orthogonal to the cuff at both of its endpoints and that divides the pair of pants into two right angled bigons with cusps.

\begin{figure}[h]
\leavevmode \SetLabels
\L(.25*.6) $\alpha_1'$\\
\L(.72*.6) $\alpha_1''$\\
\L(.47*.7) $l_1$\\
\endSetLabels
\begin{center}
\AffixLabels{\centerline{\epsfig{file =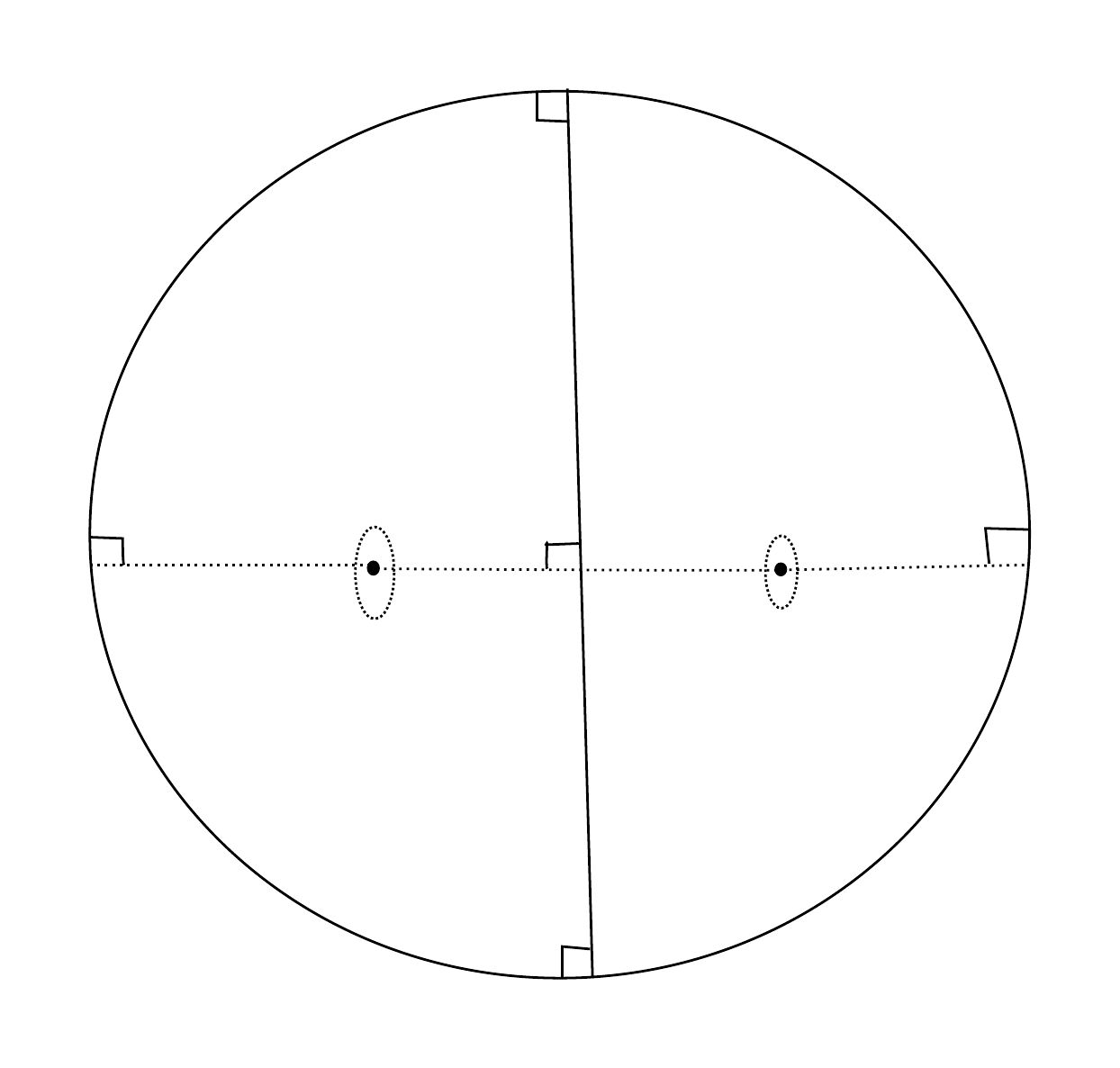,width=8.0cm,angle=0} }}
\vspace{-20pt}
\end{center}
\caption{The pair of pants with two cusps. The reflection in $l_1$ is a symmetry as well as the reflection in the dotted geodesics. The dotted ellipses represent horocycles around the two cusps.} 
\end{figure}

\begin{lemma}
\label{lem:two_cusps}
Consider a geodesic pair of pants with one cuff $\alpha_1$ and two cusps such that
$$
1/M\leq l(\alpha_1)\leq M
$$
for some $M>1$. Let $l_1$ be the length of the shortest geodesic arc connecting $\alpha_1$ to itself and separating the two cusps (see Figure 10). The arc $l_1$ divides $\alpha_1$ into two subarcs $\alpha_1'$ and $\alpha_1''$. Then 
$$
l(\alpha_1')=l(\alpha_1'')=\frac{1}{2}l(\alpha_1)
$$
and 
there exists $M_1>1$ depending only on $M$ such that 
$$
1/M_1\leq l_1\leq M_1.
$$
\end{lemma}

\begin{proof}
The proof is a standard application of hyperbolic geometry similar to the proof of Lemma \ref{lem:seams_bound} (see Figure 9). \end{proof}

 We are grateful to an anonymous referee for providing us with a concise proof of the following lemma.

\begin{lemma}
\label{lem:dist_bdd}
Let $Q$ be a Lambert quadrilateral with a zero angle and two finite sides $a$ and $b$ whose lengths are between $1/M$ and $M$, for $M>1$. Let $Q'$ be the subset of $Q$ obtained by removing a neighborhood of the vertex at infinity that has a horocyclic boundary of length $1/2$. Then the distance between  any point of $Q'$ from $a\cup b$ is at most $M_3(M)>0$.
\end{lemma}

\begin{proof}
We place $Q$ in the upper half-plane $\mathbb{H}$ such that the vertex of $Q$ with the zero angle is $\infty\in\hat{\mathbb{R}}=\partial_{\infty}\mathbb{H}$. The two finite sides $a$ and $b$ are arcs on the circles 
$|z|=r_1$ and $|z-1|=r_2$, and the infinite sides are on the vertical lines with real parts $0$ and $1$. In order to have a right angle at the vertex $z_0=x_0+iy_0$ of $Q$ that is the intersection of $|z|=r_1$ and $|z-1|=r_2$ it is necessary that $r_1^2+r_2^2=1$, where $0<r_1<1$ and $0<r_2<1$. 
An elementary computation gives $x_0=r_1^2$ and $y_0=\sqrt{r_1^2-r_1^4}$. 
Since the horocycle has length $1/2$, it is given by a horizontal Euclidean arc at the height $2$ connecting the two infinite sides. 

Note that the vertex $z_0$ is the point in $Q$ with the smallest height. Therefore
the maximal distance of a point in $Q$ to the union of finite sides is given by the distance between the vertex $z_0$ and the horocyclic side of $Q$. In other words, the maximal distance is $\log \frac{2}{y_0}$. Therefore we need to estimate $y_0$ from the below.

Consider the Euclidean ray $r\subset\mathbb{H}$ from the origin through the vertex $z_0$. The set of points on $r$ are equidistant from the positive $y$-axis (see Beardon \cite[Section 7.20]{Beardon}). Let $\varphi$ be the angle that $r$ subtends with the positive $y$-axis. If $\rho (a)$ is the length of the $a$ side, then (see \cite[Section 7.20]{Beardon})
$$
\sin \varphi=\tanh \rho (a).
$$ 
On the other hand, a right-angled triangle with vertices $0$, $z_0$ and $x_0$ gives $\sin\varphi =\cos (\frac{\pi}{2}-\varphi )=r_1$ which implies
$$
r_1=\tanh \rho (a).
$$
Since $\frac{1}{M}\leq \rho (a)\leq M$ it follows that $\tanh \frac{1}{M}\leq r_1\leq \tanh M$. 
We obtain
$$
\log \frac{2}{y_0}=\log 2-\log r_1-\frac{1}{2}\log (1-r_1^2)\leq \log 2-\log \tanh\frac{1}{M}-\frac{1}{2}
\log (1-\tanh^2M)=M_3(M).
$$
\end{proof}

\begin{lemma}
\label{lem:horocycle-geodesic}
Let $C$ be a horocycle of length $1$ on a Riemann surface $X$. Then the hyperbolic and the horocyclic distances on $C$ are bi-Lipschitz functions of each other.
\end{lemma}

\begin{proof}
We identify $\tilde{X}$ with the upper half-plane $\mathbb{H}$ such that the cyclic subgroup of $\pi_1(X)$ fixing the cusp corresponding to $C$ is generated by $z\mapsto z+1$. The lift of $C$
is the Euclidean horizontal line $h$ through $i\in\mathbb{H}$. 
The semiopen Euclidean horizontal arc on $h$ with endpoints $i$ and $1+i$ injectively covers $C$.

To prove the lemma, we need to compare the distances between $i$ and $i+s$ for $0<s\leq 1$.  
The horocyclic distance is $s$. The hyperbolic distance is 
$$
\rho (i,i+s)=\log\frac{\sqrt{4+s^2}+s}{\sqrt{4+s^2}-s}
$$
which gives
$$
\rho (i,i+s)=\log (1+\frac{2s}{\sqrt{4+s^2}-s})
$$
which is a bi-Lipschitz function of $s$ for $s\in (0,1]$.
\end{proof}

The train track $\Theta$ is defined using the standard train tracks in the fixed pants decomposition. The set $\Theta_0$ consists of all the cuffs and a maximal choice of shortest orthogonal arcs connecting cuffs in each pair of pants, where two cuffs in a pair of pants are  connected by orthogonal arcs of $\Theta_0$ if there is an edge path of $\Theta$ connecting the cuffs. We note that the cuffs and the orthogonal arcs to the cuffs do not meet horocyclic neighborhoods of cusps whose boundary has length $1$. Indeed, it is well known that any geodesic on a hyperbolic surface that enters the horoball neighborhood of a cusp with boundary length $1$ is not simple. Therefore the cuffs do not enter this neighborhood. To see that orthogonal arcs do not enter this neighborhood, it is enough to form a double of a single pair of pants and note that the orthogonal arcs together with their doubles form simple closed geodesics. Thus they do not enter the horocyclic neighborhoods with boundaries of length $1$ of the cusps.

Let $X'$ denote $X$ minus open horocyclic neighborhood around each cusp whose boundary has length $1$. We note that horocyclic neighborhoods are pairwise disjoint and their boundaries are closed horocyclic curves. Our goal is to prove that $\Theta$ as a subset of $X'$ has properties that are analogous to the properties of the train tracks in surfaces without cusps (see Section 7).

Let $\tilde{X}'$ be the lift of $X'$ to the universal covering $\tilde{X}$ of $X$. Equivalently $\tilde{X}'$ is obtained from $\tilde{X}$ by removing open horoballs at the lift of each cusp on $\partial_{\infty}\tilde{X}$. Let $\tilde{\Theta}$ and $\tilde{\Theta}_0$ be  lifts of $\Theta$ and $\Theta_0$ in $\tilde{X}'$. 

We first prove that $\tilde{\Theta}_0$ is quasi-isometric to $\tilde{X}'$ for the restriction of the hyperbolic metric.

\begin{lemma}
\label{lem:quasi-cusps} 
Given a hyperbolic surface $X$ with cusps, 
let $X'$ be obtained from $X$ by removing a horocyclic neighborhood with boundary length $1$ of each cusp. Let $\tilde{X}'$ be the lift of $X'$ to $\tilde{X}$ and $\tilde{\Theta}_0$ be the lift of $\Theta_0$ equipped with path metric. Then $\tilde{\Theta}_0$ is quasi-isometric to $\tilde{X}'$.
\end{lemma}

\begin{proof}
The set of boundary points of $\tilde{X}'$ that are on a finite distance from any interior point consists of horocycles that are based at the lifts of the punctures of $X$. By the above construction, the complementary regions of ${\Theta}_0$ in $X'$ are either right-angled hexagons or right-angled bigons minus a horocyclic neighborhood of a cusp with length $1$ boundary horocycle. Therefore the complementary regions to the lift $\tilde{\Theta}_0$ of $\Theta_0$ in $\tilde{X}'$ consists of right angled hexagons and  simply connected regions in $\tilde{X}'$ whose one boundary is a horocycle and the other boundary is a union of infinitely many geodesic arcs orthogonal to each other at their endpoints. The second complementary region has one ideal endpoint on $\partial_{\infty}\tilde{X}$ which is the point at which the horocycle is based, i.e.-the two boundary sides are asymptotic to a single point at infinity in both directions (see Figure 10).

By Lemmas \ref{lem:one_cusp} and \ref{lem:two_cusps} every edge of $\Theta_0$ has length between two positive constants which implies that every point of $X'$ is a bounded distance away from a point on $\Theta_0$ by Lemma \ref{lem:dist_bdd}. It follows that every point of $\tilde{X}'$ is on a bounded distance from $\tilde{\Theta}_0$. This distance is realized by a geodesic arc.

Let $p(x,y)$ be the path distance along $\tilde{\Theta}_0$ between two points $x,y\in\tilde{\Theta}_0$ and let $\rho (x,y)$ be the path distance between $x,y\in \tilde{X}'$. It remains to prove that $\frac{1}{A}\rho (x,y)-B\leq p(x,y)\leq A\rho (x,y)+B$ for fixed $A,B >0$ and all $x,y\in\tilde{\Theta}_0$.  It is clear that $\rho (x,y)\leq p(x,y)$ and we need to prove the opposite inequality. 

Let $l$ be the shortest path in $\tilde{X}'$ between $x,y\in\tilde{\Theta}_0$. Then $l$ is a finite union of geodesic arcs and pieces of the horocycles that are on the boundary of $\tilde{X}'$. 
We partition $l$ into union $\cup_{i=1}^pl_i$ of subpaths such that $l_i\cap l_{i+1}$ is a common endpoint and each $l_i$ is the intersection of $l$ with the closure of a single component of $\tilde{X}'-\tilde{\Theta}_0$. Let $x_i,x_{i+1}\in\tilde{\Theta}_0$ be the endpoints of $l_i$, in particular $x_1=x$ and $x_{p+1}=y$. If $l_i$ is inside a right-angled complementary hexagon then, as in the proof of Lemma \ref{lem:quasiisometry}, $l_i$ can be replaced by a biLipschitz path on $\tilde{\Theta}_0$ with the same endpoints.

Assume that $l_i$ is in the complementary region of $\tilde{X}'-\tilde{\Theta}_0$ that is a single component of the lift of the punctured bigon minus a horocyclic neighborhood of the cusp. The component is divided into pentagons with four geodesic sides and one horocyclic side by the lifts of the pentagons from either  Figure 8 or 9. 

\begin{figure}[h]
\leavevmode 
\SetLabels
\L(.64*.17)$x_i$\\
\L(.33*.19)$x_{i+1}$\\
\L(.5*.23)$l_i$\\
\endSetLabels
\begin{center}
\AffixLabels{\centerline{\epsfig{file =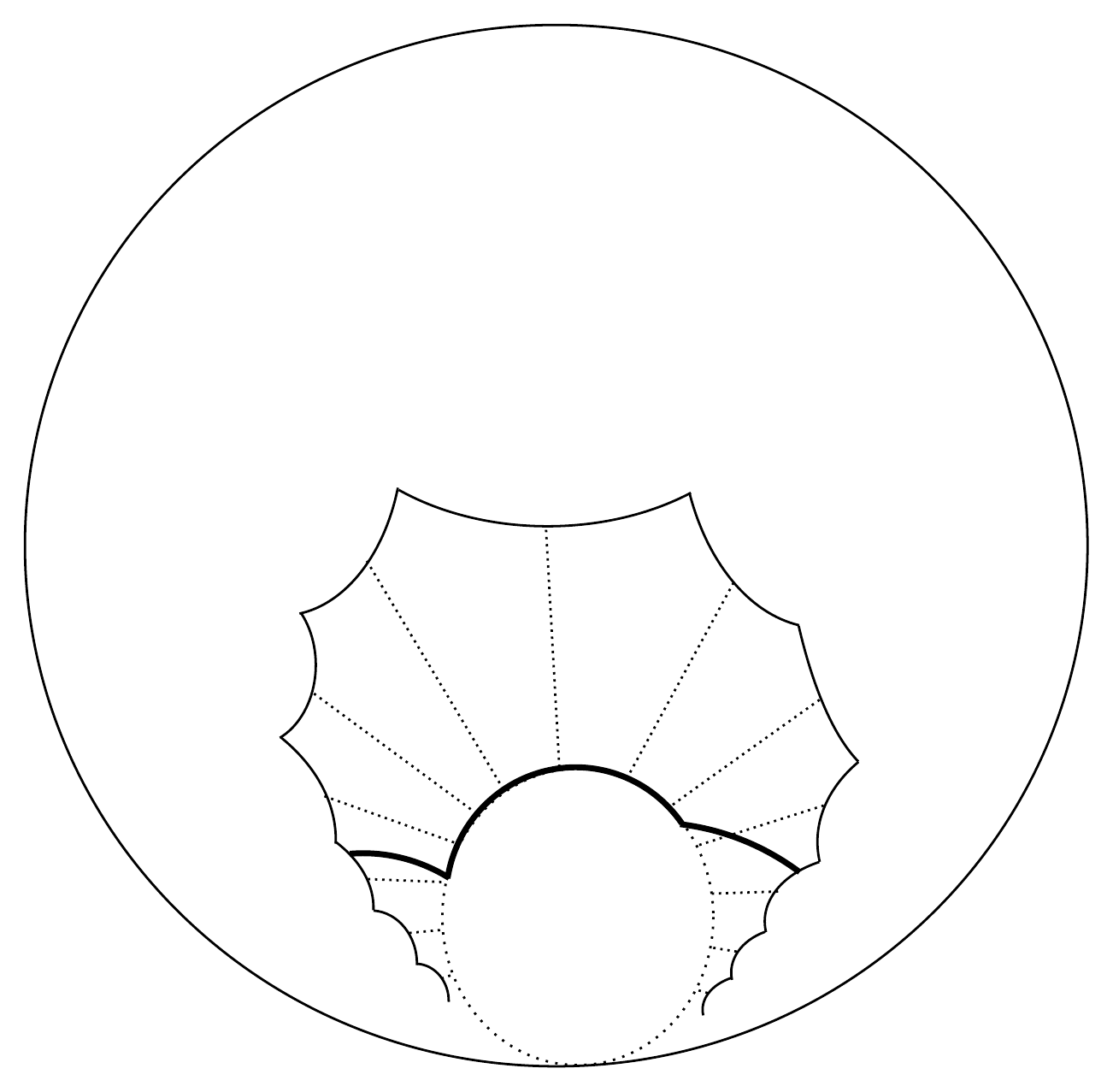,width=9.0cm,angle=0} }}
\vspace{-20pt}
\end{center}
\caption{The lift of a punctured bigon minus a horocyclic neighborhood of the cusp. The pentagon decomposition is given. The solid sides belong to $\tilde{\Theta}_0$. The arc $l_i$ is composed of geodesic arcs and a horocyclic arc.} 
\end{figure}

We assume that $l_i$ intersects $n\geq 3$ such pentagons, denoted by $P_1,P_2,\ldots ,P_n$ in that order. With the possible exception of $P_1$ and $P_n$, the path $l_i$ connects the two geodesic sides (dotted lines in Figure 10) that are orthogonal to the horocycle (dotted circle tangent to the boundary in Figure10). The part of $l_i$ that connects these two geodesic boundary sides is either a geodesic arc or a part of the horocycle or a combination of both. However, there is a lower bound $c_0$ on the length of each such part of $l_i$ by Lemmas  \ref{lem:one_cusp}, \ref{lem:two_cusps} and \ref{lem:horocycle-geodesic}. Thus the length of $l_i$ is at least $c_0(n-2)$. On the other hand, the points $x_i$ and $x_{i+1}$ are connected by at most $2n$ sides of the pentagons that are in $\tilde{\Theta}_0$ (solid sides in Figure 10). By Lemmas  \ref{lem:one_cusp} and \ref{lem:two_cusps}, there is $c>0$ an upper bound on the lengths of these sides and we have $2cn\geq p(x_i,x_{i+1})$. Therefore
$$
|l_i|\geq c_0(n-2)=\frac{c_0(n-2)}{2cn}2cn\geq c_1p(x_i,x_{i+1})
$$
where $|l_i|$ is the length of $l_i$ and $c_1=\frac{c_0}{6c}$.

Next we assume that $l_i$ intersects $n\leq 2$ pentagons. Then either $x_i$ and $x_{i+1}$ belong to adjacent geodesic segments in $\tilde{\Theta}_0$ which are orthogonal to each other (solid geodesic segments in Figure 10) or $l_i$ joins two geodesic segments of $\tilde{\Theta}_0$ that are separated by another geodesic segment of $\tilde{\Theta}_0$. In the former case, the length of $l_i$  is at least the length of the geodesic segment connecting $x_i$ to $x_{i+1}$, which is the hypothenuse of the right angled hyperbolic triangle with two other sides on $\tilde{\Theta}_0$. As in the proof of Lemma \ref{lem:quasiisometry}, we have $|l_i|\geq\frac{1}{2}p(x_i,x_{i+1})$.  In the later case, the length of $l_i$ is at least the length of the geodesic segment of $\tilde{\Theta}_0$ separating the two other geodesic segments that contain its endpoints $x_i$ and $x_{i+1}$. Then $|l_i|\geq c_0=\frac{c_0}{3c}3c\geq c_2p(x_i,x_{i+1})$ where $c_2=\frac{c_0}{3c}$.

Thus we can replace each $l_i$ by a path on $\tilde{\Theta}_0$ connecting endpoints of $l_i$ such that the length of $l_i$ is greater than a constant times the length of the path. The constant only depends on the bound $M$ for the pants decompostion. Thus $\tilde{\Theta}_0$ is quasi-isometric to $\tilde{X}'$.
\end{proof}

The proof of the following lemma is analogous to Lemma \ref{lem:quasi2}.

\begin{lemma}
\label{lem:quasi3}
Given $\Theta$ as above, there is a choice of a homeomorphic pants train track $\Theta'$ and a constant $d>0$ such that each edge path of the lifted train track $\tilde{\Theta}'$ is on a distance at most $d$ from a path in $\tilde{\Theta}_0$.
\end{lemma}

Our main result in this section is

\begin{theorem}
\label{thm:edge_weight_with_cusps}
Let $X$ be a conformally hyperbolic Riemann surface with a bounded geodesic pants decomposition $\{ P_n\}$ and possibly infinitely many cusps. Then the lift $\tilde{\mu}$ to $\tilde{X}$ of a measured lamination $\mu$ on $X$ that is weakly carried by $\tilde{\Theta}$ is bounded if and only if its edge-weight system $f_{\tilde{\mu}}:E(\tilde{\Theta} )\to\mathbb{R}$ has a finite supremum norm.
\end{theorem}

\begin{proof}
Assume that $\|f_{\tilde{\mu}}\|_{\infty}<\infty$. Every geodesic of the support of $\tilde{\mu}$ is inside of $\tilde{X}'$ by the choice of the horocyclic neighborhoods of the cusps of $X$. Then Lemmas \ref{lem:quasi-cusps} and \ref{lem:quasi3} implies that each geodesic of the support of $\tilde{\mu}$ is on a distance at most $d>0$ from the corresponding bi-infinite edge path in $\tilde{\Theta}$. Then the proof is finished as in Theorem \ref{thm:quasi-no-cusps}.

The other direction is identical to the proof of Theorem \ref{thm:quasi-no-cusps}.
\end{proof}

\section{The uniform weak* topology and edge-weight systems}

We keep the assumption that $X$ is an infinite hyperbolic surfaces with bounded pants decomposition $\{ P_n\}$ and at most countably many cusps. 
In the previous two sections we characterized bounded measured laminations on $X$ in terms of the edge weight systems on $\Theta$. In this section we describe the convergence in the uniform weak* topology on the measured laminations carried by $\Theta$ in terms of their corresponding edge weight systems. The proofs are based on the extensions of the ideas in the proofs of the weak* convergence. However the uniform weak* convergence imposes additional difficulties.

We first prove the convergence of the edge weight systems in the supremum norm.

\begin{proposition}
\label{prop:unif_edge}
Let $\tilde{\mu}_n,\tilde{\mu}$ be lifts to $\tilde{X}$ of measured laminations of $X$ that are weakly carried by $\Theta$. If 
$$
\tilde{\mu}_n\to\tilde{\mu}
$$
in the uniform weak* topology
then
$$\|f_{\tilde{\mu}_n}-f_{\tilde{\mu}}\|_{\infty}\to 0
$$
as $n\to\infty$.
\end{proposition}

\begin{proof}
Assume on the contrary that $\|f_{\tilde{\mu}_n}-f_{\tilde{\mu}}\|_{\infty}\nrightarrow 0.
$  Then there exists a subsequence $\tilde{\mu}_{n_k}$ of $\tilde{\mu}_n$ and a sequence of edges $e_{n_k}\in E(\tilde{\Theta})$ such that $| f_{\tilde{\mu}_{n_k}}(e_{n_k})-f_{\tilde{\mu}}(e_{n_k})|\geq c>0$ for some fixed $c>0$ and all $k\geq 1$. By taking a subsequence of $e_{n_k}$, if necessary, we can assume that all $e_{n_k}$ are either connector edges or cuff edges.

Let $\xi:G(\tilde{X})\to\mathbb{R}$ be a continuous function with a compact support and $\varphi_{n_k}\in\mathbb{H}(\tilde{X})$. By the uniform weak* convergence of $\tilde{\mu}_{n_k}$ to $\tilde{\mu}$ we have that
$$
\lim_{k\to\infty}|\int_{G(\tilde{X})}\xi\circ \varphi_{n_k}d[\tilde{\mu}_{n_k}-\tilde{\mu}]|=0.
$$
The uniform weak* convergence implies that the measured laminations $\tilde{\mu}_{n_k}$ are uniformly bounded on compact subsets of $G(\tilde{X})$. Lemma \ref{lem:weak*compact} implies that the push forward measures $\tilde{\nu}_{n_k}=(\varphi_{n_k})_{*}\tilde{\mu}_{n_k}$ and $\tilde{\nu}_{n_k}'=(\varphi_{n_k})_{*}\tilde{\mu}$ have subsequences that converge in the weak* topology to  measured laminations $\tilde{\nu}$ and $\tilde{\nu}'$ weakly carried by $\tilde{\Theta}$ which are lifts of  measured laminations on $X$. 
The above limit implies 
$$
\lim_{ k\to\infty}|\int_{G(\tilde{X})}\xi d[\tilde{\nu}_{n_k}-\tilde{\nu}_{n_k}']|=0.
$$
and thus $\tilde{\nu}=\tilde{\nu}'$.

We separate the rest of the proof into two cases based on the type of the edges $e_{n_k}$. 

\vskip .2 cm

\noindent {\it Case 1.}
Assume first that all $e_{n_k}$ are connector edges. Since there is only finitely many types of standard train tracks in a pair of pants (see Figures 2 and 3), there is only finitely many possible relative positions of $e_{n_k}$ with respect to the other connector edges and lifts of cuffs in a component ${P}_{n_k}$ of the lift of a pair of pants that contains $e_{n_k}$.  After taking a subsequence, we can assume that all $e_{n_k}$ are lifts of the connector edges in the pairs of pants with the same standard train track, in the same relative position in the fixed standard train track and with the same smoothing at the vertices on the cuffs of the pairs of pants. The connector edge $e_{n_k}$ is contained in an at most four finite connector edge path $\{\gamma_1^k,\ldots ,\gamma_j^k\}$ that connect two lifts of cuffs on the boundary of $P_{n_k}$ by considering lifts of the standard train tracks in Figures 2, 3 and 4. It follows that
$$
G(e_{n_k})=\cup_{i=1}^j G(\gamma_i^k)
$$
and $G(\gamma_i^k)\cap G(\gamma_{i'}^k)=\emptyset$ for $i\neq i'$. 

By Proposition \ref{prop:carrying_box_connector}  there exist boxes of geodesics $(Q_{i}^{k})'$ contained in the interior of the boxes of geodesics $Q_i^k$ such that
$$
G(\gamma_i^k)=G(\tilde{\Theta})\cap (Q_i^k)'
$$
and
$$
G(\tilde{\Theta})\cap [Q_i^k\setminus (Q_i^k)']=\emptyset .
$$

Fix $v\in\tilde{X}$ and choose $\varphi_{n_k}$ such that one vertex of $\varphi_{n_k}(e_{n_k})$ is mapped onto $v$. Since $X$ is of bounded geometry, it follows that a subsequence of $\varphi_{n_k}(Q_i^k)$ and $\varphi_{n_k}((Q_i^k)')$ converges to boxes of geodesics $Q_i$ and $Q_i'$ for $i=1,\ldots ,j$. By the lower bound on the Liouville measure of the four boxes of geodesics in  $Q_i^k\setminus (Q_i^k)'$ we obtain that $Q_i'$ is contained in the interior of $Q_i$ for all $i=1,\ldots ,j$. We choose a box of geodesics $Q_i''$ that contains $Q_i'$ in its interior $(Q_i'')^{\circ}$ and is contained in the interior $(Q_i')^{\circ}$ of $Q'_i$ for all $i$. 

By the convergence of $\varphi_{n_k}(Q_i^k)$ and $\varphi_{n_k}((Q_i^k)')$ to $Q_i$ and $Q_i'$, it follows that there exists $k_0>0$ such that for all $k\geq k_0$, 
$$\varphi_{n_k}((Q_i^k)')\subset (Q_i'')^{\circ}$$ and 
$$Q_i''\subset [\varphi_{n_k}(Q_i^k)]^{\circ}.$$
This implies that 
$$\delta Q_i''\subset [\varphi_{n_k}(Q_i^k)]^{\circ}\setminus \varphi_{n_k}((Q_i^k)').
$$
Since $\tilde{\nu}_{n_k}(\varphi_{n_k}(Q_i^k)\setminus \varphi_{n_k}((Q_i^k)'))=0$ and
$\tilde{\nu}_{n_k}'(\varphi_{n_k}(Q_i^k)\setminus \varphi_{n_k}((Q_i^k)'))=0$ we conclude that
$$
\tilde{\nu}(\delta Q_i'')=0.
$$
This gives that $\tilde{\nu}_{n_k}(Q_i'')\to\tilde{\nu} (Q_i'')$ and $\tilde{\nu}_{n_k}'(Q_i'')\to\tilde{\nu} (Q_i'')$ as $k\to\infty$ by Lemma \ref{lem:box_conv}.

On the other hand, we have $\tilde{\nu}_{n_k}(Q_i'')=\tilde{\mu}_{n_k}(Q_i^k)=\tilde{\mu}_{n_k}(G(\gamma_i^k))$ and $\tilde{\nu}_{n_k}'(Q_i'')=\tilde{\mu}(Q_i^k)=\tilde{\mu}(G(\gamma_i^k))$. Together with the above, we have 
$$
|f_{\tilde{\mu}_{n_k}}(e_{n_k})-f_{\tilde{\mu}}(e_{n_k})|=|\sum_{i=1}^j[\tilde{\mu}_{n_k}(G(\gamma_i^k))-\tilde{\mu}(G(\gamma_i^k))]|\to 0\
$$
as $k\to\infty$ which is a contradiction.

\vskip .2 cm 
 
 \noindent {\it Case 2.}
 It remains to consider the case when $e_{n_k}$ are cuff edges.  Let $\tilde{\alpha}$ be a fixed oriented geodesic of $\tilde{X}$ with the initial point $x$ and the end point $y$. Let $v\in\tilde{\alpha}$ be a fixed point. 
 Let $\varphi_{n_k}$ be an isometry of $\tilde{X}$ that maps $\tilde{\alpha}_{n_k}$ (which contains $e_{n_k}$) to $\tilde{\alpha}$ such that $v$ is one endpoint of $\varphi_{n_k}(e_{n_k})$  and the other endpoint $v_{n_k}\in\varphi_{n_k}(e_{n_k})\subset \tilde{\alpha}$ is between $v$ and $y$.  By Proposition \ref{prop:carrying_box_cuff} there exist two boxes of geodesics $Q_{n_k}'$ and $Q_{n_k}$ with $Q_{n_k}'$ contained in the interior of $Q_{n_k}$ such that
 $$
 G(e_{n_k})=G(\tilde{\Theta})\cap Q_{n_k}=G(\tilde{\Theta})\cap Q_{n_k}'
 $$
 and 
 $$
 G(\tilde{\Theta})\cap (Q_{n_k}\setminus Q_{n_k}')=\emptyset .
 $$
 
 After taking a subsequence of $\varphi_{n_k}$, for simplicity denoted by $\varphi_{n_k}$, we can assume that $\varphi_{n_k}(Q_{n_k}')\to Q'$ and  $\varphi_{n_k}(Q_{n_k})\to Q$ with the lower bound $1/m>0$ on the Liouville measure of the four component boxes of $Q\setminus Q'$ (see Proposition \ref{prop:carrying_box_cuff}). Let $Q''$ be a box of geodesics that contains $Q'$ in its interior and is contained in the interior of $Q$. Then there exists $k_0$ such that for all $k\geq k_0$ we have 
 $$
 \varphi_{n_k}(Q_{n_k}')\subset (Q'')^{\circ}
 $$
 and 
 $$
  Q''\subset [\varphi_{n_k}(Q_{n_k})]^{\circ} ,
 $$
 where $Q^{\circ}$ denotes the interior of a box of geodesics $Q$. 
 
 Since
 $$
 \tilde{\nu}_{n_k} (\varphi_{n_k}(Q_{n_k})\setminus  \varphi_{n_k}(Q_{n_k}'))= \tilde{\nu}_{n_k}' (\varphi_{n_k}(Q_{n_k})\setminus  \varphi_{n_k}(Q_{n_k}'))=0
 $$
 and 
 $$
 \delta Q''\subset [\varphi_{n_k}(Q_{n_k})]^{\circ}\setminus  \varphi_{n_k}(Q_{n_k}')
 $$
 it follows that 
 $$
 \tilde{\nu}(\delta Q'')=0.
 $$
 
 By the weak* convergence of $\tilde{\nu}_{n_k}$ and $\tilde{\nu}_{n_k}'$ to $\tilde{\nu}$ we have
 $$
 \tilde{\nu}_{n_k}(Q''),\tilde{\nu}_{n_k}'(Q'')\to\tilde{\nu}(Q'')
 $$
 as $k\to\infty$. Since $\tilde{\nu}_{n_k}(Q'')=\tilde{\mu}_{n_k}(\varphi_{n_k}(Q''))=f_{\tilde{\mu}_{n_k}}(e_{n_k})$ and $\tilde{\nu}_{n_k}'(Q'')=\tilde{\mu}(\varphi_{n_k}(Q''))=f_{\tilde{\mu}}(e_{n_k})$ we have $|f_{\tilde{\mu}_{n_k}}(e_{n_k})-f_{\tilde{\mu}}(e_{n_k})|\to 0$ as $k\to\infty$ which is a contradiction.
\end{proof}

We prove that the convergence of the edge weight systems in the supremum norm implies the convergence of measured laminations in the uniform weak* topology.

\begin{proposition}
\label{prop:edge_unif}
Let $\tilde{\mu}_n,\tilde{\mu}$ be lifts to $\tilde{X}$ of measured laminations of $X$ that are weakly carried by $\Theta$. If 
$$\|f_{\tilde{\mu}_n}-f_{\tilde{\mu}}\|_{\infty}\to 0
$$
then
$$
\tilde{\mu}_n\to\tilde{\mu}
$$
in the uniform weak* topology
as $n\to\infty$.
\end{proposition}

\begin{proof} Assume on the contrary that $\tilde{\mu}_n$ does not converge to $\tilde{\mu}$ in the uniform weak* topology. Namely there exist a continuous function $\xi :G(\tilde{X})\to\mathbb{R}$ with a compact support, a sequence $\varphi_{j}$ of isometries of $\tilde{X}$ and $c>0$ such that 
\begin{equation}
\label{eq:not_uniform}
|\int_{G(\tilde{X})}\xi\circ\varphi_{j}d[\tilde{\mu}_{j}-\tilde{\mu}]|\geq c>0.
\end{equation}
Define $\tilde{\nu}_{j}=(\varphi_{j})_{*}(\tilde{\mu}_{j})$ and $\tilde{\nu}_{j}'=(\varphi_{j})_{*}(\tilde{\mu})$. Since both sequences $\tilde{\nu}_{j}$ and $\tilde{\nu}_{j}'$ are bounded on compact subsets of $G(\tilde{X})$, there is a choice of their subsequences that converge to measured laminations $\tilde{\nu}$ and $\tilde{\nu}'$ in the weak* topology (see Lemma \ref{lem:weak*compact}). We keep the same indexing of subsequences for  simplicity.  Then from (\ref{eq:not_uniform}) we have 
\begin{equation}
\label{eq:different_limits}
|\int_{G(\tilde{X})}\xi d[\tilde{\nu}-\tilde{\nu}']|\geq c>0
\end{equation}
which implies that $\tilde{\nu}\neq\tilde{\nu}'$. 

Let $Q=I\times J$ be a box of geodesics in $G(\tilde{\Theta})$. 
Let $\{\tilde{\alpha}_k\}_k$ be the countable collection of lifts of cuff of the pants decomposition. Then there exists a subsequence $\varphi_{j_n}$ such that $\{\varphi_{j_n}(\tilde{\alpha}_k)\}_n$ converges in the Hausdorff topology on closed subsets of $G(\tilde{X})$ to either an empty set or a configuration $\{\tilde{\alpha}_k^{\infty}\}_k$ such that the components of the complement are infinite geodesic polygons. 

First 
consider the case when $\{\varphi_{j_n}(\tilde{\alpha}_k)\}_n$ converges to $\{\tilde{\alpha}_k^{\infty}\}_k$.  The shortest distance between two boundary geodesics of a complement is between $1/M$ and $M$, where $M$ is the corresponding bound for the original lifts of cuffs $\{\tilde{\alpha}_k\}_k$. By taking further subsequences, we can arrange that each complementary region of $\{\tilde{\alpha}_k^{\infty}\}_k$ is the limit of complementary regions of $\{\varphi_{j_n}(\tilde{\alpha}_k)\}_n$ that are lifts of pairs of pants of the same type ($3$, $2$ or $1$ cuff) and the standard train track in each pair of pants (accumulating to one complementary region) is of the same type and has the same tangency at the cuff. 

By slightly increasing the size of the box $Q=I\times J$, we can assume that the vertices of $Q$ are the endpoints of some geodesics in $\{\tilde{\alpha}_k^{\infty}\}_k$  and that no geodesic of $\{\tilde{\alpha}_k^{\infty}\}_k$  is on $\delta Q$. The bounded geometry implies that there is only finitely many geodesics of $\{\tilde{\alpha}_k^{\infty}\}_k$  in $Q$ (see Lemma \ref{lem:cuffs_in_box}). Since the complementary regions of $\{\tilde{\alpha}_k^{\infty}\}_k$ are limits of the complementary regions of $\{\varphi_{j_n}(\tilde{\alpha}_k)\}_n$ of the same kind and with the lifts of the same type standard train tracks and the same tangency, we have a notion of the limiting vertices and connector edges with the same combinatorics. The finite edge paths connecting lifts of cuffs under $\varphi_{j_n}$ converge to finite edge paths connecting boundary geodesics in the complements of $\{\tilde{\alpha}_k^{\infty}\}_k$.

There exists at most countably many geodesics $(\tilde{\alpha}_i^{\infty})'$ from $\{\tilde{\alpha}_k^{\infty}\}_k$ with both endpoints in $I$ that are not separated from $J$ by another such geodesic. Similarly there is an at most countably many geodesics $(\tilde{\alpha}_j^{\infty})''$ from $\{\tilde{\alpha}_k^{\infty}\}_k$ with both endpoints in $J$ that are not separated from $I$ by another such geodesic. Then there is an at most countable collection of pairs $\{ ((\tilde{\alpha}_i^{\infty})',(\tilde{\alpha}_j^{\infty})'')\}_{i,j}$ that are connected by limits $\gamma_{i,j}^{\infty}$ of the image under $\varphi_{j_n}$ of finite edge paths in $\tilde{\Theta}$. 

Any geodesic $\tilde{\alpha}_r^{\infty}$ of $\{\tilde{\alpha}_k^{\infty}\}_k$ that is in the interior of $Q$ is the limit of $\varphi_{j_n}(\tilde{\alpha}_{k_r})$. By Proposition \ref{prop:carrying_box_cuff} there exist two boxes of geodesics $Q_{k_r}'$ and $Q_{k_r}$ such that $Q_{k_r}'\subset Q_{k_r}^{\circ}$, $G(\tilde{\Theta})\cap (Q_{k_r}\setminus (Q_{k_r})')=\emptyset$ and $\tilde{\alpha}_k^r\in Q_{k_r}'$. Then the limits $(Q_r^{\infty})'$ and $Q_r^{\infty}$ satisfy $(Q_r^{\infty})'\subset (Q_r^{\infty})^{\circ}$, $\tilde{\alpha}_r^{\infty}\in (Q_r^{\infty})'$ and $G(\tilde{\Theta})\cap (Q_r^{\infty}\setminus (Q_r^{\infty})')=\emptyset$. 
Then, in an analogous fashion as in the proof of Proposition \ref{prop:unif_edge}, there is a box of geodesics $(Q_r^{\infty})''$ such that it contains $(Q_r^{\infty})'$ in its interior and is contained in the interior of $Q_r^{\infty}$, and for all $n\geq n_0>0$ we have $\varphi_{j_n}(Q_{k_r}')\subset [(Q_r^{\infty})'']^{\circ}$ and $\varphi_{j_n}(G(\tilde{\Theta}))\cap\delta (Q_R^{\infty})''=\emptyset$. It then follows that when $n\to\infty$
$$
\tilde{\nu}_{j_n}((Q_r^{\infty})'')\to\tilde{\nu} ((Q_r^{\infty})'')
$$
and
$$
\tilde{\nu}_{j_n}'((Q_r^{\infty})'')\to\tilde{\nu}' ((Q_r^{\infty})'').
$$
By $\|f_{\tilde{\mu}_n}-f_{\tilde{\mu}}\|_{\infty}\to 0$ as $n\to\infty$ and the fact that $\tilde{\mu}_n$ and $\tilde{\mu}$ are piecewise linear functions of the edge weights of the edge paths representing $Q_{k_r}$  and its immediate neighbors, it follows that $|\tilde{\nu}_{j_n}((Q_r^{\infty})'')-\tilde{\nu}_{j_n}'((Q_r^{\infty})'')|\to 0$ as $n\to\infty$. Thus
$$
\tilde{\nu}((Q_r^{\infty})'')=\tilde{\nu}'((Q_r^{\infty})'').
$$

In an analogous fashion we obtain that $\tilde{\nu}=\tilde{\nu}'$ on all boxes of geodesics $(Q_{i,j}^{\infty})''$ that correspond to the limits of $\varphi_{j_n} (Q_{i,j})$, where $Q_{i,j}$ and $Q_{i,j}'$ are defined in Proposition \ref{prop:carrying_box_connector}. 

The support of $\tilde{\nu}$ and $\tilde{\nu}'$ on $Q$ is contained in $\cup_{n\geq n_0} \varphi_{j_n}(G(\tilde{\Theta}))$ and thus it is in the disjoint union of $(Q_{i,j}^{\infty})''$ and $(Q_r^{\infty})''$. Thus $\tilde{\nu}(Q)=\tilde{\nu}'(Q)$. This equality holds for all $Q$ that have vertices at the endpoints of the geodesics $\{\tilde{\alpha}_k^{\infty}\}_k$. The endpoints of 
$\{\tilde{\alpha}_k^{\infty}\}_k$ are dense in $\partial_{\infty}\tilde{X}$ and we obtain $\tilde{\nu}=\tilde{\nu}'$ which is a contradiction.

In the case when $\{\varphi_{j_n}(\tilde{\alpha}_k)\}_n$ converges to an empty set we immediately get $\tilde{\nu}=\tilde{\nu}'=0$ which is again a contradiction.
\end{proof}

From the above two propositions we obtain

\begin{theorem}
\label{thm:uniform_cont}
A sequence $\tilde{\mu}_n$ of bounded measured laminations weakly carried by $\tilde{\Theta}$ converges to a bounded measured lamination $\tilde{\mu}$ weakly carried by $\tilde{\Theta}$ in the uniform weak* topology if and only if the corresponding edge weight systems $f_{\tilde{\mu}_n}:E(\tilde{\Theta} )\to\mathbb{R}_{\geq 0}$ of $\tilde{\mu}_n$ converge to the edge weight system $f_{\tilde{\mu}}:E(\tilde{\Theta} )\to\mathbb{R}_{\geq 0}$ of $\tilde{\mu}$ in the supremum norm.
\end{theorem}

\end{document}